\def\draft{n}
\documentclass[11pt]{amsart}
\usepackage[headings]{fullpage}
\usepackage{amssymb,epic,eepic,epsfig}
\usepackage{pb-diagram}
\usepackage[all]{xy}
\usepackage{graphicx}
\usepackage{latexsym}
\usepackage{amsmath,amsfonts,amsthm,mathdots}
\usepackage{stmaryrd}
\usepackage{enumerate}
\usepackage[T1]{fontenc}
\usepackage[latin1]{inputenc}
\usepackage[english]{babel}
\usepackage{mathtools}  
\usepackage{url}
\usepackage[bookmarks=true,%
    colorlinks=true,%
    linkcolor=blue,%
    citecolor=blue,%
    filecolor=blue,%
    menucolor=blue,%
    urlcolor=blue,%
    breaklinks=true]{hyperref}

\newtheorem{theorem}{Theorem}[section]
\theoremstyle{definition}
\newtheorem{proposition}[theorem]{Proposition}
\newtheorem{lemma}[theorem]{Lemma}
\newtheorem{definition}[theorem]{Definition}
\newtheorem{remark}[theorem]{Remark}
\newtheorem{corollary}[theorem]{Corollary}
\newtheorem{conjecture}[theorem]{Conjecture}

\newtheorem{example}[theorem]{Example}

\newtheorem{notation}[theorem]{Notation}

\numberwithin{equation}{section}

\def\printname#1{
        \if\draft y
                \smash{\makebox[0pt]{\hspace{-0.5in}
                        \raisebox{8pt}{\tt\tiny #1}}}
        \fi
}

\newcommand{\psdraw}[2]
         {\begin{array}{c} \hspace{-1.3mm}
        \raisebox{-4pt}{\epsfig{figure=draws/#1,width=#2}}
        \hspace{-1.9mm}\end{array}}

\def\lbl#1{\label{#1}\printname{#1}}

\usepackage[all]{xy}
\CompileMatrices
\SelectTips{cm}{}

\def\cxymatrix#1{\xy*[c]\xybox{\xymatrix#1}\endxy}

\usepackage{graphicx}
\usepackage{caption}
\usepackage{subfig}
\usepackage{ifpdf}
\ifpdf
\fi

\DeclareMathOperator{\Li}{Li}
\DeclareMathOperator{\FT}{FT}

\DeclareMathOperator{\Vol}{Vol}
\DeclareMathOperator{\CS}{CS}
\DeclareMathOperator{\cs}{cs}

\DeclareMathOperator{\Imag}{Im}

\DeclareMathOperator{\SL}{SL}
\DeclareMathOperator{\SU}{SU}
\DeclareMathOperator{\GL}{GL}

\DeclareMathOperator{\PSL}{PSL}

\DeclareMathOperator{\Ker}{Ker}

\DeclareMathOperator{\ind}{ind}

\DeclareMathOperator{\Tr}{Tr}
\DeclareMathOperator{\Stab}{Stab}
\DeclareMathOperator{\gen}{gen}

\DeclareMathOperator{\diag}{diag}
\DeclareMathOperator{\Adj}{Adj}
\DeclareMathOperator{\geo}{geo}
\DeclareMathOperator{\odd}{odd}
\DeclareMathOperator{\Conj}{Conj}
\DeclareMathOperator{\cusp}{cusp}

\DeclarePairedDelimiter{\Norm}{\lvert}{\rvert}

\newcommand{\R}{\mathbb R} 
\newcommand{\C}{\mathbb C}
\renewcommand{\H}{\mathbb H}
\newcommand{\F}{\mathbb F}
\newcommand{\Mat}[4]{\left(\begin{smallmatrix}#1&#2\\#3&#4\end{smallmatrix}
\right)}
\newcommand{\N}{\mathbb N} 
\newcommand{\Z}{\mathbb Z}
\newcommand{\Q}{\mathbb Q}
\newcommand{\B}{\mathcal B}
\newcommand{\Pre}{\mathcal P}

\newcommand{\Cinfty}{\C\cup\{\infty\}}
\newcommand{\Cremove}{\C\backslash\{0,1\}}

\def\BC{\mathbb C}

\def\a{\alpha}

\def\e{\epsilon}

\def\b{\beta}

\def\CS{\mathrm{CS}}

\def\Li{\mathrm{Li}}

\def\diag{\mathrm{diag}}

\def\Tr{\mathrm{Tr}}
\def\SU{\mathrm{SU}}

\newcommand{\comma}{,\penalty300}
\newcommand{\semico}{;\penalty 300 }

\begin{document}

%%%%%%%%%%%%%%%%%%%%%%{page1}

\title[The complex volume of $\SL(n,\C)$-representations of 3-manifolds]
{The complex volume of $\SL(n,\C)$-representations of 3-manifolds}
%Perhaps $\SL(n,\C)$-representations of $3$-manifold groups 
%(and hyperbolic volume)

\author{Stavros Garoufalidis}
\address{School of Mathematics \\
         Georgia Institute of Technology \\
         Atlanta, GA 30332-0160, USA \newline
         {\tt \url{http://www.math.gatech.edu/~stavros}}}
\email{stavros@math.gatech.edu}
\author{Dylan P. Thurston}
\address{Department of mathematics \\
        Columbia University \\
         MC 4436 \\
        New York, NY 10027, USA \newline 
        {\tt \url{http://www.math.columbia.edu/~dpt}}}
\email{dthurston@barnard.edu}
\author{Christian K. Zickert}
\address{University of Maryland \\
         Department of Mathematics \\
         College Park, MD 20742-4015, USA \newline
         {\tt \url{http://www2.math.umd.edu/~zickert}}}
\email{zickert@umd.edu}

\thanks{The authors were supported in part by the NSF. \\
\newline
2001 {\em Mathematics Classification.} Primary 57N10, 57M27, 58J28. 
Secondary 11R70, 19F27, 11G55.
}

\date{September 21, 2011} % \today

%\dedicatory{\large{\bf Preliminary notes.
%Please do not distribute under any circumstances!}}

\begin{abstract}
For a compact 3-manifold $M$ with arbitrary (possibly empty) boundary, we 
give a parametrization of the set of conjugacy classes of boundary-unipotent 
representations of $\pi_1(M)$ into $\SL(n,\C)$. Our parametrization uses
Ptolemy coordinates, which are inspired by coordinates on higher 
Teichm\"{u}ller spaces due to Fock and Goncharov. We show that a 
boundary-unipotent representation determines an element in Neumann's extended 
Bloch group $\widehat\B(\C)$, and use this to obtain an efficient formula for 
the Cheeger-Chern-Simons invariant, and in particular for the volume. 
Computations for the census manifolds show that boundary-unipotent 
representations are abundant, and
numerical comparisons with census volumes, suggest that the volume of a 
representation is an integral linear 
combination of volumes of hyperbolic $3$-manifolds. This is in agreement with 
a conjecture of Walter Neumann, stating that the Bloch group is generated by 
hyperbolic manifolds.
\end{abstract}

\maketitle

\tableofcontents

\section{Introduction}
\lbl{sec.intro}

For a closed 3-manifold $M$, the Cheeger-Chern-Simons 
invariant~\cite{CheegerSimons,ChernSimons} of a representation 
$\rho$ of $\pi_1(M)$ in $\SL(n,\C)$ is given by the Chern-Simons integral
\begin{equation}
\lbl{ChernSimonsintegral}
\widehat c(\rho)=\frac{1}{2}\int_Ms^*
\big(\Tr(A\wedge dA+\frac{2}{3}A\wedge A\wedge A)\big)\in \C/4\pi^2\Z,
\end{equation}
where $A$ is the flat connection in the flat $\SL(n,\C)$-bundle $E_\rho$ 
with holonomy $\rho$, and $s\colon M\to E_\rho$ is a section of $E_\rho$. 
Since $\SL(n,\C)$ is $2$-connected a section always exists, and a different 
choice of section changes the value of the integral by a multiple of $4\pi^2$.

When $n=2$, the imaginary part of the Cheeger-Chern-Simons invariant equals 
the hyperbolic volume of~$\rho$. More precisely, if 
$D\colon\widetilde M\to\H^3$ is a developing map for $\rho$ and $\nu_{\H^3}$ 
is the hyperbolic volume form, $\Imag(\widehat c(\rho))$ equals the integral 
of $D^*(\nu_{\rho})$ over a fundamental domain for $M$. In particular, if 
$M=\H^3/\Gamma$ is a hyperbolic manifold, and $\rho$ is a lift to 
$\SL(2,\C)$ of 
the geometric representation $\rho_{\geo}\colon\pi_1(M)\to\PSL(2,\C)$, the 
imaginary part equals the volume of $M$.
In fact, in this case we have
\begin{equation}
\lbl{CCSeqVolC}
\widehat c(\rho)=i(\Vol(M)+i\CS(M)),
\end{equation}
where $\CS(M)$ is the Chern-Simons invariant of $M$ (with the Riemannian 
connection). Although this result is known to experts, no proof seems to be 
available (see \cite{Snapdescription,Neumann} for discussions). We give a 
proof in Section~\ref{sec.CCS}.
The invariant $\Vol(M)+i\CS(M)$ is often referred to as \emph{complex volume}. 
Motivated by this, we define the complex volume $\Vol_\C$ of a representation 
$\rho\colon\pi_1(M)\to\SL(n,\C)$ by
\begin{equation}
\widehat c(\rho)=i\Vol_\C(\rho)
\end{equation} 
and define the \emph{volume} of $\rho$ to be the real part of the complex 
volume, i.e.~the imaginary part of the Cheeger-Chern-Simons invariant. 
Surprisingly, as we shall see, the relationship to hyperbolic volume seems 
to persist even when $n>2$.

The set of $\SL(n,\C)$-representations is a complex variety with finitely 
many components, and the complex volume is constant on components. This 
follows from the fact that representations in the same component have 
cohomologous Chern-Simons forms.
Hence, for any $M$, the set of complex volumes is a finite set.

We show that the definition of the Cheeger-Chern-Simons invariant naturally 
extends to compact manifolds with boundary, and representations 
$\rho\colon \pi_1(M)\to\SL(n,\C)$ that are \emph{boundary-unipotent}, 
i.e.~take peripheral subgroups to a conjugate of the unipotent group $N$ 
of upper triangular matrices with $1$'s on the diagonal. We formulate all 
our results in this more general setup.

The main result of the paper is a concrete algorithm for computing the set 
of complex volumes. The idea is that the set of (conjugacy classes of) 
boundary-unipotent representations can be parametrized by a variety, 
called the \emph{Ptolemy variety}, which is defined by homogeneous 
polynomials of degree $2$. The Ptolemy variety depends on a choice of 
triangulation, but if the triangulation is sufficiently fine, every 
representation is detected by the Ptolemy variety. We show that a point $c$ in the 
Ptolemy variety naturally determines an element $\lambda(c)$ in Neumann's 
extended Bloch group $\widehat\B(\C)$, such that if $\rho$ is the representation 
corresponding to $c$, we have
\begin{equation}
\lbl{VolCintro}
R(\lambda(c))=i\Vol_\C(\rho),
\end{equation}
where $R\colon\widehat\B(\C)\to\C/4\pi^2\Z$ is a Rogers dilogarithm.

There is a canonical group homomorphism
$$
\phi_n\colon\SL(2,\C)\to\SL(n,\C)
$$ 
coming from the natural $\SL(2,\C)$-action on the vector space 
$\text{Sym}^{n-1}(\BC^2)$. The map $\phi_n$ preserves unipotent elements, 
and we show that composing a boundary-unipotent representation in $\SL(2,\C)$ 
with $\phi_n$ multiplies the complex volume by $\binom{n+1}{3}$. If 
$M=\H^3/\Gamma$ is a hyperbolic $3$-manifold, the geometric representation 
$\rho_{\geo}$ always lifts to a representation in $\SL(2,\C)$, but if $M$ has 
cusps, lifts are not necessarily boundary-unipotent. In fact, by a result of 
Calegari~\cite{Calegari}, if $M$ has a single cusp, any lift of the geometric representation takes a longitude to an element with 
trace $-2$.
When $n$ is even, we shall thus, more generally, be interested in 
boundary-unipotent representations in
\begin{equation}
p\SL(n,\C)=\SL(n,\C)\big/\langle \pm I\rangle.
\end{equation}
Such representations have a complex volume defined modulo $\pi^2i$, and our 
algorithm computes these as well. By studying representations in 
$p\SL(n,\C)$, we make sure that when $M$ is hyperbolic, there is always at 
least one representation with non-trivial complex volume, namely 
$\phi_n\circ \rho_{\geo}$.

Walter Neumann has conjectured that every element in the Bloch group 
$\B(\C)$ is an integral linear combination of Bloch group elements of 
hyperbolic $3$-manifolds. Since the extended Bloch group equals the Bloch 
group up to torsion, Neumann's conjecture would imply that all complex 
volumes are, up to rational multiples of $i\pi^2$, integral linear 
combinations of complex volumes of hyperbolic $3$-manifolds.
In particular, the volumes should all be integral linear combinations of 
volumes of hyperbolic manifolds.

Our algorithm has been implemented by Matthias Goerner. The algorithm uses 
Magma~\cite{Magma} to 
compute a primary decomposition of the Ptolemy variety, and then uses 
\eqref{VolCintro} 
to compute the complex volumes. %We have run the algorithm on all manifolds 
%in the census of 
%cusped hyperbolic manifolds with up to $8$ simplices \cite{SnapPy}, all link 
%complements with up to $11$ crossings \cite{SnapPy}, as well as selected 
%manifolds in 
%the Regina census of closed, prime $3$-manifolds with up to $11$ 
%simplices~\cite{Regina}. 
For $n=2$, we have computed primary decompositions of the Ptolemy varieties for all census manifolds with 
$\leq 8$ simplices (these usually finish within a fraction of a second) and all link complements with $\leq 16$ simplices in the SnapPy census~\cite{SnapPy} of knots with up to 11 crossings and links with up to 10 crossings. 
When there are more than 16 simplices some of the computations don't terminate.
For $n=3$, computations are feasible for many manifolds with up to $4$ simplices, but for $n=4$ the computations run 
out of memory for all manifolds with more than $2$ simplices. It would be 
interesting to perform numerical calculations for $n \geq 4$. 
Our computations have revealed numerous (numerical) examples of linear 
combinations as predicted by Neumann's conjecture. 
To the best of our knowledge, our examples are the first concrete 
computations (the first of which were carried out in $2009$) of the Cheeger-Chern-Simons invariant (complex volume) 
for $n>2$.

\subsection{Statement of our results}
\lbl{sub.results}

This section gives a brief summary of our main results. More details can be 
found in the paper.

Let $M$ be a compact, oriented $3$-manifold with (possibly empty) boundary, 
and let $K$ be a closed $3$-cycle (triangulated complex; see 
Definition~\ref{cycledefn}) homeomorphic to the space obtained from $M$ by 
collapsing each boundary component to a point. We identify each of the 
simplices of $K$ with a standard simplex
\begin{equation}
\Delta^3_n=\left\{(x_0,x_1,x_2,x_3)\in\R^4\bigm \vert 0\leq x_i\leq n,
\quad x_0+x_1+x_2+x_3=n\right\}.
\end{equation}

Let $\Delta^3_n(\Z)$ be the set of points in $\Delta^3_n$ with integral coordinates, and let $\dot\Delta^3_n(\Z)$ be $\Delta^3_n(\Z)$ with the $4$ vertex points removed.

\begin{definition}
\lbl{Ptolemyassignmentdefnintro} 
A \emph{Ptolemy assignment} on 
$\Delta^3_n$ is an assignment $\dot\Delta^3_n(\Z)\to\C^*$, $t\mapsto c_t$, 
of a non-zero complex number $c_t$ to each (non-vertex) integral point $t$ 
of $\Delta^3_n$ such that for each $\alpha\in \Delta^3_{n-2}$, the 
\emph{Ptolemy relation}
\begin{equation}
\lbl{Ptolemyrelationintro}
c_{\alpha_{03}}c_{\alpha_{12}}+c_{\alpha_{01}}c_{\alpha_{23}}=c_{\alpha_{02}}c_{\alpha_{13}}
\end{equation}
is satisfied. Here, $\alpha_{ij}$ denotes the integral point $\alpha+e_i+e_j$.
A Ptolemy assignment on $K$ is a Ptolemy assignment $c^i$ on each simplex 
$\Delta_i$ of $K$ such that the Ptolemy coordinates agree on identified faces.
\end{definition}

\begin{remark} 
The name is inspired by the resemblance of \eqref{Ptolemyrelationintro} 
with the Ptolemy relation between the lengths of the sides and diagonals 
of an inscribed quadrilateral (see Figure~\ref{Ptolemymotivation}). In the 
work of Fock and Goncharov~\cite{FockGoncharov}, the Ptolemy relations 
appear as relations between coordinates on the higher Teichm\"{u}ller 
space when the triangulation of a surface is changed by a flip.
\end{remark}

\begin{figure}[ht]
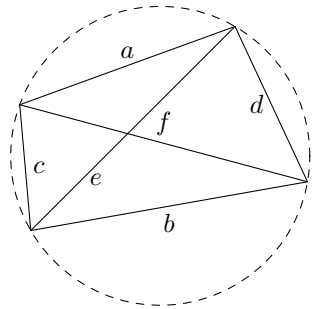

$$
\psdraw{SLrepFigures.21}{4cm}
$$
\caption{A quadrilateral is inscribed in a circle if and only if 
$ab+cd=ef$.}\lbl{Ptolemymotivation}
\end{figure}

It follows immediately from the definition that the set of Ptolemy 
assignments on $K$ is an algebraic set $P_n(K)$, which we shall refer to as 
the \emph{the Ptolemy variety}. As we shall see, the Ptolemy variety 
parametrizes \emph{generically decorated} boundary-unipotent representations 
on in the sense of Definition~\ref{genericdecorationdefn}. If 
$\rho\colon\pi_1(M)\to \SL(n,\C)$ is a boundary-unipotent representation, 
and $E_\rho$ is the corresponding flat $\SL(n,\C)$-bundle, a decoration of 
$\rho$ determines a reduction of $(E_\rho)_{\partial M}$ to an $N$-bundle (see Proposition~\ref{decorationeqreduction}).

% Caption. 

%Note that a Ptolemy assignment on $\Delta^3_n$ induces a Ptolemy assignment on 
%each subsimplex. Maybe put this in a remark in the main text.

The extended pre-Bloch group $\widehat\Pre(\C)$ is generated by 
tuples $(u,v)\in \C^2$ with $e^u+e^v=1$. We refer to Section~\ref{sec.bloch} for a review. Using~\eqref{Ptolemyrelationintro}, we obtain that a 
Ptolemy assignment $c$ on $\Delta^3_n$ gives rise to an element 
\begin{equation}
\lambda(c)=\sum_{\alpha\in T^3(n-2)}(\widetilde c_{\alpha_{03}}+\widetilde c_{\alpha_{12}}-\widetilde c_{\alpha_{02}}-\widetilde c_{\alpha_{13}},\widetilde c_{\alpha_{01}}+\widetilde c_{\alpha_{23}}-\widetilde c_{\alpha_{02}}-\widetilde c_{\alpha_{13}})\in\widehat\Pre(\C),
\end{equation}
where the tilde denotes a branch of logarithm (the particular choice is 
inessential).
We thus have a map 
\begin{equation}
\lbl{betamapdefn}
\lambda\colon P_n(K)\to\widehat\Pre(\C),\quad 
c\mapsto \sum_i\epsilon_i\lambda(c^i),
\end{equation} 
where the sum is over the simplices of $K$.
Let $R_{\SL(n,\C),N}(M)\big/\Conj$ denote the set of conjugacy classes of 
boundary-unipotent representations $\pi_1(M)\to\SL(n,\C)$.

\begin{theorem}[Proof in Section~\ref{Proofofmainthm}]
\lbl{mainthmintro} 
A Ptolemy assignment $c$ uniquely determines a boundary-unipotent representation $\rho(c)\in R_{\SL(n,\C),N}(M)\big/\Conj$. The map $\lambda$ has image in the 
extended Bloch group $\widehat\B(\C)$, and we have a commutative diagram.
\begin{equation}
\lbl{mainthmdiag}
\cxymatrix{@C=3em{P_n(K)\ar[r]^-\lambda\ar[d]^\rho&{\widehat\B(\C)}\ar[d]^R
\\{R_{\SL(n,\C),N}(M)\big/\Conj}\ar[r]^-{i\Vol_\C}&{\C/4\pi^2\Z}}}
\end{equation}
If $c\in P_n(K)$ is a Ptolemy assignment on $K$, $\lambda(c)$ only depends 
on the representation $\rho(c)$. Moreover, if the triangulation 
is sufficiently fine (a single barycentric subdivision suffices), the map 
$\rho$ is surjective.\qed
\end{theorem}

The above theorem (as well as Theorem~\ref{mainthmpSLintro} below) gives an efficient algorithm 
for computing the set of complex volumes. For numerous examples, see Section
\ref{sec.examples}.

\begin{corollary}
\lbl{Corintro}
A boundary-unipotent representation $\rho\colon\pi_1(M)\to \SL(n,\C)$ 
determines an element $[\rho]\in\widehat\B(\C)$ such that
$R([\rho])=i\Vol_\C(\rho)$. \qed
\end{corollary}

The Cheeger-Chern-Simons invariant can be viewed as a characteristic class
$H_3(\SL(n,\C))\to \C/4\pi^2\Z$, and the result underlying the proof of 
commutativity of \eqref{mainthmdiag} is Theorem~\ref{cocycleintro} below, 
giving an explicit cocycle formula for the Cheeger-Chern-Simons class. The 
formula generalizes the formula in Goette-Zickert~\cite{GoetteZickert} for 
$n=2$. Recall that a homology class can be represented by a formal sum of 
tuples $(g_0,\dots,g_3)$. 
To such a tuple, we can assign a Ptolemy assignment $c(g_0,\dots,g_3)$ defined 
by
\begin{equation}
\lbl{Ptolemydeterminants}
c(g_0,\dots,g_3)_t=\det(\{g_0\}_{t_0}\cup\dots\cup\{g_3\}_{t_3}),
\quad t=(t_0,\dots,t_3),
\end{equation}
where $\{g_i\}_{t_i}$ denotes the ordered set consisting of the first $t_i$ 
column vectors of $g_i$. One can always represent a homology class by tuples, 
such that all the determinants in~\eqref{Ptolemydeterminants} are non-zero. 

\begin{theorem}[Proof in Section~\ref{sec.computationofchat}]
\lbl{cocycleintro}
The Cheeger-Chern-Simons class $\widehat c$ factors as
\begin{equation}
\xymatrix{{H_3(\SL(n,\C))}\ar[r]^-{\lambda}&{\widehat\B(\C)}\ar[r]^-R&
{\C/4\pi^2\Z}}, 
\end{equation}
where $\lambda$ is induced by the map taking a tuple $(g_0,\dots,g_3)$ to 
$\lambda(c(g_0,\dots,g_3))\in\widehat\Pre(\C)$.\qed
\end{theorem}

We stress that the variety $P_n(K)$ depends on the triangulation of $K$, 
and may be empty. If a representation $\rho$ is in the image of 
$P_n(K)\to R_{\SL(n,\C),N}(M)\big/\Conj$, we say that $P_n(K)$ \emph{detects} 
$\rho$. 

Let $\phi_n\colon \SL(2,\C)\to\SL(n,\C)$ denote the canonical irreducible 
representation. Note that when $n$ is odd $\phi_n$ factors through $\PSL(2,\C)$.

\begin{theorem}[Proof in Section~\ref{sub.essential}]
\lbl{detectgeo} 
Suppose $M=\H^3/\Gamma$ is an oriented, 
hyperbolic manifold with finite volume and geometric representation 
$\rho_{\geo}\colon\pi_1(M)\to\PSL(2,\C)$. If the triangulation of $K$ has no 
non-essential edges, and if $n$ is odd, $P_n(K)$ is non-empty and detects 
$\phi_n\circ \rho_{\geo}$.  \qed
\end{theorem}

When $n$ is even, $\phi_n\circ\rho_{\geo}$ is only a representation in 
$p\SL(n,\C)=\SL(n,\C)\big/\langle\pm I\rangle$.

\begin{definition}
Let $\sigma\in Z^2(\Delta^3_n;\Z/2\Z)$ be a cocycle. A 
$p\SL(n,\C)$-\emph{Ptolemy assignment} on $\Delta^3_n$ with 
\emph{obstruction cocycle} $\sigma$ is an assignment of Ptolemy 
coordinates to the integral points of $\Delta^3_n$ such that 
\begin{equation}
\sigma_2\sigma_3c_{\alpha_{03}}c_{\alpha_{12}}
+\sigma_0\sigma_3c_{\alpha_{01}}c_{\alpha_{23}}=c_{\alpha_{02}}c_{\alpha_{13}}.
\end{equation}
Here $\sigma_{i}\in\Z/2\Z=\langle\pm 1\rangle$ is the value of $\sigma$ 
on the face opposite the $i$th vertex of $\Delta^3_n$. A $p\SL(n,\C)$-Ptolemy 
assignment on $K$ with obstruction cocycle $\sigma\in Z^2(K;\Z/2\Z)$ is a collection of 
$p\SL(n,\C)$-Ptolemy assignments $c^i$ on $\Delta_i$ with obstruction class 
$\sigma_{\Delta_i}$ such that the Ptolemy coordinates agree on common faces.
\end{definition}

The set of $p\SL(n,\C)$-Ptolemy assignments on $K$ with obstruction cocycle 
$\sigma$ is an algebraic set $P_n^\sigma(K)$, which up to canonical 
isomorphism, only depends on the cohomology class of $\sigma$. The obstruction class
to lifting a boundary-unipotent representation in $p\SL(n,\C)$ to a boundary-unipotent representation in $\SL(n,\C)$ is 
a class in $H^2(M,\partial M;\Z/2\Z)=H^2(K;\Z/2\Z)$. For $\sigma\in H^2(K;\Z/2\Z)$, let 
$R^\sigma_{p\SL(n,\C),N}(M)$ denote the set of boundary-unipotent representations in $p\SL(n,\C)$ 
with obstruction class $\sigma$. 

\begin{theorem}[Proof in Section~\ref{Proofofmainthm}]
\lbl{mainthmpSLintro} 
Let $n$ be even. For each 
$\sigma\in H^2(K;\Z/2\Z)$, we have a commutative diagram
\begin{equation}
\cxymatrix{@C=3em{P^\sigma_n(K)\ar[r]^-\lambda\ar[d]^\rho&
{\widehat\B(\C)_{\PSL}}\ar[d]^R\\
{R^{\sigma}_{p\SL(n,\C),N}(M)\big/\Conj}\ar[r]^-{i\Vol_\C}&{\C/\pi^2\Z}}}
\end{equation}
The extended Bloch group element of a Ptolemy assignment $c$ only depends on the 
representation $\rho(c)$, and if the triangulation of $K$ is sufficiently 
fine, $\rho$ is surjective. If $M=\H^3/\Gamma$ is hyperbolic, and if $K$ 
has no non-essential edges, $P_n^{\sigma_{\geo}}(K)$ detects 
$\phi_n\circ\rho_{\geo}$. Here $\sigma_{\geo}$ is the obstruction class to 
lifting the geometric representation to a boundary-unipotent representation 
in $\SL(2,\C)$.\qed
\end{theorem}

\begin{remark} 
As a corollary, we see that a boundary-unipotent representation in 
$p\SL(n,\C)$ determines an element in $\widehat\B(\C)_{\PSL}$ computing the 
complex volume. For $n=2$, this corollary (and also Corollary~\ref{Corintro}), 
was proved in Zickert~\cite{ZickertDuke}.
\end{remark}
\begin{remark}
If the triangulation has a non-essential edge, all Ptolemy varieties are empty. Hence, if $P^\sigma_2(K)$ is non-empty for some $\sigma$, and if $M$ is hyperbolic, the Ptolemy variety $P^{\sigma_{\geo}}(K)$ will detect the geometric representation.
\end{remark}

\begin{theorem}[Proof in Section~\ref{sec.adjoint}]
\lbl{thm.rhoSL2}
Let $\rho$ be a representation in $\SL(2,\C)$ or $\PSL(2,\C)$. The extended 
Bloch group element of $\phi_n\circ \rho$ is $\binom{n+1}{3}$ times that of 
$\rho$. In particular, composition with $\phi_n$ multiplies complex volume by 
$\binom{n+1}{3}$.\qed
\end{theorem}

When $n=2$, Thurston's gluing equation variety $V(K)$ is another variety, 
which is often used to compute volume. It is given by an equation 
for each edge of $K$ and an equation for each generator of the fundamental 
groups of the boundary-components of $M$ (see Section~\ref{sec.gluingequations}).

\begin{theorem}[Proof in Section~\ref{sec.gluingequations}]
\lbl{thm.Mhboundary}
Suppose $M$ has $h$ boundary components. There is a surjective regular map
\begin{equation}
\coprod_{\sigma\in H^2(K;\Z/2\Z)} P_2^\sigma(K)\to V(K)
\end{equation}
with fibers disjoint copies of $(\C^*)^h$.\qed
\end{theorem}
\begin{remark} The Ptolemy variety offers significant computational advantage over the gluing equations and allows for exact computations that are practical even if the manifold has many simplices.
\end{remark}

As shown in Zickert~\cite{ZickertAlgK}, the extended Bloch group can also 
be defined over a number field $F$, and we have a canonical isomorphism 
$\widehat\B(F)\cong K_3^{\ind}(F)$.
\begin{theorem}[Proof in Section~\ref{otherfields}] Let $F$ be a number field. A boundary-unipotent 
representation $\rho\colon\pi_1(M)\to \SL(n,F)$ determines an element of 
$\widehat\B(F)=K_3^{\ind}(F)$ such that for each embedding $\tau\colon F\to\C$, 
we have 
\begin{equation}
R(\tau([\rho]))=i\Vol_\C(\tau\circ\rho).
\end{equation}
If $\rho$ is irreducible, $[\rho]$ lies in $\widehat\B(\Tr(\rho))$, where 
$\Tr(\rho)\subset F$ is the trace field of $\rho$.\qed 
\end{theorem}

\subsection{Neumann's conjecture} 
\lbl{sub.neumannc}

The fact that \eqref{betamapdefn} has image in $\widehat\B(\C)$ as 
opposed to $\widehat\Pre(\C)$ has very interesting conjectural consequences. 
It is well known (see e.g.~Suslin~\cite{Suslin}) that the Bloch group $\B(\C)$ 
is a $\Q$-vector space, and Walter Neumann has conjectured that it is generated by Bloch invariants of hyperbolic manifolds.
More generally, Walter Neumann has proposed the following stronger conjecture~\cite{NeumannRealize}:

%\begin{conjecture}[Neumann]
%\lbl{WeakNeumannConj}
%The Bloch group $\B(\C)$ is generated over $\Q$ by Bloch group elements of 
%hyperbolic manifolds.
%\end{conjecture}

%This conjecture implies that every element of $\B(\C)$ comes from an element 
%in $\B(F)$, where $F$ is a concrete number field.  
\begin{conjecture}\lbl{StrongNeumannConj} Let $F\subset\C$ be a 
concrete number field which is not in $\R$. The Bloch group $\B(F)$ is 
generated (integrally) modulo torsion by hyperbolic manifolds with invariant 
trace field contained in $F$.
\end{conjecture}
%For a discussion of these conjectures see~\cite{NeumannRealize}. 
Using Theorems~\ref{mainthmintro} and \ref{mainthmpSLintro}, 
Conjecture~\ref{StrongNeumannConj} implies:

\begin{conjecture}
\lbl{newconjecture} 
Let $\rho$ be a boundary-unipotent 
representation of $\pi_1(M)$ in $\SL(n,\C)$ or $p\SL(n,\C)$.  There exist 
hyperbolic $3$-manifolds $M_1,\dots,M_k$ and integers $r_1,\dots ,r_k$ such that
\begin{equation}
\Vol_\C(\rho)=\sum r_i\Vol_\C(M_i)\in \C/i\pi^2\Q.
\end{equation}
In particular, $\Vol(\rho)=\sum r_i\Vol(M_i)\in\R$.
\end{conjecture}

We give some examples in Section~\ref{sec.examples}.    

\begin{remark}
The concept of a Ptolemy assignment and its associated shapes first appeared in Zickert~\cite{ZickertAlgK} (where it was called an \emph{ideal cochain}).
Subsequently, other authors have studied similar ideas. These include Bergeron-Falbel-Guilloux~\cite{FrenchPeople}, Garoufalidis-Goerner-Zickert~\cite{GaroufalidisGoernerZickert} and Dimofte-Gabella-Goncharov~\cite{DGon}.
\end{remark}

\subsection{Overview of the paper}
\lbl{sub.plan}

Section~\ref{sec.CCS} gives a detailed review of the Cheeger-Chern-Simons 
classes for flat bundles. Many details are included in order to give a 
self-contained proof of~\eqref{CCSeqVolC}. Section~\ref{sec.bloch} gives a 
brief review of the two variants of the extended Bloch group, and 
Section~\ref{sec.decorations} reviews the theory, introduced in 
Zickert~\cite{ZickertDuke}, of decorated representations and relative 
fundamental classes. In Section~\ref{sec.genericdecorations}, we introduce 
the notion of \emph{generic} decorations and define the Ptolemy variety $P_n(K)$.
In Section~\ref{sec.ptolemy}, we construct a chain complex of Ptolemy assignments, and use it to construct a 
map from $H_3(\SL(n,\C),N)$ to $\widehat\B(\C)$ commuting with stabilization. 
This shows that a decorated boundary-unipotent representation determines an element in the extended Bloch group, 
which is given explicitly in terms of the Ptolemy coordinates.
In Section~\ref{decorationsection}, we show that the extended Bloch group 
element of a decorated representation is independent of the decoration, 
and in Section~\ref{sec.computationofchat}, we show that the 
Cheeger-Chern-Simons class is given as in Theorem~\ref{cocycleintro}. In 
Section~\ref{proofsection}, we show that the Ptolemy variety parametrizes generically decorated representations,
and give an explicit formula for recovering a representation from its Ptolemy coordinates. 
In Section~\ref{sec.examples}, we give some examples of 
computations, and list some interesting findings. Section~\ref{sec.adjoint} 
discusses the irreducible representations of $\SL(2,\C)$, and 
Section~\ref{sec.gluingequations} discusses the relationship to Thurston's 
gluing equations when $n=2$. Finally, Section~\ref{otherfields} is a brief 
discussion of other fields.

\subsection{Acknowledgment}
The authors wish to thank Ian Agol, Johan Dupont, Matthias Goerner and 
Walter Neumann for stimulating conversations. We are particularly grateful 
to Matthias Goerner for a computer implementation of our formulas, and for
supplying our theory with computational data for more than 20000 manifolds.
The software has been incorporated into SnapPy~\cite{SnapPy}, and computational data can be found at \url{http://unhyperbolic.org/ptolemy.html}.

\section{The Cheeger-Chern-Simons classes}
\lbl{sec.CCS}

%\subsection{Secondary characteristic classes}
%\lbl{sub.secondary}

The Cheeger-Chern-Simons classes~\cite{CheegerSimons,ChernSimons} are 
characteristic classes of principal bundles with connection. For general 
bundles, the characteristic classes are differential 
characters~\cite{CheegerSimons}, but for flat bundles they reduce to 
ordinary (singular) cohomology classes. In this paper we will focus 
exclusively on flat bundles. %All bundles are assumed to be smooth.
Let $\F$ denote either $\R$ or $\C$, and let $\Lambda$ be a proper subring 
of $\F$. Let $G$ be a Lie group over $\F$ with finitely many components. 
There is a characteristic class $S_{P,u}$ for each pair $(P,u)$ consisting 
of an invariant polynomial $P\in I^k(G;\F)$ and a class 
$u\in H^{2k}(BG;\Lambda)$, whose image in $H^{2k}(BG;\F)$ equals $W(P)$, 
where $W$ is the Chern-Weil homomorphism
\begin{equation}
W\colon I^k(G;\F)\to H^{2k}(BG;\F).
\end{equation}
The characteristic class $S_{P,u}$ associates to each flat $G$-bundle 
$E\to M$ a cohomology class $S_{P,u}(E)\in H^{2k-1}(M;\F/\Lambda)$.

\subsection{Simply connected, simple Lie groups} 
\lbl{sub.simplelie}

If $G$ is simply connected and simple, $H^1(G;\Z)$ and $H^2(G;\Z)$ are 
trivial, and $H^3(G;\Z)\cong\Z$. Hence, by the Serre spectral sequence 
for the universal bundle, we have an isomorphism
\begin{equation}
S\colon H^{4}(BG;\Z)\cong H^{3}(G;\Z)\cong\Z
\end{equation}
called the \emph{suspension}.
The Killing form on $G$ defines an invariant polynomial $B\in I^2(G;\F)$, 
and since $B$ is real on the maximal compact subgroup $K$ of $G$, $W(B)$ 
is a real class. Hence, there exists a unique positive real number $\alpha$ 
such that $W(\alpha B)$ is a generator of $H^4(BG;4\pi^2\Z)$. We refer to 
$\alpha B$ as the \emph{renormalized Killing form}, and denote the 
Cheeger-Chern-Simons class $S_{\alpha B,W(\alpha B)}$ by $\widehat c$. 

Recall that every class in $H^3(G;\F)$ can be represented by a $G$-invariant 
$3$-form.
The following is well known (see e.g.~Kamber-Tondeur~\cite[(5.74) p.~116]{KamberTondeur}).

\begin{proposition} Let $P\in I^2(G;\F)$. The suspension of $W(P)$ is 
represented by the invariant $3$-form
\begin{equation}\lbl{sigmaPformula}
\sigma(P)=-\frac{1}{6}P(\omega\wedge[\omega,\omega])\in\Omega^3(G;\F)^G
\end{equation}
where $\omega$ is the Maurer-Cartan form on $G$.\qed
\end{proposition}

Let $E\to M$ be a $G$-bundle with flat connection $\theta$. We can view 
$\theta$ as a map $\mathfrak g^*\to\Omega^1(E;\F)$, so by taking exterior 
powers, $\theta$ induces a map
\begin{equation}
\theta\colon\Omega^3(G)^G=\wedge^3(\mathfrak g^*)\to\Omega^3(E;\F).
\end{equation}

Note that $\theta(\sigma(P))=-\frac{1}{6}P(\theta\wedge[\theta,\theta])$. 
In the following, $P$ denotes the renormalized Killing form.

\begin{proposition}[{\cite[Proposition~2.8]{CheegerSimons}}]
Let $E\to M$ be a $G$-bundle, with flat connection $\theta$, over a closed 
$3$-manifold $M$.
The cohomology class $\widehat c(E)\in H^3(M;\F/4\pi^2\Z)$ satisfies 
\begin{equation}
\lbl{CSintegral}
\widehat c(E)([M])=\int_{M} s^*\big(\theta(\sigma(P))\big)\in \F/4\pi^2\Z,
\end{equation}
where $s$ is a section of $E$ (which exists since $G$ is $2$-connected).\qed
\end{proposition}

\begin{remark}
Since $\sigma(P)\in H^3(G;4\pi^2\Z)$ is a generator, it follows that a 
change of section changes the integral by a multiple of $4\pi^2\Z$.
\end{remark}

\begin{example}
\lbl{SLnExample}
For $G=\SL(n,\C)$, the renormalized killing form $P$ equals $\frac{1}{2}\Tr$, 
where $\Tr$ is the trace form $(A,B)\mapsto \Tr(AB)$. For a flat connection, $d\theta=-\frac{1}{2}[\theta,\theta]=-\theta\wedge\theta$, so~\eqref{CSintegral} yields
\begin{equation}
\widehat c(E)([M])=\frac{1}{2}\int_{M} 
s^*\big(\Tr(\theta\wedge d\theta+\frac{2}{3}\theta\wedge\theta\wedge\theta)\big)\in 
\C/4\pi^2\Z
\end{equation}
recovering the Chern-Simons integral~\eqref{ChernSimonsintegral}. Note that $P$ also equals the (renormalized) second Chern-polynomial $c_2$. 
It thus follows that $\widehat c=\widehat c_2$. 
\end{example}

\subsection{Complex groups and volume}
\lbl{sub.cvolume}

Recall that there is a $1$-$1$ correspondence between flat $G$-bundles over 
$M$ and representations $\pi_1(M)\to G$ up to conjugation. This correspondence 
takes a flat bundle to its holonomy representation. If 
$\rho\colon\pi_1(M)\to G$ is a representation, we let $E_\rho$ denote the 
corresponding flat bundle.
In the following $G$ denotes a simply connected, simple, complex Lie group, 
and $M$ a closed, oriented $3$-manifold. 
The following definition is motivated by Theorem~\ref{SL2complexvolume} 
below.

\begin{definition}
\lbl{compvoldefn} 
The \emph{complex volume} $\Vol_\C(\rho)$ of a representation 
$\rho\colon\pi_1(M)\to G$ is defined by
\begin{equation}
\widehat c(E_\rho)(M)=i\Vol_\C(\rho)\in\C/4\pi^2\Z.
\end{equation}
The \emph{volume} $\Vol(\rho)$ of $\rho$ is the real part of $\Vol_\C(\rho)$.
\end{definition}

The bundle $E_\rho$ is isomorphic to $\widetilde M\times_\rho G$, and we thus have a $1$-$1$ correspondence between
sections of $E_\rho$ and $\rho$-equivariant maps $\widetilde M\to G$ such that $f\colon\widetilde M\to G$ 
corresponds to the section $s(x)=[\widetilde x,f(\widetilde x)]$. 
\begin{lemma}
\lbl{rhomapandVolC}
For any $\rho$-equivariant map $f\colon\widetilde M\to G$, we have 
$i\Vol_\C(\rho)=\int_D f^*(\sigma(P))$, where $D$ is a fundamental domain for $M$ in $\widetilde M$.
\end{lemma}
\begin{proof}
For any invariant form $\eta\in\Omega^3(G)^G$, the form $\theta(\eta)\in \Omega^3(E_\rho;\F)$ 
is induced by the pullback of $\eta$ under the projection $\widetilde M\times G\to G$. Letting $\eta=\sigma(P)$, the result follows from~\eqref{CSintegral}.
\end{proof}

Let $\H^3=\SL(2,\C)/\SU(2)$ be hyperbolic $3$-space. We identify the 
orthonormal frame bundle $F(\H^3)$ of $\H^3$ with $\PSL(2,\C)$.

\begin{lemma}
\lbl{sigmaPandCS} 
For $G=\SL(2,\C)$, $\sigma(P)=-h^*\wedge e^*\wedge f^*$, where
$h=\Mat{1}{0}{0}{-1}$, $e=\Mat{0}{1}{0}{0}$ and $f=\Mat{0}{0}{1}{0}$ are the 
standard generators of $\mathfrak{sl}(2,\C)$ over $\C$.
\end{lemma}

\begin{proof}
As in Example~\ref{SLnExample}, $P=\frac{1}{2}\Tr$. Using the fact that 
$\Tr(AB)=\Tr(BA)$, it follows from~\eqref{sigmaPformula} that 
$\sigma(P)\in\Omega^3(G)^G=\wedge^3(\mathfrak g^*)$ is given by
\begin{equation}
\mathfrak g\times\mathfrak g\times\mathfrak g\to \C,\quad (A,B,C)\mapsto -\frac{1}{2}\Tr(A[B,C]).
\end{equation}
A simple computation shows that if $A=\Mat{a}{b}{c}{-a}$, $B=\Mat{e}{f}{g}{-e}$ and $C=\Mat{i}{j}{k}{-i}$
\begin{equation}
-\frac{1}{2}\Tr(A[B,C])=-\det\left(\begin{smallmatrix}a&b&c\\e&f&g\\i&j&k\end{smallmatrix}\right)=-h^*\wedge e^*\wedge f^*(A,B,C).
\end{equation}
This proves the result.
%Dupont~\cite[p.~148]{Dupont} proves that $\Vol_{\H^3}$ is cohomologous to $2\Imag (\zeta_2^*\wedge\zeta_3^*\wedge\zeta_4^*$), where $\zeta_2=-ih^*$, $\zeta_3=\frac{1}{2}(f-e)$, $\zeta_4=-\frac{1}{2}i(f+e)$. %Since $2\zeta_2^*\wedge\zeta_3^*\wedge\zeta_4^*=h^*\wedge e^*\wedge f^*$ %it follows that the constant $k$ is $-1$. This proves the result.
\end{proof}

\begin{theorem}
\lbl{SL2complexvolume} 
Let $M=\H^3/\Gamma$ be a closed hyperbolic $3$-manifold, and let $\rho\colon\pi_1(M)\to\SL(2,\C)$ be a lift of the geometric representation. We have
\begin{equation}
\widehat c(E_\rho)([M])=i(\Vol(M)+i\CS(M))\text{ in }\C/2\pi^2\Z,
\end{equation}
where $\CS(M)=2\pi^2\cs(M)$, and $\cs(M)$ is the (Riemannian) Chern-Simons invariant~\cite[(6.2)]{ChernSimons}.
\end{theorem}

\begin{proof} 
The fact that the imaginary part equals volume is well known, and follows from the fact (see Dupont~\cite{Dupont}) that the imaginary part of $\sigma(P)$ is cohomologous to the pullback of the hyperbolic volume form. Yoshida~\cite[Lemma~3.1]{Yoshida} shows that the real part of the form $h^*\wedge e^*\wedge f^*$ equals $2\pi^2\cs$, where $\cs$ is the Riemannian Chern-Simons form on $F(\H^3)=\PSL(2,\C)$ (pulled back to $\SL(2,\C)$. Note that the Riemannian connection on $F(\H^3)=\PSL(2,\C)$ descends to the Riemannian connection on $F(M)=\PSL(2,\C)/\Gamma$. If $f\colon\widetilde M\to\SL(2,\C)$ is $\rho$-equivariant, the composition
\begin{equation}
\xymatrix{{\widetilde M}\ar[r]^-f&{\SL(2,\C)}\ar[r]&{\PSL(2,\C)}\ar[r]&{\PSL(2,\C)/\Gamma=F(M)}}
\end{equation}
is $\rho$-invariant, and thus descends to a section of $F(M)$. The result now follows from Yoshida's result together with Lemma~\ref{sigmaPandCS} and Lemma~\ref{rhomapandVolC}. 
\end{proof}

\begin{remark} 
Note that Theorem~\ref{SL2complexvolume} implies that modulo $2\pi^2$, the complex volume of a representation lifting the geometric representation only depends on $M$ and not on the choice of lift.
\end{remark}

\begin{remark} 
Since $P$ is real on $K$, the imaginary part of $\sigma(P)$ is cohomologous to an invariant $3$-form on $G/K$. Since $H^3(\mathfrak g,\mathfrak k;\R)=\R$, there is a unique such form up to scaling. We may thus think of $\Imag(\sigma(P))$ as a volume form. 
\end{remark}

\subsection{The universal classes and group cohomology}
\lbl{sub.universal}

The Cheeger-Chern-Simons classes are also defined for the universal flat bundle $EG^\delta\to BG^\delta$. For an explicit construction, we refer to Dupont-Kamber~\cite{DupontKamber} or Dupont-Hain-Zucker~\cite{DupontHainZucker}.
In particular, we have a class $\widehat c\in H^3(BG^\delta;\C/4\pi^2\Z)$. 
If $\rho\colon \pi_1(M)\to G$ is a representation, with classifying map $B\rho\colon M\to BG^\delta$, we thus have

\begin{equation}
\widehat c(B\rho_*([M]))=i\Vol_\C(\rho).
\end{equation}

It is well known that the homology of $BG^\delta$ is the homology of the chain complex $C_*\otimes_{\Z[G]}\Z$, where $C_*$ is any free $\Z[G]$-resolution of $\Z$. A convenient choice of free resolution is the complex $C_*$, generated in degree $n$ by tuples $(g_0,\dots, g_n)$, and with boundary map given by 
\begin{equation}
\partial (g_0,\dots,g_n)=\sum (-1)^i (g_0,\dots,\widehat g_i,\dots,g_n).
\end{equation} 
%It is a long standing problem to find general formulas for the $\widehat c$. 
The homology of $C_*\otimes_{\Z[G]}\Z$ is denoted $H_*(G)$, so $H_*(G)=H_*(BG^\delta)$. Theorem~\ref{cocycleintro} gives a concrete cocycle formula for $\widehat c\colon H_3(\SL(n,\C))\to \C/4\pi^2\Z$.

\subsection{Compact manifolds with boundary}
\lbl{sub.compactMb}
%Let $N$ be the group of upper triangular matrices with $1$'s on the diagonal.
In Section~\ref{sub.map2ebloch} below, we construct a natural extension of $\widehat c\colon H_3(\SL(n,\C))\to\C/4\pi^2\Z$ to a homomorphism
\begin{equation}
\widehat c\colon H_3(\SL(n,\C),N)\to \C/4\pi^2\Z,
\end{equation}
where $N$ is the subgroup of upper triangular matrices with $1$'s on the diagonal.
\begin{definition}
\lbl{complexvolumedefn} 
Let $\rho\colon\pi_1(M)\to \SL(n,\C)$ be a boundary-unipotent representation. The \emph{complex volume} of $\rho$ is defined by
\begin{equation}
\widehat c(B\rho_*([M,\partial M]))=i\Vol_\C(\rho),
\end{equation}
where $B\rho\colon(M,\partial M)\to (B\SL(n,\C)^\delta,BN^\delta)$ is a classifying map for $\rho$.
\end{definition}

\begin{remark}
Unlike when $M$ is closed, the classifying map is not uniquely determined by $\rho$; it depends on a choice of decoration (see Section~\ref{sec.decorations}). The complex volume, however, is independent of this choice.
\end{remark}

\subsection{Central elements of order $2$}
\lbl{sub.central}

For any simple complex Lie group $G$, there is a canonical homomorphism (defined up to conjugation) 
\begin{equation}
\phi_G\colon \SL(2,\C)\to G.
\end{equation}
The element $s_G=\phi_G(-I)$ is a central element of $G$ of order dividing $2$, and equals $(-I)^{n+1}$ if $G=\SL(n,\C)$ (see e.g.~Fock-Goncharov~\cite[Corollary~2.1]{FockGoncharov}). Let  
\begin{equation}
pG=G/\langle s_G\rangle.
\end{equation} 
Note that $\phi_G$ descends to a homomorphism $\PSL(2,\C)\to pG$. The following follows easily from the Serre spectral sequence.

\begin{proposition}
\lbl{orderfourkernel}
Suppose $s_G$ has order $2$. The canonical map $p^*\colon H^4(BpG;\Z)\to H^4(BG;\Z)$ is surjective with kernel of order dividing $4$.\qed
\end{proposition}

\begin{corollary}
There is a canonical characteristic class 
$\widehat c\colon H_3(pG)\to \C/\pi^2\Z$.
\end{corollary}

\begin{proof}
By Proposition~\ref{orderfourkernel}, there exists a canonical class $u\in
H^4(BpG;\pi^2\Z)$ such that $p^*(u)=W(P)\in H^4(BG\semico \pi^2\Z)$. Define $\widehat c=S_{P,u}$.
\end{proof}

In Section~\ref{sub.pSLassignments}, we construct a homomorphism 
\begin{equation}
\widehat c\colon H_3(p\SL(n,\C),N)\to \C/\pi^2\Z,
\end{equation}
which extends $\widehat c$ to a characteristic class of bundles with boundary-unipotent holonomy. The complex volume of a representation in $p\SL(n,\C)$ is defined as in Definition~\ref{complexvolumedefn}.

\section{The extended Bloch group}
\lbl{sec.bloch}

We use the conventions of Zickert~\cite{ZickertAlgK}; the original 
reference is Neumann~\cite{Neumann}.

\begin{definition} 
\lbl{def.bloch}
The \emph{pre-Bloch group} $\Pre(\C)$ is the free abelian group on $\C\setminus\{0,1\}$ modulo the \emph{five term relation}
\begin{equation}
x-y+\frac{y}{x}-\frac{1-x^{-1}}{1-y^{-1}}+\frac{1-x}{1-y}=0,\quad \text{for }x\neq y\in\C\setminus\{0,1\}.
\end{equation}
The \emph{Bloch group} is the kernel of the map $\nu\colon \Pre(\C)\to\wedge^2(\C^*)$ taking $z$ to $z\wedge(1-z)$.
\end{definition}

\begin{definition} 
\lbl{def.ebloch}
The \emph{extended pre-Bloch group} $\widehat\Pre(\C)$ is the free abelian group on the set
\begin{equation}
\widehat\C=\left\{(e,f)\in\C^2 \bigm\vert \exp(e)+\exp(f)=1\right\}
\end{equation} 
modulo the \emph{lifted five term relation}
\begin{equation}
(e_0,f_0)-(e_1,f_1)+(e_2,f_2)-(e_3,f_3)+(e_4,f_4)=0
\end{equation}
if the equations
\begin{equation}
\begin{gathered}\lbl{fiveeq}
e_2=e_1-e_0,\quad e_3=e_1-e_0-f_1+f_0,\quad f_3=f_2-f_1\\
e_4=f_0-f_1,\quad f_4=f_2-f_1+e_0
\end{gathered}
\end{equation}
are satisfied. The \emph{extended Bloch group} is the kernel of the map $\widehat\nu\colon\widehat\Pre(\C)\to\wedge^2(\C)$ taking $(e,f)$ to $e\wedge f$.
\end{definition}

An element $(e,f)\in\widehat \C$ with $\exp(e)=z$ is called a \emph{flattening} with \emph{cross-ratio} $z$.
Letting $\mu_{\C}$ denote the roots of unity in $\C^*$, we have a
commutative diagram. 
\begin{equation}
\lbl{bigdiagram}
\cxymatrix{{&0\ar[d]&0\ar[d]&0\ar[d]&&\\0\ar[r]&\mu_{\C}\ar[r]^-{2\log}\ar[d]^-{\chi}&{\C/4\pi i\Z}\ar[r]\ar[d]^{\chi}&{\C^*/\mu_{\C}}\ar[r]\ar[d]&0\ar[d]&\\
0\ar[r]&{\widehat \B(\C)}\ar[r]\ar[d]^-\pi&{\widehat \Pre(\C)}\ar[r]^-{\widehat \nu}\ar[d]^-\pi&{\wedge^2(\C)}\ar[r]\ar[d]&
{K_2(\C)}\ar@2{-}[d]\ar[r]&0\\0\ar[r]&{\B(\C)}\ar[r]\ar[d]&{\Pre(\C)}\ar[r]^-\nu\ar[d]&{\wedge^2(\C^*)\ar[d]}\ar[r]&{K_2(\C)}\ar[r]\ar[d]&0\\&0&0&0&0}}
\end{equation}
The map $\pi$ is induced by the map taking a flattening to its cross-ratio, and $\chi$ is the map taking $e\in\C/4\pi i\Z$ to $(e,f+2\pi i)-(e,f)$, where $f\in\C$ is any element such that $(e,f)\in\widehat\C$.

\subsection{The regulator}
\lbl{sub.regulator}

By fixing a branch of logarithm, we may write a flattening with cross-ratio $z$ as $[z;p,q]=\big(\log(z)+p\pi i,\log(1-z)+q\pi i\big)$, where $p,q\in\Z$ are \emph{even} integers. There is a well defined \emph{regulator map}
\begin{equation}
\lbl{regulator}
\begin{gathered}
R\colon\widehat\Pre(\C)\to\C/4\pi^2\Z,\\
{[z;p,q]}\mapsto \Li_2(z)+\frac{1}{2}(\log(z)+p\pi i)(\log(1-z)-q\pi i)-\pi^2/6.
\end{gathered}
\end{equation}

\subsection{The $\PSL(2,\C)$-variant of the extended Bloch group}
\lbl{sub.vbloch}

There is another variant of the extended Bloch group using flattenings 
$[z;p,q]$, where $p$ and $q$ are allowed to be odd. This group is defined 
as above using the set
\begin{equation}
\widehat{\C}_{\odd}=\left\{(e,f)\in\C^2 \bigm\vert \pm\exp(e)\pm\exp(f)=1\right\},
\end{equation}
and fits in a diagram similar to~\eqref{bigdiagram}. We use a subscript 
$\PSL$ to denote the variant allowing odd flattenings. We have an exact 
sequence
\begin{equation}
\xymatrix{0\ar[r]&{\Z/4\Z}\ar[r]&{\widehat\B(\C)}\ar[r]&
{\widehat\B(\C)_{\PSL}}\ar[r]&0.}
\end{equation}
For odd flattenings, the regulator~\eqref{regulator} is well defined modulo 
$\pi^2\Z$.

\begin{theorem}
[Neumann~\cite{Neumann}, Goette-Zickert~\cite{GoetteZickert}] 
There are natural isomorphisms
\begin{equation}
H_3(\PSL(2,\C))\cong\widehat\B(\C)_{\PSL},\quad H_3(\SL(2,\C))\cong\widehat\B(\C)
\end{equation}
such that the Cheeger-Chern-Simons classes agree with the regulators.\qed
\end{theorem}

The following result is needed in Section~\ref{decorationsection}. The 
first part is proved in Zickert \cite[Lemma~3.16]{ZickertAlgK}, and 
the second has a similar proof, which we leave to the reader.
\begin{lemma}
\lbl{chipqlemma} 
For $(e,f)\in\widehat\C$ and $p,q\in\Z$, we have%
\setlength{\belowdisplayskip}{0pt}%
\setlength{\belowdisplayshortskip}{0pt}%
\setlength{\parskip}{0pt}%
\begin{align}
(e+2\pi ip,f+2\pi iq)-(e,f)&=\chi(qe-pf+2pq\pi i)\in\widehat\Pre(\C),\\
(e+\pi ip,f+\pi iq)-(e,f)&=\chi(qe-pf+pq\pi i)\in\widehat\Pre(\C)_{\PSL}.
\end{align}

\nopagebreak%
\vspace{-\baselineskip}%
\nopagebreak%
\qed
\end{lemma}

\subsection{Arbitrary fields}
\lbl{sub.allfields}

In Zickert~\cite{ZickertAlgK}, extended Bloch groups $\widehat\B_E(F)$ and 
$\widehat\B_E(F)_{\PSL}$ are defined for an arbitrary field $F$ and a primitive extension $E$ of $F^*$ by $\Z$. The definitions 
are as above using the sets
\begin{equation}
\widehat E_F=\left\{(e,f)\in E^2\bigm\vert \pi(e)+\pi(f)=1\right\},
\quad (\widehat E_F)_{\odd}=\left\{(e,f)\in E^2\bigm\vert 
\pm\pi(e)\pm\pi(f)=1\right\}.
\end{equation}
If $F$ is a number field, the extended Bloch groups are up to canonical isomorphism 
independent of the choice of extension, so we may omit the subscript $E$. 

\begin{theorem} 
[{Zickert~\cite[Theorem~1.1]{ZickertAlgK}}]  
Let $F$ be a number field. There is a natural isomorphism
\begin{equation}
K_3^{\ind}(F)\cong\widehat\B(F)
\end{equation}
respecting Galois actions.\qed
\end{theorem}

\begin{corollary}
[{Zickert~\cite[Corollary~7.14]{ZickertAlgK}}]
\lbl{injectivity} 
For each embedding $\tau\colon F\to\C$, the induced map 
$\tau\colon\widehat\B(F)\to\widehat\B(\C)$ is injective.\qed
\end{corollary}

\begin{corollary}
[{Galois descent; Zickert~\cite[Corollary~7.15]{ZickertAlgK}}]
\lbl{Galoisdescent}
Let $F_2:F_1$ be an extension of number fields. An element in $\widehat\B(F_2)$ is in 
$\widehat\B(F_1)$ if and only if it is invariant under all automorphisms of $F_2$ over $F_1$.\qed 
\end{corollary}

\section{Decorations of representations}
\lbl{sec.decorations}

In this section we review the notion of decorated representations 
introduced in Zickert~\cite{ZickertDuke}. Throughout the section, $G$ 
denotes an arbitrary group, not necessarily a Lie group. Let $H$ be subgroup 
of $G$. An \emph{ordered simplex} is a simplex with a fixed vertex ordering.

\begin{definition}
\lbl{cycledefn} 
A \emph{closed $3$-cycle} is a cell complex $K$ obtained from a finite 
collection of ordered $3$-simplices $\Delta_i$ by gluing together pairs 
of faces using order preserving simplicial attaching maps. We assume that 
all faces have been glued, and that the space $M(K)$, obtained by truncating 
the $\Delta_i$'s before gluing, is an oriented $3$-manifold with boundary. 
Let $\epsilon_i$ be a sign indicating whether or not the orientation of 
$\Delta_i$ given by the vertex ordering agrees with the orientation of $M(K)$.
\end{definition}

Note that up to removing disjoint balls (which does not effect the 
fundamental group), the manifold $M(K)$ only depends on the underlying 
topological space of $K$, and not on the choice of $3$-cycle structure. 
Also note that for any compact, oriented $3$-manifold $M$ with 
(possibly empty) boundary, the space $\widehat M$ obtained from $M$ by 
collapsing each boundary component to a point has a structure of a closed 
$3$-cycle $K$ such that $M=M(K)$. 

Let $K$ be a closed $3$-cycle, and let $M=M(K)$. Let $L$ denote the space 
obtained from the universal cover $\widetilde M$ of $M$ by collapsing each 
boundary component to a point. The $3$-cycle structure of $K$ induces a 
triangulation of $L$, and also a triangulation of $M$ by truncated simplices. 
The covering map extends to a map $L\to K$, and the action of $\pi_1(M)$ on 
$\widetilde M$ by deck transformations extends to an action on $L$, which is 
determined by fixing, once and for all, a base point in $M$ together with 
one of its lifts. 
Note that the stabilizer of each zero cell is a \emph{peripheral} subgroup 
of $\pi_1(M)$, i.e.~a subgroup induced by inclusion of a boundary component. 

\begin{definition} 
Let $H$ be a subgroup of $G$. A representation $\rho\colon\pi_1(M)\to G$ is a 
$(G,H)$-\emph{representation} if the image of each peripheral subgroup is a 
conjugate of $H$. 
\end{definition}

\begin{definition}
\lbl{decorationdefn} 
Let $\rho$ be a $(G,H)$-representation. A \emph{decoration} (on $K$) of 
$\rho$ is a $\rho$-equivariant assignment of a left $H$-coset to each vertex 
of $L$, i.e.~if $\alpha\in\pi_1(M)$ and the coset $g_eH$ is assigned to $e$, 
the coset assigned to $\alpha e$ must be $\rho(\alpha)g_eH$.
\end{definition}

Note that $g_e^{-1}\rho(\Stab(e))g_e\subset H$, where $\Stab(e)$ is the 
stabilizer of $e$. 

\begin{definition} 
Two decorations $\{g_eH\}$ and $\{g^\prime_eH\}$ of $\rho$ are 
\emph{equivalent} if the corresponding subgroups $g_e^{-1}\rho(\Stab(e))g_e$ and 
$g^{\prime-1}_e\rho(\Stab(e))g^\prime_e$ of $H$ are conjugate (in $H$).
\end{definition}
 
Note that if $\{g_eH\}$ is a decoration of $\rho$, then $\{gg_eH\}$ is a 
decoration of $g\rho g^{-1}$. Since we are only interested in representations 
up to conjugation, we consider such two decorations to be equal.

\begin{remark}
\lbl{normalizerremark} 
Letting $N_G(H)$ denote the normalizer of $H$ in $G$, and $h$ the number of 
boundary components of $M$, right multiplication induces an action of $(N_G(H)/H)^h$ on the set of equivalence classes of 
decorations. If $G=\SL(n,\C)$ and $H=N$, $N_G(H)/H$, is the group of diagonal matrices, and since two subgroups of $N$ are conjugate if and only if they are conjugate by an element in the Borel subgroup $B$ of upper diagonal matrices, this action is transitive.
\end{remark}

\begin{proposition}
\lbl{decorationeqreduction} 
Let $E$ be a flat $G$-bundle over $M$ whose holonomy representation is a 
$(G,H)$-representation $\rho$. There is a $1$-$1$ correspondence between 
decorations of $\rho$ up to equivalence, and reductions of $E_{\partial M}$ to 
an $H$-bundle over $\partial M$.
\end{proposition}

\begin{proof} 
For each boundary component $S_i$ of $M$, choose a base point in $S_i$ and 
a path to the base point of $M$. This determines a lift $e_i$ in $L$ of the 
vertex of $K$ corresponding to $S_i$, and an identification of $\pi_1(S_i)$ 
with $\Stab(e_i)\subset\pi_1(M)$. If $F$ is a reduction of $E_{\partial M}$, the 
holonomy representations $\rho_i\colon\pi_1(S_i)\to H$ of $F_{S_i}$ are 
conjugate to $\rho$, so there exist $g_i\in G$ such that 
$g_i^{-1}\rho g_i=\rho_i$. Assigning the coset $g_iH$ to $e_i$ yields a 
decoration, which up to equivalence is independent of the choice of $g_i$'s. 
On the other hand, a decoration assigns cosets $g_iH$ to $e_i$ such that 
$g_i^{-1}\rho(\Stab(e_i))g_i\subset H$. Hence, $g_i$ defines an isomorphism 
of $E_{S_i}$ with an $H$-bundle, which up to isomorphism only depends on the 
equivalence class of the decoration.
\end{proof}

\subsection{The fundamental class of a decorated representation}
\lbl{sub.fundamental}

A flat $G$-bundle over $M$ determines a classifying map $M\to BG^\delta$, 
where the $\delta$ indicates that $G$ is regarded as a discrete group. It 
thus follows from Proposition~\ref{decorationeqreduction} that a decorated 
representation $\rho\colon\pi_1(M)\to G$ determines a map
\begin{equation}
B\rho\colon(M,\partial M)\to (BG^\delta,BH^\delta).
\end{equation}
In particular, $\rho$ gives rise to a fundamental class 
\begin{equation}
[\rho]=B\rho_*([M,\partial M])\in H_3(G,H),
\end{equation}
where, by definition, $H_*(G,H)=H_*(BG^\delta,BH^\delta)$. Note that the 
fundamental class is independent of the particular $3$-cycle structure 
on $K$.

%\begin{definition} 
%\lbl{def.GHcocycle}
%A $(G,H)$-\emph{cocycle} on $M$ is a $G$-cocycle $\sigma$ on $M$ taking all 
%edges of $\partial M$ to $H$. The \emph{coboundary action} of an $H$-valued 
%zero-cochain $\tau$ takes $\sigma$ to the $(G,H)$-cocycle $\sigma\tau$ 
%taking a directed edge $vw$ to $\tau^{-1}(v)\sigma(vw)\tau(w)$.
%\end{definition}

Recall that $M$ is triangulated by truncated simplices. By 
restriction, a $(G,H)$ cocycle on $M$ determines a $(G,H)$-cocycle on each 
truncated simplex $\overline{\Delta_i}$. 
Let $\overline B_*(G,H)$ denote the chain complex generated in degree $n$ by 
$(G,H)$-cocycles on a truncated $n$-simplex. As proved in 
Zickert~\cite[Section~3]{ZickertDuke}, $\overline B_*(G,H)$ computes the 
homology groups $H_3(G,H)$. Note that a $(G,H)$-cocycle on $M$ determines (up to conjugation) a decorated $(G,H)$-representation.

\begin{proposition}
[{Zickert~\cite[Proposition~5.10]{ZickertDuke}}]
\lbl{fundclassrepprop} 
Let $\tau$ be a $(G,H)$-cocycle on $M$ representing a decorated $(G,H)$-representation $\rho$. The cycle
\begin{equation}\lbl{fundclassrep}
\sum\epsilon_i\tau_{\overline\Delta_i}\in \overline B_3(G,H),
\end{equation} 
represents the fundamental class of $\rho$.  \qed
\end{proposition}

\section{Generic decorations and Ptolemy coordinates}
\lbl{sec.genericdecorations}

%\subsection{Boundary-unipotent decorations}
%\lbl{sub.buni}

In all of the following, $G=\SL(n,\C)$, and $N$ is the subgroup of upper 
triangular matrices with $1$'s on the diagonal. A $(G,N)$-representation 
$\rho\colon\pi_1(M)\to G$ is called \emph{boundary-unipotent}. For a matrix 
$g\in G$ and a positive integer $i\leq n\in \N$, let $\{g\}_i$ be the ordered 
set consisting of the first $i$ column vectors of $g$. 

\begin{definition} 
A tuple $(g_0N,\dots,g_kN)$ of $N$-cosets is \emph{generic} if for each tuple 
$t=(t_0,\dots,t_k)$ of non-negative integers with sum $n$, we have
\begin{equation}\lbl{ctformula}
c_t:=\det\left(\bigcup_{i=0}^k\{g_i\}_{t_i}\right)\neq 0,
\end{equation}
where the determinant is viewed as a function on ordered sets of $n$ vectors 
in $\C^n$. 
The numbers $c_t$ are called \emph{Ptolemy coordinates}.
\end{definition} 

\begin{definition}
\lbl{genericdecorationdefn}
A decoration of a boundary-unipotent representation is \emph{generic} if for each simplex $\Delta$ of $L$, the tuple of cosets assigned to the vertices of $\Delta$ is generic.
\end{definition}

For a set $X$, let $C_*(X)$ be the acyclic chain complex generated in degree $k$ by tuples $(x_0,\dots,x_k)$. If $X$ is a $G$-set, the diagonal $G$-action makes $C_*(X)$ into a complex of $\Z[G]$-modules. Let $C_*^{\gen}(G/N)$ be the subcomplex of $C_*(G/N)$ generated by generic tuples. %We wish to prove $C_*^{\gen}(G/N)\otimes_{\Z[G]}\Z$ computes the relative homology. 

\begin{proposition}
\lbl{generichomology} 
The complex $C_*^{\gen}(G/N)\otimes_{\Z[G]}\Z$ computes the relative homology. If $\rho\colon\pi_1(M)\to G$ is a generically decorated representation, the fundamental class of $\rho$ is represented by 
\begin{equation}
\lbl{cosetrep}
\sum\epsilon_i(g_0^iN,g_1^iN,g_2^iN,g_3^iN)\in C^{\gen}_3(G/N),
\end{equation}
where $(g_0^iN,\dots,g_3^iN)$ are the cosets assigned to lifts $\widetilde\Delta_i$ of the $\Delta_i$'s.\qed
\end{proposition}

Proposition~\ref{generichomology} is proved in Section~\ref{proofsection}. The idea is that a generic tuple canonically determines a $(G,N)$-cocycle on a truncated simplex. Hence, $C_*^{\gen}(G/N)\otimes_{\Z[G]}\Z$ is isomorphic to a subcomplex of $\overline{B}_3(G,N)$, and the representation~\eqref{cosetrep} of the fundamental class is then an immediate consequence of~\eqref{fundclassrep}. 

%Given a closed $3$-cycle $K$, not all boundary-unipotent representations have generic decorations. However, we have:

\begin{proposition}
\lbl{barycentric} 
After a single barycentric subdivision of $K$, every decoration of a boundary-unipotent representation $\rho\colon \pi_1(M)\to G$ is equivalent to a generic one.
\end{proposition}

\begin{proof}
After a barycentric subdivision of $K$, every simplex $\Delta$ of $K$ has distinct vertices and at least three vertices of $\Delta$ are interior (link is a sphere). Fix lifts $e_i\in L$ of each interior vertex of $K$. Since the stabilizer of a lift of an interior vertex is trivial, assigning any coset $g_iH$ to $e_i$ yields an equivalent decoration. Since the $g_i$'s can be chosen arbitrarily, the result follows.
\end{proof}

\subsection{The geometry of the Ptolemy coordinates}
\lbl{sub.geometry}
We canonically identify each ordered $k$-simplex with a standard simplex
\begin{equation}
\Delta^k_n=\big\{(x_0,\dots,x_k)\in\R^{k+1}\bigm\vert 0\leq x_i\leq n,\,\sum_{i=0}^k x_i=n\big\}.
\end{equation}

Recall that a tuple $(g_0N,\dots,g_kN)$ has a Ptolemy coordinate for each tuple of $k+1$ non-negative integers summing to $n$. In other words, there is a Ptolemy coordinate for each integral point of $\Delta^k_n$. We denote the set of integral points in $\Delta^k_n$ by $\Delta^k_n(\Z)$.

\begin{definition}
\lbl{Ptolemyassignmentdefn}
A \emph{Ptolemy assignment} on $\Delta^k_n$ is an assignment of a non-zero complex number $c_t$ to each integral point $t$ of $\Delta^k_n$ such that the $c_t$'s are the Ptolemy coordinates of some tuple $(g_0N,\dots,g_kN)\in C_k^{\gen}(G/N)$. 
A Ptolemy assignment on $K$ is a Ptolemy assignment on each simplex $\Delta_i$ of $K$ such that the Ptolemy coordinates agree on identified faces.
\end{definition}

Note that a generically decorated boundary-unipotent representation
determines a Ptolemy assignment on $K$. In Section~\ref{proofsection}, we
show that every Ptolemy assignment is induced by a unique decorated
representation. 

\begin{lemma}
\lbl{numberofsubsimplices}
The number of elements in $\Delta^k_l(\Z)$ is $\binom{l+k}{k}$.
\end{lemma}
\begin{proof}
The map $(a_0,\dots,a_k)\mapsto\{a_0+1,a_0+a_1+2,\dots,a_0+\dots+a_{k-1}+k\}$ gives a bijection between $T^k(l)$ and subsets of $\{1,\dots,l+k\}$ with $k$ elements.
\end{proof}

Let $e_i$, $0\leq i\leq k$, be the $i$th standard basis vector of $\Z^{k+1}$. For each $\alpha\in \Delta^k_{n-2}(\Z)$, the points $\alpha+2e_i$ in $\Delta^k_n$ span a simplex $\Delta^k(\alpha)$, whose integral points are the points $\alpha_{ij}:=\alpha+e_i+e_j$, see Figure~\ref{PtolemyFigure}. We refer to $\Delta^k(\alpha)$ as a \emph{subsimplex} of $\Delta^k_n$. %We identify a subsimplex with $\Delta^k_2$ in the obvious way. 
By Lemma~\ref{numberofsubsimplices}, $\Delta^3_n$ has $\binom{n+3}{3}$ integral points and $\binom{n+1}{3}$ subsimplices.
\begin{figure}[htpb]
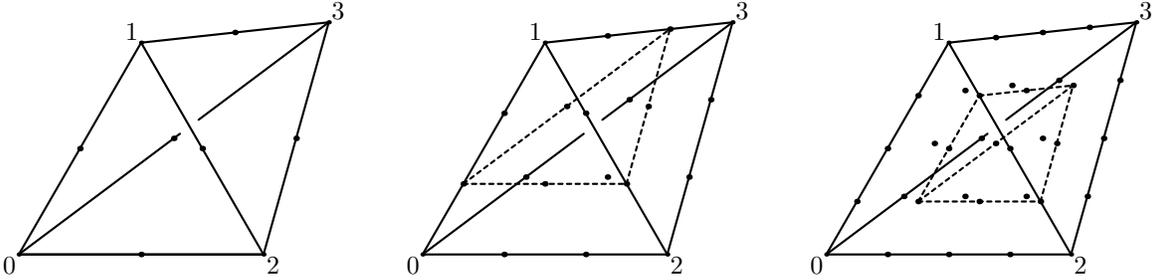

$$
\psdraw{SLrepFigures.1}{1.8in} \qquad
\psdraw{SLrepFigures.2}{1.8in} \qquad
\psdraw{SLrepFigures.3}{1.8in}
$$
\caption{The integral points on $\Delta^3_n$ for $n=2$, $3$ and $4$. The indicated subsimplices correspond to $\alpha=(0,1,0,0)$ and $\alpha=(0,1,1,0)$.}\lbl{PtolemyFigure}
\end{figure}

\begin{proposition}[{Fock-Goncharov~\cite[Lemma~10.3]{FockGoncharov}}]\label{PtolemyProposition}
The Ptolemy coordinates of a generic tuple $(g_0N\comma g_1N\comma g_2N\comma g_3N)$ satisfy the \emph{Ptolemy relations}
\begin{equation}\lbl{Ptolemyrelation}
c_{\alpha_{03}}c_{\alpha_{12}}+c_{\alpha_{01}}c_{\alpha_{23}}=c_{\alpha_{02}}c_{\alpha_{13}},\quad \alpha\in \Delta^3_{n-2}(\Z).
\end{equation}
\end{proposition}

\begin{proof}
Let $\alpha=(a_0,a_1,a_2,a_3)\in \Delta^3_{n-2}(\Z)$.
By performing row operations, we may assume that the first $n-2$ rows of the $n\times (n-2)$ matrix 
\begin{equation}\big(\{g_0\}_{a_0},\{g_1\}_{a_1},\{g_2\}_{a_2},\{g_3\}_{a_3}\big)\end{equation} 
are the standard basis vectors. Letting $x_i$ and $y_i$ denote the last two entries of $(g_i)_{a_i+1}$, the Ptolemy relation for $\alpha$ is then equivalent to the (Pl\"{u}cker) relation
\begin{equation}
\det\begin{pmatrix}x_0&x_3\\y_0&y_3\end{pmatrix}\det\begin{pmatrix}x_1&x_2\\y_1&y_2\end{pmatrix}+
\det\begin{pmatrix}x_0&x_1\\y_0&y_1\end{pmatrix}\det\begin{pmatrix}x_2&x_3\\y_2&y_3\end{pmatrix}=
\det\begin{pmatrix}x_0&x_2\\y_0&y_2\end{pmatrix}\det\begin{pmatrix}x_1&x_3\\y_1&y_3\end{pmatrix},
\end{equation}
which is easily verified.
\end{proof}

Note that the Ptolemy coordinate assigned to the $i$th vertex of $\Delta^k_n$ is $\det(\{g_i\}_n)=\det(g_i)=1$. We shall thus often ignore the vertex points. Let $\dot\Delta^k_n(\Z)$ denote the non-vertex integral points of $\Delta^k_n$.
The following is proved in Section~\ref{proofsection}.

\begin{proposition}
\lbl{uniquePtolemyassignment}
For every assignment $c\colon\dot\Delta^3_n(\Z)\to\C^*$, $t\mapsto c_t$ satisfying the Ptolemy relations~\eqref{Ptolemyrelation}, there is a unique Ptolemy assignment on $\Delta_n^3$ whose Ptolemy coordinates are $c_t$.\qed
\end{proposition}

\begin{corollary} 
The set of Ptolemy assignments on $K$ is an algebraic set $P_n(K)$ called the \emph{Ptolemy variety}. Its ideal is generated by the Ptolemy relations~\eqref{Ptolemyrelation} (together with an extra equation making sure that all Ptolemy coordinates are non-zero).\qed
\end{corollary}

\begin{remark} 
It thus follows that Definition~\ref{Ptolemyassignmentdefn} agrees with Definition~\ref{Ptolemyassignmentdefnintro} when $k=3$. When $k>3$ and $n>2$ there are further relations among the Ptolemy coordinates. We shall not need these here.
\end{remark}

%In Section~\ref{proofsection}, we prove that every Ptolemy assignment on $K$ is the Ptolemy assignment of a unique decorated representation. 

\subsection{$p\SL(n,\C)$-Ptolemy coordinates}
\lbl{sub.pSL}

When $n$ is even, a $p\SL(n,\C)$-Ptolemy assignment on $\Delta^k_n$ may be defined as in Definition~\ref{Ptolemyassignmentdefn}. Note, however, that the Ptolemy coordinates are now only defined up to a sign. Since we are mostly interested in $3$-cycles, the following definition is more useful.

\begin{definition}
\lbl{pSLPtolemydefn}
Let $\Delta=\Delta^3_n$, and let $\sigma\in Z^2(\Delta;\Z/2\Z)$ be a cellular $2$-cocycle.
A $p\SL(n,\C)$-\emph{Ptolemy assignment} on $\Delta$ with \emph{obstruction cocycle} $\sigma$ is an assignment $c\colon\dot\Delta^3_n(\Z)\to\C^*$ satisfying the $p\SL(n,\C)$-\emph{Ptolemy relations}
\begin{equation}\lbl{pSLPtolemyrelation}
\sigma_2\sigma_3c_{\alpha_{03}}c_{\alpha_{12}}+\sigma_0\sigma_3c_{\alpha_{01}}c_{\alpha_{23}}=c_{\alpha_{02}}c_{\alpha_{13}}.
\end{equation}
Here $\sigma_{i}\in\Z/2\Z=\langle\pm 1\rangle$ is the value of $\sigma$ on the face opposite the $i$th vertex of $\Delta$. 
A $p\SL(n,\C)$-Ptolemy assignment on $K$ with obstruction cocycle $\sigma\in Z^2(K;\Z/2\Z)$ is a $p\SL(n,\C)$-Ptolemy-assignment $c^i$ on each simplex $\Delta_i$ of $K$ such that the Ptolemy coordinates agree on identified faces, and such that the obstruction cocycle of $c^i$ is $\sigma_{\Delta_i}$.  
\end{definition}

Note that for each $\sigma\in Z^2(K;\Z/2\Z)$, the set of $p\SL(n,\C)$-Ptolemy-assignments on $K$ form a variety $P_n^\sigma(K)$. We show in Section~\ref{proofsection} that this variety only depends on the cohomology class of $\sigma$ in $H^2(K;\Z/2\Z)=H^2(M,\partial M;\Z/2\Z)$ and that the Ptolemy variety parametrizes generically decorated boundary-unipotent $p\SL(n,\C)$-representations whose obstruction class to lifting to a boundary-unipotent $\SL(n,\C)$-representation is $\sigma$. Note that when $\sigma$ is the trivial cocycle taking all $2$-cells to $1$, $P^\sigma(K)=P(K)$.

\subsection{Cross-ratios and flattenings}
\lbl{sub.ptolemyB}
For $x\in\C\backslash\{0\}$, let $\widetilde x=\log(x)$, where $\log$ is 
some fixed (set theoretic) section of the exponential map. 

Given a Ptolemy assignment $c$ on $\Delta^3_{n=2}$, we endow $\Delta^3_{n=2}$ with the shape of an ideal simplex with cross-ratio $z=\frac{c_{03}c_{12}}{c_{02}c_{13}}$ and a flattening 
\begin{equation}\lambda(c)=(\widetilde c_{03}+\widetilde c_{12}
-\widetilde c_{02}-\widetilde c_{13},\widetilde 
c_{01}+\widetilde c_{23}-\widetilde c_{02}-\widetilde c_{13})\in\widehat\Pre(\C).
\end{equation}
%Note that the Ptolemy relation is equivalent to the relation $1-z=\frac{c_{01}c_{23}}{c_{02}c_{13}}$.

By Propositions~\ref{PtolemyProposition} and \ref{uniquePtolemyassignment}, a Ptolemy assignment on $\Delta^3_n$ induces a Ptolemy assignment $c_\alpha$ on each subsimplex $\Delta^3(\alpha)$. We thus have a map 
\begin{equation}\lbl{defoflambda}
\lambda\colon P_n(K)\to \widehat\Pre(\C),\qquad c\mapsto \sum_i\epsilon_i\sum_{\alpha\in \Delta^3_{n-2}(\Z)}\lambda(c_\alpha^i).
\end{equation}
Similarly, we have a map $P^\sigma_n(K)\to\widehat\Pre(\C)_{\PSL}$ defined by the same formula. We next prove that these maps have image in the respective extended Bloch groups. 

\begin{remark}
The shapes associated to a Ptolemy assignment satisfy equations resembling Thurston's gluing equations. This is studied in Garoufalidis-Goerner-Zickert~\cite{GaroufalidisGoernerZickert}.
\end{remark}

\section{A chain complex of Ptolemy assignments}
\lbl{sec.ptolemy}

Let $Pt_k^n$ be the free abelian group on Ptolemy assignments on $\Delta^k_n$. 
The usual boundary map induces a boundary map $Pt_k^n\to Pt_{k-1}^n$ and the 
natural map $C_*^{\gen}(G/N)\to Pt_*^n$ taking a tuple $(g_0N,\dots,g_kN)$ to 
its Ptolemy assignment is a chain map. The result below is proved in 
Section~\ref{proofsection}.

\begin{proposition}
\lbl{tupleuniqueuptoGaction}
A generic tuple is determined up to the diagonal $G$-action by its 
Ptolemy coordinates.\qed
\end{proposition} 

\begin{corollary}
\lbl{Ptolemycomputeshomology} 
The natural map induces an isomorphism 
\begin{equation}\lbl{CgenPt}
C_*^{\gen}(G/N)\otimes_{\Z[G]}\Z\cong Pt_*^n.
\end{equation}
In particular, $H_*(G,N)=H_*(Pt^n_*)$.\qed
\end{corollary}

\begin{lemma}
\lbl{inducedPtolemyassignment}
Let $c\in Pt_k^n$ be a Ptolemy assignment, and let $\alpha\in\Delta^k_{n-2}(\Z)$. The 
Ptolemy coordinates $c_{\alpha_{ij}}$, $i\neq j$ are the Ptolemy coordinates 
of a unique Ptolemy assignment $c_\alpha$ on the subsimplex $\Delta^k(\alpha)$.
\end{lemma}

\begin{proof} 
For $1\leq k\leq 3$, this follows from Proposition~\ref{uniquePtolemyassignment}. 
For $k>3$, the result follows by induction, using the fact that $5$ Ptolemy 
coordinates on $\Delta^3_2$ determines the last.
\end{proof}

A Ptolemy assignment $c$ on $\Delta^k_n$ thus induces a Ptolemy assignment $c_\alpha$ on each 
subsimplex. We thus have maps
\begin{equation}
\lbl{Jkn}
J_k^n\colon Pt_k^n\to Pt_k^2, \quad c\mapsto\sum_{\alpha\in \Delta^k_{n-2}(\Z)}c_{\alpha}.
\end{equation}

For a Ptolemy assignment $c\in Pt_k^n$ let $c_{\underline i}\in Pt_{k-1}^n$ be the 
induced Ptolemy assignment on the $i$th face of $\Delta^k_n$, i.e.~we have 
$\partial(c)=\sum_{i=0}^k (-1)^ic_{\underline i}$. Note that
\begin{equation}\lbl{helpfuleq1}
(c_{\underline i})_{(a_0,\dots,a_{k-1})}=c_{(a_0,\dots,a_{i-1},0,a_{i},\dots a_{k-1})_{\underline i}}
\in Pt_{k-1}^2.
\end{equation}

For $\beta\in \Delta^k_{n-3}(\Z)$, let $c_{\beta^i}=c_{(\beta+e_i)_{\underline i}}\in Pt_{k-1}^2$, 
and define $\partial_{\beta}(c)\in Pt_{k-1}^2$ by 
\begin{equation}
\partial_{\beta}(c)=\sum_{i=0}^k (-1)^i c_{\beta^i}\in Pt_{k-1}^2.
\end{equation}  
The geometry is explained in Figure~\ref{cbetafigure}.

\begin{figure}[htpb]
$$
\psdraw{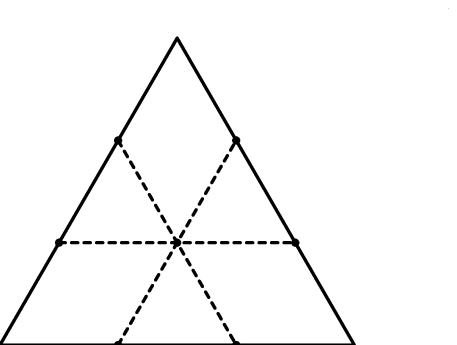}{1.7in} \qquad
\psdraw{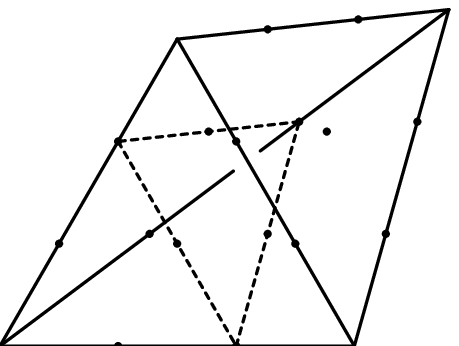}{1.6in}
$$
\caption{The dotted lines in the left figure indicate $c_{\beta^0}$, 
$c_{\beta^1}$ and $c_{\beta^2}$ for $k=2$. The triangle in the right figure 
indicates $c_{\beta^0}$ for $k=3$. Here, $n=3$ and $\beta=0$.}
\lbl{cbetafigure}
\end{figure}

\begin{proposition}
\lbl{DeltaJ} 
Let $c\in Pt_k^n$. We have
\begin{equation}\partial (J_k^n(c))-J^n_{k-1}(\partial(c))
=\sum_{\beta\in \Delta^k_{n-3}(\Z)}\partial_\beta(c)\in Pt_{k-1}^2.
\end{equation}
\end{proposition}

\begin{proof}
By~\eqref{helpfuleq1}, we have
\begin{equation}
\begin{aligned}
\partial (J_k^n(c))-J_{k-1}^n(\partial(c))=&
\sum_{i=0}^k(-1)^i\sum_{\alpha\in \Delta^k_{n-2}(\Z)}c_{\alpha_{\underline i}}-
\sum_{i=0}^k(-1)^i\sum_{\substack{\alpha\in \Delta^k_{n-2}(\Z)\\a_i=0}}c_{\alpha_{\underline i}}
\\=&\sum_{i=0}^k(-1)^i\sum_{\substack{\alpha\in \Delta^k_{n-2}(\Z)\\a_i>0}}c_{\alpha_{\underline i}}
\\=&\sum_{\beta\in \Delta^k_{n-3}(\Z)}\sum_{i=0}^k(-1)^ic_{(\beta+e_i)_{\underline i}}
\\=&\sum_{\beta\in \Delta^k_{n-3}(\Z)}\partial_\beta(c)
\end{aligned}
\end{equation}
as desired
\end{proof}

\subsection{The map to the extended Bloch group}
\lbl{sub.map2ebloch}

We wish to define a map 
\begin{equation}
\lambda\colon H_3(\SL(n,\C),N)\to\widehat\B(\C).
\end{equation}

%For $c\in \C\setminus\{0\}$ let $\widetilde c\in \C$ denote $\log(c)$, where $\log$ is some fixed (set theoretic) section of the exponential map. Let $c\in Pt_3^2$ be a Ptolemy assignment, and let $c_{ij}$ denote the Ptolemy coordinate assigned to the edge $ij$. By \eqref{Ptolemyrelation}, we have
%\begin{equation}\lbl{idealassignment}
%c_{03}c_{12}+c_{01}c_{23}=c_{02}c_{13},
%\end{equation}
%and it follows that the pair 
%\begin{equation*}
%(\widetilde c_{03}+\widetilde c_{12}-\widetilde c_{02}-\widetilde c_{13},\widetilde c_{01}+\widetilde c_{23}-\widetilde c_{02}-\widetilde c_{13})
%\end{equation*} 
%is a flattening with cross-ratio $\frac{c_{03}c_{12}}{c_{02}c_{13}}$.
Letting $\widetilde x$ denote a logarithm of $x$, we consider the maps 
\begin{gather}
\lambda\colon Pt_3^2\to\Z[\widehat\C],\quad c\mapsto (\widetilde c_{03}+\widetilde c_{12}-\widetilde c_{02}-\widetilde c_{13},\widetilde c_{01}+\widetilde c_{23}-\widetilde c_{02}-\widetilde c_{13})\\
\mu\colon Pt_2^2\to\wedge^2(\C),\quad c\mapsto -\widetilde c_{01}\wedge\widetilde c_{02}+\widetilde c_{01}\wedge\widetilde c_{12}-\widetilde c_{02}\wedge\widetilde c_{12
}+\widetilde c_{02}\wedge\widetilde c_{02}.
\end{gather}

\begin{remark} 
The term $\widetilde c_{02}\wedge\widetilde c_{02}$ vanishes in $\wedge^2(\C)$, 
but over general fields this term is needed. General fields are discussed in 
Section~\ref{otherfields}.
\end{remark}

\begin{lemma}
[{Zickert~\cite[Lemma~6.9]{ZickertAlgK}}]
\lbl{I3Bhat}
Let $\Z[\widehat\FT]$ be the subgroup of $\Z[\widehat\C]$ generated by the 
lifted five term relations. There is a commutative diagram
\begin{equation}
\cxymatrix{{Pt_4^2\ar[r]^-\partial\ar[d]^-{\lambda\circ\partial}&
Pt_3^2\ar[r]^-\partial\ar[d]^-{\lambda}&Pt_2^2\ar[d]^-\mu
\\{\Z[\widehat{\FT}]}\ar@{^{(}->}[r]&{\Z[\widehat \C]}\ar[r]^-{\widehat\nu}&
\wedge^2(\C).}}
\end{equation}
\qed
\end{lemma}

It follows that $\lambda$ induces a map 
$\lambda\colon H_3(\SL(2,\C),N)\to\widehat\B(\C)$. This map equals the map 
defined in Zickert~\cite[Section~7]{ZickertDuke}. The fact that $\lambda$ 
is independent of the choice of logarithm is proved in 
Zickert~\cite[Remark~6.11]{ZickertDuke}, and also follows from 
Proposition~\ref{independenceoflog} below. 

\begin{lemma}
\lbl{beta4}
For each $c\in Pt_4^n$ and each $\beta\in \Delta^4_{n-3}(\Z)$, we have 
\begin{equation}\lambda(\partial_\beta(c))=0\in\widehat\Pre(\C).
\end{equation}
\end{lemma}

\begin{proof}
Let $(e_i,f_i)=\lambda(c_{\beta^i})$ be the flattening associated to $c_{\beta^i}$. 
We prove that the flattenings satisfy the five term relation by proving that 
the equations~\eqref{fiveeq} are satisfied. We have
\begin{equation}
\begin{gathered}
e_0=\widetilde c_{\beta+(1,1,0,0,1)}+\widetilde c_{\beta+(1,0,1,1,0)}-\widetilde c_{\beta+(1,1,0,1,0)}-\widetilde c_{\beta+(1,0,1,0,1)}\\
e_1=\widetilde c_{\beta+(1,1,0,0,1)}+\widetilde c_{\beta+(0,1,1,1,0)}-\widetilde c_{\beta+(1,1,0,1,0)}-\widetilde c_{\beta+(0,1,1,0,1)}\\
e_2=\widetilde c_{\beta+(1,0,1,0,1)}+\widetilde c_{\beta+(0,1,1,1,0)}-\widetilde c_{\beta+(1,0,1,1,0)}-\widetilde c_{\beta+(0,1,1,0,1)}\\
%e_3=\widetilde c_{\beta+(1,0,0,1,1)}+\widetilde c_{\beta+(0,1,1,1,0)}-\widetilde c_{\beta+(1,0,1,1,0)}-\widetilde c_{\beta+(0,1,0,1,1)}\\
%e_4=\widetilde c_{\beta+(1,0,0,1,1)}+\widetilde c_{\beta+(0,1,1,0,1)}-\widetilde c_{\beta+(1,0,1,0,1)}-\widetilde c_{\beta+(0,1,0,1,1)}\\
%f_0=\widetilde c_{\beta+(1,1,1,0,0)}+\widetilde c_{\beta+(1,0,0,1,1)}-\widetilde c_{\beta+(1,1,0,1,0)}-\widetilde c_{\beta+(1,0,1,0,1)}\\
%f_1=\widetilde c_{\beta+(1,1,1,0,0)}+\widetilde c_{\beta+(0,1,0,1,1)}-\widetilde c_{\beta+(1,1,0,1,0)}-\widetilde c_{\beta+(0,1,1,0,1)}\\
%f_2=\widetilde c_{\beta+(1,1,1,0,0)}+\widetilde c_{\beta+(0,0,1,1,1)}-\widetilde c_{\beta+(1,0,1,1,0)}-\widetilde c_{\beta+(0,1,1,0,1)}\\
%f_3=\widetilde c_{\beta+(1,1,0,1,0)}+\widetilde c_{\beta+(0,0,1,1,1)}-\widetilde c_{\beta+(1,0,1,1,0)}-\widetilde c_{\beta+(0,1,0,1,1)}\\
%f_4=\widetilde c_{\beta+(1,1,0,0,1)}+\widetilde c_{\beta+(0,0,1,1,1)}-\widetilde c_{\beta+(1,0,1,0,1)}-\widetilde c_{\beta+(0,1,0,1,1)}.
\end{gathered}
\end{equation}
and it follows that $e_2=e_1-e_0$ as desired. The other $4$ equations are 
proved similarly.
\end{proof}

\begin{lemma}
\lbl{beta3} 
For each $c\in Pt_3^n$ and each $\beta\in \Delta^3_{n-3}(\Z)$, 
$\mu(\partial_\beta(c))=0\in\wedge^2(\C)$.
\end{lemma}

\begin{proof}
We have
\begin{multline}
\mu(c_{\beta^0})=-\widetilde c_{\beta+(1,1,1,0)}\wedge \widetilde c_{\beta+(1,1,0,1)}
+\widetilde c_{\beta+(1,1,1,0)}\wedge \widetilde c_{\beta+(1,0,1,1)}
\\-\widetilde c_{\beta+(1,1,0,1)}\wedge \widetilde c_{\beta+(1,0,1,1)}
+\widetilde c_{\beta+(1,1,0,1)}\wedge \widetilde c_{\beta+(1,1,0,1)}.
\end{multline}
Using this together with the similar formulas for $\mu(c_{\beta^i})$, 
we obtain that 
\[\sum(-1)^i\mu(c_{\beta^i})=0\in\wedge^2(\C),\]
proving the result.
\end{proof}

\begin{corollary} 
The map $\lambda\circ J^n_3$ induces a map
\begin{equation}
\lambda\colon H_3(\SL(n,\C),N)\to\widehat\B(\C).
\end{equation}
\end{corollary}

\begin{proof} 
Using Proposition~\ref{DeltaJ}, this follows from Lemma~\ref{beta4} and 
Lemma~\ref{beta3}.
\end{proof}

\begin{remark}
For $n=3$, this map agrees with the map considered in Zickert~\cite{ZickertAlgK}.
\end{remark}

\begin{definition}
The \emph{extended Bloch group element} of a decorated $(G,N)$-representation $\rho$ is defined by $\lambda([\rho])$, where $[\rho]\in H_3(\SL(n,\C),N)$ is the fundamental class of $\rho$.
\end{definition}
Note that if the decoration of $\rho$ is generic, and $c$ is the corresponding Ptolemy assignment, the extended Bloch group element is given by $\lambda(c)$, where $\lambda\colon P_n(K)\to \widehat\Pre(\C)$ is given by~\eqref{defoflambda}. 
%The following is an immediate consequence of the definition of $\lambda$.
\begin{proposition}\label{lambdaimageinBhat} The map $\lambda\colon P_n(K)\to \widehat\Pre(\C)$ has image in $\widehat\B(\C)$. 
\end{proposition}
\begin{proof}If $c\in P_n(K)$ is a Ptolemy assignment on $K$, we have a cycle $\alpha=\sum_i\epsilon_ic^i\in Pt^n_3$, and one easily checks that $\lambda(c)$ as defined in~\eqref{defoflambda} equals $\lambda([\alpha])$. This proves the result.
\end{proof}

\subsection{Stabilization}
\lbl{stabilizationsection}

We now prove that the map $\lambda\colon H_3(\SL(n,\C),N)\to\widehat\B(\C)$ 
respects stabilization. We regard $\SL(n-1,\C)$ as a subgroup of $\SL(n,\C)$ 
via the standard inclusion adding a $1$ as the upper left entry. 

Let $\pi\colon M(n,\C)\to M(n-1,\C)$ be the map sending a matrix to the 
submatrix obtained by removing the first row and last column. The subgroup 
$D_k(\SL(n,\C)/N)$ of $C^{\gen}_k(\SL(n,\C)/N)$ generated by tuples 
$(g_0N,\dots,g_kN)$ such that the upper left entry of each $g_i$ is $1$ and such that
\begin{equation}
(\pi(g_0)N,\dots,\pi(g_k)N)\in C^{\gen}_k(\SL(n-1,\C)/N)
\end{equation}
form an $\SL(n-1,\C)$-complex. Consider the $\SL(n-1,\C)$-invariant chain maps
\begin{gather} \pi\colon D_*(\SL(n,\C)/N)\to Pt_*^{n-1}\\
i\colon D_*(\SL(n,\C)/N)\to Pt_*^n,
\end{gather}
where the first map is induced by $\pi$ and the second is induced by the 
inclusion $D_*(\SL(n,\C)/N)\to C_*^{\gen}(\SL(n,\C)/N$. Let 
$D_k=D_k(\SL(n,\C)/N)\otimes_{\Z[\SL(n-1,\C)]}\Z$.

\begin{lemma}
\lbl{stabilizationlemma} 
The maps $\lambda\circ\pi$ and $\lambda\circ i$ from $D_3$ to 
$\widehat\Pre(\C)$ agree on cycles.
\end{lemma}

\begin{proof} 
Let $c\in D_k$ be induced by a tuple $(g_0N,\dots,g_kN)\in D_k(\SL(n,\C)/N)$, 
and let $c^I$ be the collection of Ptolemy coordinates associated to 
$(N,g_0N,\dots,g_kN)$. Since some Ptolemy coordinates may be zero, $c^I$ 
is not necessarily a Ptolemy assignment. Note, however, that $c^I_\alpha$ \emph{is} a Ptolemy assignment for each $(a_0,\dots,a_{k+1})\in \Delta^{k+1}_{n-2}(\Z)$ with $a_0=0$. Note also that $c^I_\alpha\in Pt_{k+1}^2$ only depends on $c$.
Hence, there is a map 
\begin{equation}
P\colon D_k\to Pt_{k+1}^2,\quad c\mapsto \sum_{\substack{\alpha\in \Delta^{k+1}_{n-2}(\Z)\\a_0=0}}c^I_{\alpha}.
\end{equation}
We wish to prove the following:
\begin{equation}\lbl{theresult}
\partial P(c)+P\partial(c)=J_k^n(i(c))-J_k^{n-1}(\pi(c))+\sum_{\substack{\beta\in \Delta^{k+1}_{n-3}(\Z)\\b_0=0}}\partial_\beta(c^I)\in Pt_{k+1}^2.
\end{equation}
Given this, the result follows immediately from Lemma~\ref{beta4}. 

One easily verifies that
\begin{gather}
c^I_{(\underline 1,b_0,\dots, b_k)}=\pi(c)_{(b_0,\dots,b_k)}\in Pt_k^{n-1},\quad (b_0,\dots,b_k)\in \Delta^k_{n-3}(\Z).\\
c^I_{(\underline 0,a_0,\dots,a_k)}=i(c)_{(a_0,\dots,a_k)}, \quad (a_0,\dots,a_k)\in \Delta^k_{n-2}(\Z).
\end{gather}
Using this, one has
\begin{equation}
\begin{aligned}
\partial P(c)+P\partial(c)&=\sum_{\alpha\in \Delta^k_{n-2}(\Z)}i(c)_\alpha\,+\,\sum_{i=1}^{k+1}(-1)^{i}\!\sum_{\substack{\alpha\in \Delta^{k+1}_{n-2}(\Z)\\a_0=0}}c^I_{\alpha_{\underline{i}}}\,+\,\sum_{i=0}^k(-1)^i\!\!\!\!\!\!\sum_{\substack{\alpha\in \Delta^{k+1}_{n-2}(\Z)\\a_0=0,a_{i+1}=0}}\!\!c^I_{\alpha_{\underline{i+1}}}\\
&=\sum_{\alpha\in \Delta^k_{n-2}(\Z)}i(c)_\alpha\,+\,\sum_{i=1}^{k+1}(-1)^{i}\!\!\!\sum_{\substack{\alpha\in \Delta^{k+1}_{n-2}(\Z)\\a_0=0,a_{i}>0}}c^I_{\alpha_{\underline{i}}}\\
&=\sum_{\alpha\in \Delta^k_{n-2}(\Z)}i(c)_\alpha+\,\,\sum_{\substack{\beta\in \Delta^{k+1}_{n-3}(\Z)\\b_0=0}}\,\sum_{i=1}^{k+1}(-1)^{i}c^I_{\beta^{i}}\\
&=\sum_{\alpha\in \Delta^k_{n-2}(\Z)}i(c)_\alpha-\sum_{\substack{\beta\in \Delta^{k+1}_{n-3}(\Z)\\b_0=0}}c^I_{\beta^0}+\sum_{\substack{\beta\in \Delta^{k+1}_{n-3}(\Z)\\b_0=0}}\partial_\beta(c^I)\\
&=J_k^n(i(c))-J_k^{n-1}(\pi(c))+\sum_{\substack{\beta\in \Delta^{k+1}_{n-3}(\Z)\\b_0=0}}\partial_\beta(c^I).
\end{aligned}
\end{equation}
This proves \eqref{theresult}, hence the result.
\end{proof}

\begin{proposition} 
The map $\lambda\colon H_3(\SL(n,\C),N)\to\widehat\B(\C)$ respects 
stabilization.
\end{proposition}

\begin{proof}
First note that $\pi$ induces an isomorphism $D^0(\SL(n,\C)/N)\cong C^0(\SL(n-1)/N)$. Using a standard cone argument, one easily checks that $D_*(\SL(n,\C)/N)$ 
is a free $\SL(n-1,\C)$-resolution of $\Ker (D^0(\SL(n,\C)/N)\to \Z)$. Hence, $D_*$ computes $H_*(\SL(n-1,\C),N)$, and the result follows from Lemma~\ref{stabilizationlemma}.
\end{proof}

\subsection{$p\SL(n,\C)$-Ptolemy assignments}
\lbl{sub.pSLassignments}

When $n$ is even, define $pPt_*^n$ to be the complex of Ptolemy coordinates 
of generic tuples in $p\SL(n,\C)/N$. The Ptolemy coordinates are defined as 
in~\eqref{ctformula}, and take values in $\C^*\big/\langle\pm 1\rangle$. As 
in \eqref{CgenPt}, we have an isomorphism
$C_*^{\gen}(p\SL(n,\C)/N)_{p\SL(n,\C)}\cong pPt_*^n$.
For $c\in C^*/\langle\pm 1\rangle$ let $\widetilde c\in \C$ be the image of 
some fixed set theoretic section of 
$\C\xrightarrow{\exp}\C^*\rightarrow\C^*/\langle\pm 1\rangle$, 
e.g.~$\frac{1}{2}\log(x^2)$ (the particular choice is inessential).
The map 
\begin{equation}
\lambda\colon pPt_3^2\to\Z[\widehat{\C}_{\odd}], \quad 
c\mapsto (\widetilde c_{03}+\widetilde c_{12}
-\widetilde c_{02}-\widetilde c_{13},\widetilde c_{01}
+\widetilde c_{23}-\widetilde c_{02}-\widetilde c_{13})
\end{equation} 
induces a map $H_3(\PSL(2,\C),N)\to\widehat\B(\C)_{\PSL}$, which agrees 
with the map constructed in Zickert~\cite[Section~3]{ZickertDuke}. By precomposing $\lambda$ with the map $pJ_3^n\colon pPt_3^n\to pPt_3^2$ defined as in \eqref{Jkn} we obtain a map
\begin{equation}
\lambda\colon H_3(p\SL(n,\C),N)\to\widehat\B(\C)_{\PSL},
\end{equation}
which commutes with stabilization. This proves that a decorated boundary-unipotent representation
in $p\SL(n,\C)$ determines an element in $\widehat\B(\C)_{\PSL}$.
The proofs of the above assertions are word by word identical to their 
$\SL(n,\C)$-analogs.

\section{Independence of the decoration}
\lbl{decorationsection}

We now show that the extended Bloch group element of a decorated 
representation is independent of the decoration. We first prove that 
we can choose logarithms of the Ptolemy coordinates independently, 
without affecting the extended Bloch group element.

\begin{definition} 
Let $c\colon\dot\Delta^k_n(\Z)\to \C^*$ be a Ptolemy assignment. A \emph{lift} 
of $c$ is an assignment $\widetilde c\colon\dot\Delta^k_n(\Z)\to\C$ such 
that $\exp(\widetilde c)=c$.
\end{definition}

For any lift $\widetilde c$ of a Ptolemy assignment $c$ on $\Delta^3_2$, we have a flattening
\begin{equation}\lbl{cwidetilde}
\lambda(\widetilde c)=(\widetilde c_{03}+\widetilde c_{12}
-\widetilde c_{02}-\widetilde c_{13},\widetilde c_{01}+\widetilde c_{23}
-\widetilde c_{02}-\widetilde c_{13})\in\widehat\C.
\end{equation}

\begin{definition} 
The \emph{log-parameters} of a flattening 
$(e,f)\in\widehat\C$ are defined by
\begin{equation}\lbl{logparameters}
w_{ij}=\begin{cases}e&\text{ if }ij=01\text{ or }ij=23\\
-f&\text{ if }ij=12\text{ or }ij=03\\-e+f&\text{ if }ij=02\text{ or }ij=13.
\end{cases}
\end{equation}
\end{definition}

\begin{lemma}
\lbl{addtwopi} 
Let $\widetilde c\colon\dot\Delta^3_2(\Z)\to\C$ be a lifted Ptolemy assignment, and let 
$w_{ij}$ be the log-parameters of $\lambda(\widetilde c)$. Fix 
$i<j\in\{0,\dots,3\}$ and let $\widetilde c^\prime$ be the lifted Ptolemy 
assignment obtained from  $\widetilde c$ by adding $2\pi\sqrt{-1}$ to 
$\widetilde c_{ij}$. Then
\begin{equation}
\lambda(\widetilde c^\prime)-\lambda(\widetilde c)=\chi(w_{ij}
+2\pi\sqrt{-1}\delta_{ij}),
\end{equation}  
where $\delta_{ij}$ is $1$ if $ij=02$ or $13$ and $0$ otherwise.
\end{lemma}

\begin{proof} 
Denote the flattening $\lambda(\widetilde c)$ by $(e,f)$. If $ij=03$ or 
$12$, it follows from \eqref{cwidetilde} that $\lambda(\widetilde c^\prime)
=(e+2\pi\sqrt{-1},f)$. Similarly, $\lambda(\widetilde c^\prime)
=(e,f+2\pi\sqrt{-1})$ if $ij=01$ or $23$, and $\lambda(\widetilde c^\prime)
=(e-2\pi\sqrt{-1},f-2\pi\sqrt{-1})$ if $ij=02$ or $13$. By 
Lemma~\ref{chipqlemma},
\begin{equation}
\begin{aligned}
(e+2\pi\sqrt{-1},f)-(e,f)=&\chi(-f)\\ (e,f+2\pi\sqrt{-1})-(e,f)=&\chi(e)\\
(e-2\pi\sqrt{-1},f-2\pi\sqrt{-1})=&\chi(-e+f+2\pi\sqrt{-1}).
\end{aligned}
\end{equation} 
This proves the result.
\end{proof}

Let $\widetilde c$ be a lift of a Ptolemy assignment $c$. For each 
$\alpha\in \Delta^3_{n-2}(\Z)$, $\widetilde c$ induces a lift $\widetilde c_\alpha$ 
of $c_\alpha$. Consider the element 
\begin{equation}\lbl{tauinPhat}
\tau=\sum_{\alpha\in \Delta^k_{n-2}(\Z)}\lambda(\widetilde c_\alpha)\in\widehat\Pre(\C).
\end{equation}
Fix a point $t_0\in\dot\Delta^k_n(\Z)$. We wish to understand the effect 
on $\tau$ of adding $2\pi\sqrt{-1}$ to $\widetilde c_{t_0}$. This changes 
$\tau$ into an element $\tau^{\prime}\in\widehat\Pre(\C)$. Let $w_{ij}(\alpha)$ 
denote the log-parameters of $\lambda(\widetilde c_\alpha)$. Note that $t_0$ 
either lies on an edge, on a face, or in the interior of $\Delta^3_n$. 

\begin{lemma}
\lbl{edgepoint} 
Suppose $t_0$ is on the edge $ij$ of $\Delta^3_n$. Then 
\begin{equation}\tau^{\prime}-\tau=\chi(w_{ij}(\alpha)+2\pi\sqrt{-1}\delta_{ij}),
\end{equation}
where $\alpha=t-e_i-e_j$, (i.e.~$\alpha$ is such that $t_0$ is an edge point 
of $\Delta^3(\alpha)$).
\end{lemma}

\begin{proof}
This follows immediately from Lemma~\ref{addtwopi}.
\end{proof}

\begin{lemma}
\lbl{facepoint} 
Suppose $t_0$ is on a face opposite vertex $i$. Then 
$\tau^\prime-\tau=(-1)^i\chi(\kappa+2\pi\sqrt{-1})$, where $\kappa$ is given by
\begin{equation}
\kappa=\widetilde c_{\eta_i(0,-1,1)}-\widetilde c_{\eta_i(0,1,-1)}-\big(\widetilde 
c_{\eta_i(-1,0,1)}-\widetilde c_{\eta_i(1,0,-1)}\big)+\widetilde c_{\eta_i(-1,1,0)}
-\widetilde c_{\eta_i(1,-1,0)},
\end{equation}
where $\eta_i$ inserts a zero as the $i$th vertex.
\end{lemma}

\begin{proof}
For simplicity assume $i=0$. The other cases are proved similarly.
There are exactly three $\alpha$'s for which $t_0$ is an edge point of 
$\Delta^3(\alpha)$. These are
\begin{equation}
\lbl{alphais}
\alpha_0=t_0-(0,0,1,1),\quad\alpha_1=t_0-(0,1,0,1),\quad\alpha_2=t_0-(0,1,1,0).
\end{equation}  
Note that $\widetilde c_t=(\widetilde c_{\alpha_0})_{23}
=(\widetilde c_{\alpha_1})_{13}=(\widetilde c_{\alpha_2})_{12}$.
Since adding $2\pi\sqrt{-1}$ to $\widetilde c_{t_0}$ leaves $\widetilde c_\alpha$ unchanged unless $\alpha\in\{\alpha_0,\alpha_1,\alpha_2\}$, Lemma~\ref{addtwopi} implies that
\begin{equation}
\tau^\prime-\tau=\chi(w_{23}(\alpha_0))+\chi(w_{13}(\alpha_1)+2\pi\sqrt{-1})+\chi(w_{12}(\alpha_2)).
\end{equation}
One easily checks that
\begin{equation}
\lbl{firstthree}
\begin{gathered}
w_{23}(\alpha_0)=\widetilde c_{(1,0,-1,0)}+\widetilde c_{(0,1,0,-1)}-\widetilde c_{(1,0,0,-1)}-\widetilde c_{(0,1,-1,0)}\\
w_{13}(\alpha_1)=\widetilde c_{(1,0,0,-1)}+\widetilde c_{(0,-1,1,0)}-\widetilde c_{(1,-1,0,0)}-\widetilde c_{(0,0,1,-1)}\\
w_{12}(\alpha_2)=\widetilde c_{(1,-1,0,0)}+\widetilde c_{(0,0,-1,1)}-\widetilde c_{(1,0,-1,0)}-\widetilde c_{(0,-1,0,1)},
\end{gathered}
\end{equation}
from which the result follows.
\end{proof}

\begin{lemma}
\lbl{interiorpoint} 
If $t_0$ is an interior point, $\tau^\prime=\tau$.
\end{lemma}

\begin{proof}
If $t_0$ is an interior point, there are six $\alpha$'s for which $t_0$ is an edge point of $\Delta^3(\alpha)$.
These are $\alpha_0$, $\alpha_1$ and $\alpha_2$ as defined in \eqref{alphais} as well as 
\begin{equation}
\alpha_3=t_0-(1,1,0,0),\quad \alpha_4=t_0-(1,0,1,0),\quad \alpha_5=t_0-(1,0,0,1).
\end{equation}
Again, by Lemma~\ref{addtwopi}
\begin{multline}\lbl{intpoint}
\tau^\prime-\tau=\chi(w_{23}(\alpha_0))+\chi(w_{13}(\alpha_1)+2\pi\sqrt{-1})+\chi(w_{12}(\alpha_2))+\\\chi(w_{01}(\alpha_3))+\chi(w_{02}(\alpha_4)+2\pi\sqrt{-1})+\chi(w_{03}(\alpha_5)).
\end{multline}
Using \eqref{firstthree} as well as
\begin{equation}
\begin{gathered}
w_{01}(\alpha_3)=\widetilde c_{(0,-1,0,1)}+\widetilde c_{(-1,0,1,0)}-\widetilde c_{(0,-1,1,0)}-\widetilde c_{(-1,0,0,1)}\\
w_{02}(\alpha_4)=\widetilde c_{(0,1,-1,0)}+\widetilde c_{(-1,0,0,1)}-\widetilde c_{(0,0,-1,1)}-\widetilde c_{(-1,1,0,0)}\\
w_{03}(\alpha_5)=\widetilde c_{(0,0,1,-1)}+\widetilde c_{(-1,1,0,0)}-\widetilde c_{(0,1,0,-1)}-\widetilde c_{(-1,0,1,0)}
\end{gathered}
\end{equation}
we see that all terms in~\eqref{intpoint} cancel out. Hence, $\tau^\prime=\tau$.
\end{proof}

\begin{proposition}
\lbl{independenceoflog} 
Let $c$ be a Ptolemy assignment on $K$. For any lift $\widetilde c$ of $c$, the element
\begin{equation}\lbl{cwidetildeelement}
\lambda(\widetilde c)=\sum_i\sum_{\alpha\in \Delta^k_{n-2}(\Z)}\epsilon_i\lambda(\widetilde c^i_\alpha)\in\widehat\Pre(\C)
\end{equation}
is independent of the choice of lift. In particular, if $c$ is the Ptolemy assignment of a decorated representation $\rho$, $\lambda(\widetilde c)$ is the extended Bloch group element of $\rho$.
\end{proposition}

\begin{proof} 
Let $\widetilde c$ and $\widetilde c^\prime$ be lifts of $c$. Let 
$t_0\in\dot\Delta^3_n(\Z)$. We wish to prove that 
$\lambda(\widetilde c)=\lambda(\widetilde c^\prime)$. It is enough to prove 
this when $\widetilde c^\prime$ is obtained from $\widetilde c$ by adding 
$2\pi\sqrt{-1}$ to $\widetilde c_t$. If $t_0$ is an interior point, the 
result follows immediately from Lemma~\ref{interiorpoint}.
If $t_0$ is a face point, $t_0$ lies in exactly two simplices of $K$, and 
it follows from Lemma~\ref{facepoint} that the two contributions to the 
change in $\lambda(\widetilde c)$ appear with opposite signs (by~\eqref{bigdiagram}, 
$2\chi(2\pi\sqrt{-1})=0$). Suppose $t_0$ is an edge point. Let $C$ be the 
$3$-cycle obtained by gluing together all the $\Delta^3(\alpha)$'s having 
$t_0$ as an edge point, using the face pairings induced from $K$. Let $e$ 
be the (interior) $1$-cell of $C$ containing $t_0$. The argument in 
Zickert~\cite[Theorem~6.5]{ZickertDuke} shows that the total log-parameter 
around $e$ is zero. It thus follows from Lemma~\ref{edgepoint} that adding 
$2\pi\sqrt{-1}$ to $\widetilde c_{t_0}$ changes $\lambda(\widetilde c)$ by 
$2$-torsion which is trivial if and only if the number $n$ of simplices in 
$C$ for which $t$ is a $02$ edge or a $13$ edge is even. Consider a curve 
$\lambda$ in $C$ encircling $e$. The vertex ordering induces an orientation 
on each face of each simplex of $C$, such that when $\lambda$ passes through 
two faces of a simplex in $C$, the two orientations agree unless $e$ is a 
$02$ edge or a $13$ edge. Since $M$ is orientable, it follows that $n$ is even.
The second statement follows by letting $\widetilde c=\log c$.
\end{proof}

\begin{proposition}
\lbl{independenceofdec} 
The extended Bloch group element of a decorated 
boundary-unipotent representation is independent of the decoration.
\end{proposition}

\begin{proof}
By performing a barycentric subdivision if necessary, we may assume that any 
decoration is generic.
Fix a lift $\widetilde\Delta_i$ of each simplex $\Delta_i$ of $K$. A 
decoration $D$ assigns a coset $g_j^iN$ to each vertex $j$ of 
$\widetilde\Delta_i$. Suppose we have another decoration $D^\prime$ of $\rho$ 
with cosets $h_j^iN$. Since equivalent decorations have the same fundamental 
class, we may assume that $h_j^i=g_j^id_j^i$, where the $d_j^i$'s are diagonal 
matrices (see Remark~\ref{normalizerremark}). Let $c$ and $c^\prime$ denote the 
Ptolemy assignments on $K$ induced by $D$ and $D^\prime$. Suppose 
$d^i_j=\diag(d^i_{j0},\dots,d^i_{j,n-1})$. By \eqref{ctformula} we have 
\begin{equation}
c^{i\prime}_t=c^i_t\prod_{k=0}^{t_0}d^i_{0k}\prod_{k=0}^{t_1}d^i_{1k}
\prod_{k=0}^{t_2}d^i_{2k}\prod_{k=0}^{t_3}d^i_{3k}.
\end{equation}
Fix a lift $\widetilde c$ of $c$. Letting $\log$ denote a logarithm, 
define a lift $\widetilde c^\prime$ of $c^\prime$ by
\begin{equation}
\widetilde c^{\prime i}_t=\widetilde c^i_t+\sum_{k=0}^{t_0}\log(d^i_{0k})
+\sum_{k=0}^{t_1}\log(d^i_{1k})+\sum_{k=0}^{t_2}\log(d^i_{2k})+\sum_{k=0}^{t_3}\log(d^i_{3k}).
\end{equation}
Using this, one easily checks that $\lambda(c^i_\alpha)
=\lambda(c^{\prime i}_\alpha)$ for each $i$ and each $\alpha\in \Delta^3_{n-2}(\Z)$.
%\begin{multline}
%(\widetilde c_{\alpha_{03}}^{i\prime}+\widetilde c_{\alpha_{12}}^{i\prime}-\widetilde c_{\alpha_{02}}^{i\prime}-\widetilde c_{\alpha_{13}}^{i\prime},\widetilde c_{\alpha_{01}}^{i\prime}+\widetilde c_{\alpha_{23}}^{i\prime}-\widetilde c_{\alpha_{02}}^{i\prime}-\widetilde c_{\alpha_{13}}^{i\prime})=\\(\widetilde c^i_{\alpha_{03}}+\widetilde c^i_{\alpha_{12}}-\widetilde c^i_{\alpha_{02}}-\widetilde c^i_{\alpha_{13}},\widetilde c^i_{\alpha_{01}}+\widetilde c^i_{\alpha_{23}}-\widetilde c^i_{\alpha_{02}}-\widetilde c^i_{\alpha_{13}}).
%\end{multline}
The result now follows from Proposition~\ref{independenceoflog}.
\end{proof}

\subsection{$p\SL(n,\C)$-decorations}
\lbl{sub.pSLdecorations}

Let $n$ be even. All results in this section have 
natural analogs for $p\SL(n,\C)$. The proofs of these are obtained by replacing 
$2\pi\sqrt{-1}$ by $\pi\sqrt{-1}$, and logarithms by lifts of 
$\C\xrightarrow{\exp}\C^*\big/\langle\pm 1\rangle$. In particular, we have

\begin{proposition} 
\lbl{pSLindependence}
The fundamental class $[\rho]\in\widehat\B(\C)_{\PSL}$ of a decorated boundary-unipotent representation $\rho\colon\pi_1(M)\to p\SL(n,\C)$ is independent 
of the decoration.\qed
\end{proposition}

\section{A cocycle formula for $\widehat c$}
\lbl{sec.computationofchat}

Let $i_*\colon H_3(\SL(n,\C))\to H_3(\SL(n,\C),N)$ denote the natural map.
We wish to prove that the composition
\begin{equation}
\xymatrix{{H_3(\SL(n,\C))}\ar[r]^-{i_*}&{H_3(\SL(n,\C),N)}\ar[r]^-{\lambda}&
{\widehat\B(\C)}\ar[r]^-R&{\C/4\pi^2\Z}}
\end{equation}
equals the Cheeger-Chern-Simons class $\widehat c$. Note that $i_*$ is induced by the map $(g_0,\dots,g_3)\mapsto (g_0N,\dots,g_3N)$.

We shall make use of the canonical isomorphisms
\begin{equation}\label{canonicalisomorphisms}
H_3(\SL(n,\C))\cong H_3(\SL(3,\C))\cong H_3(\SL(2,\C))\oplus K_3^M(\C).
\end{equation}
The first isomorphism is induced by stabilization (see Suslin~\cite{Suslin1046}) and the second isomorphism 
is the $\pm$-eigenspace decomposition with respect to the transpose-inverse 
involution on $\SL(3,\C)$ (see Sah~\cite{Sah}).

\begin{lemma}
[Suslin~\cite{Suslin1046}]
\lbl{Milnorgenerator}
Let $D\subset\SL(3,\C)$ be the subgroup of diagonal matrices. The $K_3^M(\C)$ 
summand of $H_3(\SL(3,\C))$ is generated by the elements $B\rho_*([T])$, where 
$T=S^1\times S^1\times S^1$ is the $3$-torus, and $\rho\colon \pi_1(T)\to D$ is 
a representation.\qed
\end{lemma}

\begin{lemma}
\lbl{zeroonMilnor} 
Let $T=S^1\times S^1\times S^1$ and let $\rho\colon\pi_1(T)\to D$ 
be a representation.
The extended Bloch group element $[\rho]\in\widehat\B(\C)$ of $\rho$ is trivial.
\end{lemma}

\begin{proof}
We regard $T$ as a cube $C$ with opposite faces identified, and triangulate 
$C$ as the cone on the triangulation on $\partial C$ indicated in 
Figure~\ref{cubefigure} with cone point in the center. We order the vertices 
of each simplex by codimension, i.e.~the $0$-vertex is the cone point, the 
$1$-vertex is a face point etc. Let $\rho\colon\pi_1(T)\to D$ be a 
representation, and pick a decoration of $\rho$ by cosets in general 
position (the triangulation is such that this is always possible). 
Note that for every $3$-simplex $\Delta$ of $T$, there is a unique opposite 
$3$-simplex $\Delta^{\text{opp}}$, such that the faces opposite the cone point 
are identified. We may assume that the cone point is decorated by the coset 
$N$. If a simplex $\Delta$ is decorated by the cosets $(N,g_0N,g_1N,g_2N)$, 
the simplex $\Delta^{\text{opp}}$ must be decorated by the cosets 
$(N,dg_0N,dg_1N,dg_2N)$, where $d$ is the image of the generator of 
$\pi_1(T)$ pairing the faces of $\Delta$ and $\Delta^{\text{opp}}$. It thus 
follows from \eqref{cosetrep} that the fundamental class is represented by 
a sum of terms of the form
\begin{equation}
\lbl{terms}
(N,dg_0N,dg_1N,dg_2N)-(N,g_0N,g_1N,g_2N)\in C_3^{\gen}(\SL(n,\C)/N).
\end{equation}
Let $c$ and $c^\prime$ be the Ptolemy assignments associated to the tuples 
$(N\comma g_0N\comma g_1N\comma g_2N)$ and $(N\comma dg_0N\comma dg_1N\comma dg_2N)$. Letting 
$d=\diag(d_1,\dots,d_n)$, we have $c_t^\prime=c_t\prod_{i=t_0}^n d_i$. Fix a lift 
$\widetilde c$ of $c$, and consider the lift  
\begin{equation}
\widetilde c^\prime_t=\widetilde c_t+\sum_{i=t_0}^n\log(d_i)
\end{equation}
of $c^\prime$. One now checks that $\lambda(\widetilde c_\alpha^\prime)
=\lambda(\widetilde c_\alpha)$ for all $\alpha\in\dot\Delta^k_n(\Z)$, so 
$\lambda(\widetilde c)-\lambda(\widetilde c^\prime)=0$. This proves the result.
\end{proof}

\begin{theorem}
\lbl{formulaforchat} 
The composition $R\circ\lambda\circ i_*$ equals $\widehat c$.
\end{theorem}

\begin{proof}
Since $\lambda$ commutes with stabilization, it follows from Goette-Zickert~\cite{GoetteZickert} that $R\circ\lambda\circ i_*=\widehat c$ 
on $H_3(\SL(2,\C))$. Since $\widehat c$ is zero on $K_3^M(\C)$ (this follows 
from Lemma~\ref{Milnorgenerator} and 
\cite[Theorem~8.22]{CheegerSimons}), the result follows from~\eqref{canonicalisomorphisms} and
Lemma~\ref{zeroonMilnor}.
\end{proof}

\begin{figure}[htpb]
$$
\psdraw{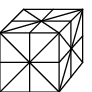}{1.2in}
$$
\caption{Triangulation of $\partial C$.}\lbl{cubefigure}
\end{figure}

\begin{remark}\label{extensionremark}By defining $\widehat c=R\circ\lambda\colon H_3(\SL(n,\C),N)\to\C/4\pi^2\Z$, we have a natural extension of the 
Cheeger-Chern-Simons class to bundles with boundary-unipotent holonomy, and we can define the complex volume as in Definition~\ref{complexvolumedefn}.
\end{remark}

\section{Recovering a representation from its Ptolemy coordinates}
\lbl{proofsection}

% DPT: Added "explicit"
We now show that a Ptolemy assignment on $K$ determines a
generically decorated
boundary-unipotent representation, which is given explicitly in terms of the Ptolemy coordinates. 
The idea is that a Ptolemy assignment 
canonically determines a $(G,N)$-cocycle on $M$. %Note that a $(G,N)$-cocycle 
%determines a boundary-unipotent representation (up to conjugation). 

\begin{definition} 
An $n\times n$ matrix $A$ is \emph{counter diagonal} if the only non-zero 
entries of $A$ are on the lower left to upper right diagonal, i.e.~$A_{ij}=0$ 
unless $j=n-i+1$. If $A_{ij}=0$ for $j>n-i+1$ (resp.~$j<n-i+1$), $A$ is 
\emph{upper} (resp.~\emph{lower}) \emph{counter triangular}.
\end{definition}

Given subsets $I,J$ of $\{1,\dots,n\}$, let $A_{I,J}$ denote the submatrix 
of $A$ whose rows and columns are indexed by $I$ and $J$, respectively. If 
$\Norm{I}=\Norm{J}$, let $\Norm A_{I,J}$ denote the minor $\det(A_{I,J})$.
Let $I^c$ denote $\{1,\dots,n\}\setminus I$. 

%Def of generic/general position (perhaps define earlier).
% Maybe generic is a better term than in general position.

Recall that the adjugate $\Adj(A)$ of a matrix $A$ is the matrix whose 
$ij$th entry is $(-1)^{i+j}\Norm{A}_{\{j\}^c,\{i\}^c}$. It is well known that 
$\Adj(A)=\det(A)A^{-1}$. The following result by Jacobi 
(see e.g.~\cite[Section~42]{DeterminantBook}) expresses the minors of 
$\Adj(A)$ in terms of the minors of $A$.

\begin{lemma} 
Let $I,J$ be subsets of $\{1,\dots,n\}$ with $\Norm{I}=\Norm{J}=r$. We have
\begin{equation}\lbl{Jacobi}
\Norm{\Adj(A)}_{I,J}=(-1)^{\sum(I,J)}\det(A)^{r-1}\Norm{A}_{J^c,I^c},
\end{equation}
% DPT: changed $i$, $j$ to $I$, $J$
where $\sum(I,J)$ is the sum of the indices occurring  in $I$ and~$J$.\qed
\end{lemma}

\begin{definition} 
A matrix $A\in\GL_n(\C)$ is \emph{generic} if 
$\Norm{A}_{\{k,\dots,n\},\{1,\dots,n-k+1\}}\neq 0$ for all $k\in\{1,\dots,n\}$.
\end{definition}

Note that $A$ is generic if and only if the Ptolemy coordinates of $(N,AN)$ 
are non-zero. The following is a generalization of 
Zickert~\cite[Lemma~3.5]{ZickertDuke}.

\begin{proposition}
\lbl{xyqExistence} 
Let $A\in \GL_n(\C)$ be generic. There exist unique $x\in N$ and $y\in N$ 
such that $q=x^{-1}Ay$ is counter diagonal. The entries of $x$, $y$ and $q$ 
are given by
\begin{gather} 
q_{n,1}=A_{n,1}, \quad q_{n-j+1,j}=(-1)^{j-1}
\frac{\Norm{A}_{\{n-j+1,\dots,n\},\{1,\dots, j\}}}{\Norm{A}_{\{n-j+2,\dots,n\},\{1,\dots,j-1\}}} 
\text{ for }1<j\leq n\lbl{q}\\
x_{ij}=\frac{\Norm{A}_{\{i,j+1,\dots,n\},\{1,\dots,n-j+1\}}}{
\Norm{A}_{\{j,\dots,n\},\{1,\dots,n-j+1\}}} \text{ for $j>i$}\lbl{x}\\
y_{ij}=(-1)^{i+j}\frac{\Norm{A}_{\{n-j+2,\dots,n\},\{1,\dots,\widehat i,\dots,j\}}}{
\Norm{A}_{\{n-j+2,\dots,n\},\{1,\dots,j-1\}}} \text{ for $j>i$}\lbl{y}.
\end{gather}
\end{proposition}

\begin{proof}
It is enough to prove existence and uniqueness of $x$ and $y$ in $N$ such 
that $Ay$ and $x^{-1}A$ are upper and lower counter triangular, respectively. 
Suppose $Ay$ is upper counter triangular. Then the vector $y_{\{1,\dots,j-1\},\{j\}}$ 
consisting of the part above the counter diagonal of the $j$th column vector 
of $y$ must satisfy 
\begin{equation}A_{\{n-j+2,\dots,n\},\{1,\dots,j-1\}}y_{\{1,\dots, j-1\},\{j\}}
+A_{\{n-j+2,\dots,n\},\{j\}}=0.
\end{equation}
%As an example, if 
%\begin{equation*}
%\begin{bmatrix} a&b&c\\d&e&f\\g&h&i\end{bmatrix}\begin{bmatrix}1&x&y\\0&1&z\\0&0&1\end{bmatrix}=\begin{bmatrix}0&0&\lambda\\0&\beta&0\\\alpha&0&0\end{bmatrix},
%\end{equation*}
%we must have $gx+h=0$ and $\left[\begin{smallmatrix}d&e\\g&h\end{smallmatrix}\right]\left[\begin{smallmatrix}y\\z\end{smallmatrix}\right]+\left[\begin{smallmatrix}f\\i\end{smallmatrix}\right]=0$.
The existence and uniqueness of $y$, as well as the given formula for the entries, now follow from Cramer's rule.
Since $x^{-1}A$ is lower counter-triangular if and only if $A^{-1}x$ is upper counter-triangular, existence and uniqueness of $x$ follows. The explicit formula for the entries follows from Jacobi's identity~\eqref{Jacobi} and the formula for the entries of $y$. To obtain the formula for the entries of $q$, note that $q_{n-j+1,j}=(Ay)_{n-j+1,j}$. Hence, $q_{n,1}=A_{n,1}$, and for $1<j\leq n$,
\begin{align*}
q_{n-j+1,j}&=\sum_{i=1}^{j-1}A_{n-j+1,i}y_{i,j}+A_{n-j+1,j}\\&=\frac{\sum_{i=1}^{j}(-1)^{i+j}A_{n-j+1,i}\Norm{A}_{\{n-j+2,\dots,n\},\{1,\dots\widehat i,\dots,j\}}}{\Norm{A}_{\{n-j+2,\dots,n\},\{1,\dots,j-1\}}}\\&=(-1)^{j-1}\frac{\Norm{A}_{\{n-j+1,\dots,n\},\{1,\dots,j\}}}{\Norm{A}_{\{n-j+2,\dots,n\},\{1,\dots,j-1\}}},
\end{align*}
where the second equality follows from~\eqref{y}.
\end{proof}

%\begin{corollary}\lbl{naturalrepresentative} The diagonal left $\SL(n,\C)$-action on tuples $(g_0N,\dots,g_kN)$ in general position is free when $k\geq 0$. Every orbit has a \emph{natural representative} of the form
%\begin{equation}(N,q_1N,\alpha_2q_2N,\dots,\alpha_kq_kN),
%\end{equation}
%where $c^i$ is counter diagonal ($i\geq 1$), and $\alpha_i\in N$ ($i\geq 2$). 
%\end{corollary}
%\begin{proof} Let $(g_0N,\dots,g_kN)$ be a tuple in general position. For each $0\leq j\leq k$, let $x_j\in N$, $y_j\in N$ and $c_j$ be such that $x_j^{-1}g_0^{-1}g_jy_j=c_j$, where $q_j$ is counter diagonal. Letting $\alpha_j=x_1^{-1}x_j$ for $j\geq 2$, we see that 
%\begin{equation*}(g_0N,\dots,g_kN)=(g_0N,g_0x_2q_1N,\dots,g_0x_kq_kN)\sim (N,q_1N,\alpha_2q_2N,\dots,\alpha_kq_kN).
%\end{equation*}
%Freeness of the action is now immediate.
%\end{proof}

For a generic matrix $A$, let $x_A$, $y_A$ and $q_A$ be the unique matrices provided by Proposition~\ref{xyqExistence}. Given cosets $g_{i}N$, $g_jN$, $g_kN$, define
\begin{equation}
\lbl{qandalpha}
q_{ij}=q_{g_i^{-1}g_j},\qquad \alpha^i_{jk}=(x_{g_i^{-1}g_j})^{-1}x_{g_i^{-1}g_k}.
\end{equation}

\begin{corollary}
\lbl{freeaction}
The diagonal left $G$-action on $C_k^{\gen}(G/N)$ is free when $k\geq 1$, and the chain complex $C_{*\geq 1}^{\gen}(G/N)\otimes_{\Z[G]}\Z$ computes relative homology.
\end{corollary}

\begin{proof}
By Proposition~\ref{xyqExistence}, every generic tuple $(g_0N,\dots,g_kN)$ may be uniquely written as
\begin{equation}
g_0x_{g_0^{-1}g_1}(N,q_{01}N,\alpha^0_{12}q_{02}N,\dots,\alpha^0_{1k}q_{0k}N).
\end{equation}
This proves that the $G$-action is free. Also note that for each generic tuple $(g_0N,\dots,g_kN)$, there exists a coset $gN$ such that $(gN,g_0N,\dots,g_kN)$ is generic. Hence, $C_{*\geq 1}^{\gen}(G/N)$ is acyclic, and is thus a free resolution of $\Ker(C_0(G/N)\to\Z$). This proves the result (see e.g.~Zickert~\cite[Theorem~2.1]{ZickertDuke}).
\end{proof}

A generic tuple $(g_0N,\dots,g_kN)$ determines a $(G,N)$-cocycle on a truncated simplex $\overline\Delta$, by labeling the long edges by $q_{ij}$ and the short edges by $\alpha^i_{jk}$ (see Figure~\ref{cocyclefigure}). In particular, a generically decorated representation $\rho$ determines a $(G,N)$-cocycle representing $\rho$. 

\begin{figure}[htpb]
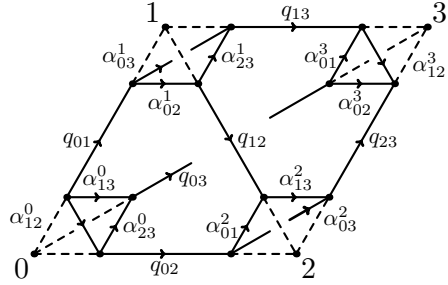

$$
\psdraw{SLrepFigures.8}{2.3in}
$$
\caption{A $(G,N)$-cocycle on a truncated $3$-simplex.}\lbl{cocyclefigure}
\end{figure}

A $(G,N)$-cocycle $\tau$ coming from a generic tuple is called \emph{generic}. 
Letting $\overline B_*^{\gen}(G,N)$ be the subcomplex of $\overline B_*(G,N)$ generated by generic cocycles on a standard simplex, it follows that we have a canonical isomorphism
\begin{equation}
\overline B_*^{\gen}(G,N)=C_*^{\gen}(G/N)\otimes_{\Z[G]}\Z.
\end{equation} 
It now follows from Proposition~\ref{fundclassrepprop}, that the fundamental class can be represented as in \eqref{fundclassrep}. 

We wish to prove that a generic $(G,N)$-cocycle is uniquely determined by the Ptolemy coordinates.
\begin{notation}
Let $k\in\{1,\dots,n-1\}$.
\begin{enumerate}[(i)]
\item For $a_1,\dots,a_n\in\C^*$, let $q(a_1,\dots,a_n)$ be the counter-diagonal matrix whose entries on the counter-diagonal (from lower left to upper right) are $a_1,\dots ,a_n$.
\item For $x\in\C$, let $x_k(x)$ be the elementary matrix whose $(k,k+1)$ entry is $x$. 
\item For $b_1,\dots, b_k\in\C$, let $\pi_k(b_1,\dots,b_k)=x_1(b_1)x_2(b_2)\cdots x_k(b_k)$.
\end{enumerate}
\end{notation}

% Given non-zero complex numbers $a_1,\dots,a_n$, let 
%\begin{equation} 
%q(a_1,\dots,a_n)=\begin{bmatrix}&&a_n\\&\iddots&\\a_1&&\end{bmatrix}.
%\end{equation}

\begin{proposition}
\lbl{longedges} 
The long edges of a generic $(G,N)$-cocycle are determined by the Ptolemy coordinates as follows:
\begin{equation}
\lbl{qij}
q_{ij}=q(a_1,\dots,a_n),\qquad a_k=(-1)^{k-1}\frac{c_{(n-k)e_i+ke_j}}{c_{(n-k+1)e_i+(k-1)e_j}}.
\end{equation} 
\end{proposition}

\begin{proof} 
Let $(g_0N,\dots,g_kN)$ be a generic tuple, and let $A=g_i^{-1}g_j$. 
Then $q_{ij}=q_A$. Since 
\begin{equation}
\Norm{A}_{\{n-j+1,\dots,n\},\{1,j\}}=\det(\{g_i\}_{n-k},\{g_j\}_k)=c_{(n-k)e_i+ke_j},
\end{equation} 
the result follows from~\eqref{q}.
\end{proof}

The corresponding formula for the short edges requires considerable more work, and is given in Proposition~\ref{shortedgeformula} below. %We refer to Example~\ref{longshortPtolemyexample} for an example. 

\begin{lemma} 
Let $A$ be generic, and let $L=x_A^{-1}A$.
The entries $L_{i,n-i+2}$ right below the counter diagonal are given by
\begin{equation}\lbl{belowcounterdiag}L_{i,n-i+2}=(-1)^{n-i}\frac{\Norm{A}_{\{i,\dots,n\},\{1,\dots,\widehat{n-i+1},n-i+2\}}}{\Norm{A}_{\{i+1,\dots,n\},\{1,\dots,n-i\}}}.
\end{equation}
\end{lemma}

\begin{proof}
We proceed as in the proof of Proposition~\ref{xyqExistence}. Let $x=x_A^{-1}$. Since $L$ is lower counter-triangular, we must have
\begin{equation}
x_{\{i\},\{i+1,\dots,n\}}A_{\{i+1,\dots,n\},\{1,\dots,n-i\}}+A_{\{i\},\{1,\dots,n-i\}}=0,
\end{equation}
so by Cramer's rule,
\begin{equation}
x_{ij}=(-1)^{i+j}\frac{\Norm{A}_{\{i,\dots,\widehat j,\dots,n\},\{1,\dots,n-i\}}}{\Norm{A}_{\{i+1,\dots,n\},\{1,\dots,n-i\}}}\text{ for }j>i.
\end{equation}
We thus have
\begin{align*}
\Norm{A}_{\{i+1,\dots,n\},\{1,\dots,n-i\}}L_{i,n-i+2}&=A_{i,n-i+2}\Norm{A}_{\{i+1,\dots,n\},\{1,\dots,n-i\}}\\&\hphantom{==}+\sum_{k=i+i}^n(-1)^{i+k}\Norm{A}_{\{j,\dots,\widehat k,\dots,n\},\{1,\dots,n-j\}}A_{k,n-i+2}\\&=\sum_{k=j}^n(-1)^{i+k}\Norm{A}_{\{j,\dots,\widehat k,\dots,n\},\{1,\dots,n-i\}}A_{k,n-i+2}\\&=(-1)^{n-i}\Norm{A}_{\{i,\dots,n\},\{1,\dots,\widehat{n-i+1},\dots,n-i+2\}},
\end{align*}
which proves the result.
\end{proof}

\begin{definition} 
Let $A,B\in\GL(n,\C)$. 
\begin{enumerate}[(i)]
\item $A$ and $B$ are related by a \emph{type $0$ move} if all but the last column of $A$ and $B$ are equal.
\item $A$ and $B$ are related by a \emph{type $1$ move} if all but the second last column of $A$ and $B$ are equal.
\item $A$ and $B$ are related by a \emph{type $2$ move} if for some $j<n-1$, $B$ is obtained from $A$ by switching columns $j$ and $j+1$. 
\end{enumerate}
\end{definition}

%For $x\in\C$ and $1\leq k\leq n-1$, let $x_k(x)$ be the elementary matrix whose $(k,k+1)$ entry is $x$.
%For a matrix $A$, let $A_i$ denote the $i$th column of $A$. 

\begin{proposition}
\lbl{moveeffectonx} 
Let $A$ and $B$ be generic, and let $A_i$ and $B_i$ denote the $i$th column of $A$, resp.~$B$.
\begin{enumerate}[(i)]
\item If $A$ and $B$ are related by a type $0$ move, $x_B=x_A$.
\item If $A$ and $B$ are related by a type $1$ move, $x_B=x_Ax_1(x)$, where
\begin{equation}
%x=-\frac{\Norm{A}_{\{3,\dots,n\},\{1,\dots,n-2\}}\det(C)}{\Norm{A}_{\{2,\dots,n\},\{1,\dots,n-1\}}\Norm{B}_{\{2,\dots,n\},\{1,\dots,n-1\}}}
x=-\frac{\det(A_1,\dots,A_{n-1},B_{n-1})\det(e_1,e_2,A_1,\dots,A_{n-2})}{\det(e_1,A_1,\dots,A_{n-1})\det(e_1,A_1,\dots,A_{n-2},B_{n-1})}.
\end{equation}
\item If $A$ and $B$ are related by a type $2$ move switching columns $j$ and $j+1$, $x_B=x_Ax_{n-j}(x)$, where
\begin{equation}
x=-\frac{\det(e_1,\dots,e_{n-j-1},A_1,\dots,A_{j+1})\det(e_1,\dots,e_{n-j+1},A_1,\dots,A_{j-1})}{\det(e_1,\dots,e_{n-j},A_1,\dots,A_j)\det(e_1,\dots,e_{n-j},A_1,\dots,A_{j-1},B_j)}.
\end{equation}
\end{enumerate}
\end{proposition}

\begin{proof}
The first statement follows from the fact that $x_A$ is independent of the last column of $A$.
Suppose $A$ and $B$ are related by a type $1$ move. Using \eqref{x}, one sees that $(x_A)_{ij}=(x_B)_{ij}$ except when $i=1$ and $j=2$. It thus follows that $x_B=x_Ax_1(x)$, where $x=(x_B)_{12}-(x_A)_{12}$. Letting $C$ be the matrix obtained from $A$ by replacing the $n$th column by the $(n-1)$th column of $B$, one has 
\begin{equation*}
\begin{gathered}
\Norm{A}_{\{1,3,\dots,n\},\{1,\dots,n-1\}}=\Adj(C)_{n,2},\quad \Norm{B}_{\{1,3,\dots,n\},\{1,\dots,n-1\}}=\Adj(C)_{n-1,2},\\
\Norm{A}_{\{2,\dots,n\},\{1,\dots,n-1\}}=\Adj(C)_{n,1},\quad \Norm{B}_{\{2,\dots,n\},\{1,\dots,n-1\}}=\Adj(C)_{n-1,1},
\end{gathered}
\end{equation*}
and it follows from~\eqref{x} that 
\begin{equation}\lbl{type1eq} x=(x_B)_{12}-(x_A)_{12}=\frac{\Adj(C)_{n-1,2}}{\Adj(C)_{n-1,1}}-\frac{\Adj(C)_{n,2}}{\Adj(C)_{n,1}}.
\end{equation}
%Letting $s=\Adj(C)_{n,1}\Adj(C)_{n-1,1}$, we have
We then have
\begin{align*}
x\Adj(C)_{n,1}\Adj(C)_{n-1,1}&=\Adj(C)_{n-1,2}\Adj(C)_{n,1}-\Adj(C)_{n-1,1}\Adj(C)_{n,2}\\&=-\Norm{\Adj(C)}_{\{n-1,n\},\{1,2\}}\\&=
-\det(C)\Norm{C}_{\{3,\dots,n\},\{1,\dots,n-2\}}\\&=-\det(A_1,\dots,A_{n-1},B_{n-1})\det(e_1,e_2,A_1,\dots,A_{n-2}),
\end{align*}
where the third equality follows from Jacobi's identity~\eqref{Jacobi}.
Since \begin{equation*}\Adj(C)_{n,1}\Adj(C)_{n-1,1}=\det(e_1,A_1,\dots,A_{n-1})\det(e_1,A_1,\dots,A_{n-2},B_{n-1}),\end{equation*} this proves the second statement. 

Now suppose $A$ and $B$ are related by a type $2$ move. Let $E_{j,j+1}$ be the elementary matrix obtained from the identity matrix by switching the $j$th and $(j+1)$th columns. Then $B=AE_{j,j+1}$. Since $L=x_A^{-1}A$ is lower counter triangular, $x_{n-j}(-\frac{L_{n-j,j+1}}{L_{n-j+1,j+1}})LE_{j,j+1}$ must also be lower counter triangular. We thus have
\begin{equation}
x_B=x_Ax_{n-j}(-\frac{L_{n-j,j+1}}{L_{n-j+1,j+1}})^{-1}=x_Ax_{n-j}(\frac{L_{n-j,j+1}}{L_{n-j+1,j+1}}).
\end{equation}
By \eqref{belowcounterdiag} and \eqref{q}, we have
\begin{equation}
\begin{aligned}
L_{n-j+1,j+1}&=(-1)^{j-1}\frac{\Norm{A}_{\{n-j+1,\dots,n\},\{1,\dots,\widehat j,j+1\}}}{\Norm{A}_{\{n-j+2,\dots,n\},\{1,\dots,j-1\}}}\\
L_{n-j,j+1}&=(-1)^j\frac{\Norm{A}_{\{n-j,\dots,n\},\{1,\dots,j+1\}}}{\Norm{A}_{\{n-j+1,\dots,n\},\{1,\dots,j\}}}.
\end{aligned}
\end{equation}
Hence,
\begin{align*}
\frac{L_{n-j,j+1}}{L_{n-j+1,j+1}}&=-\frac{\Norm{A}_{\{n-j,\dots,n\},\{1,\dots,j+1\}}\Norm{A}_{\{n-j+2,\dots,n\},\{1,\dots,j-1\}}}{\Norm{A}_{\{n-j+1,\dots,n\},\{1,\dots,j\}}\Norm{A}_{\{n-j+1,\dots,n\},\{1,\dots,\widehat j,j+1\}}}\\&=-
\frac{\det(e_1,\dots,e_{n-j-1},A_1,\dots,A_{j+1})\det(e_1,\dots,e_{n-j+1},A_1,\dots,A_{j-1})}{\det(e_1,\dots,e_{n-j},A_1,\dots,A_j)\det(e_1,\dots,e_{n-j},A_1,\dots,A_{j-1},B_j)},
\end{align*} 
proving the third statement.
\end{proof}

Note that any two matrices $A,B\in\GL(n,\C)$ are related by a sequence of moves of type $1$, $2$ and $0$ as follows:
\begin{equation}\lbl{movestring}
\begin{aligned}A\xrightarrow{1}&[A_1,\dots,A_{n-2},B_1,A_n]\xrightarrow{2}[A_1,\dots,A_{n-3},B_1,A_{n-2},A_n]\xrightarrow
{2}\dots\xrightarrow{2}\\&[B_1,A_1,\dots,A_{n-2},A_n]\xrightarrow{1}[B_1,A_1,\dots,A_{n-3},B_2,A_n]\xrightarrow{2}\dots\xrightarrow{2}\\&[B_1,B_2,A_1,\dots,A_{n-3},A_n]\xrightarrow{1,2}\dots\xrightarrow{1,2}[B_1,\dots,B_{n-1},A_n]\xrightarrow{0}B.
\end{aligned}
\end{equation}

%\subsubsection{Diamond coordinates}
Consider the tilings of a face ${ijk}$, $i<j<k$, of $\Delta^2_n$ by \emph{diamonds} shown in Figure~\ref{diamondfigure}. We refer to the diamonds as being of type $i$, $j$ and $k$, respectively. 
\begin{definition}\lbl{diamondcoordinates}
The \emph{diamond coordinate} $d^k_{r,s}$ of a diamond $(r,s)$ of type $k$ is the alternating product of the Ptolemy coordinates assigned to its vertices, see Figure~\ref{diamondfigure}.
\end{definition}
\begin{figure}[ht]
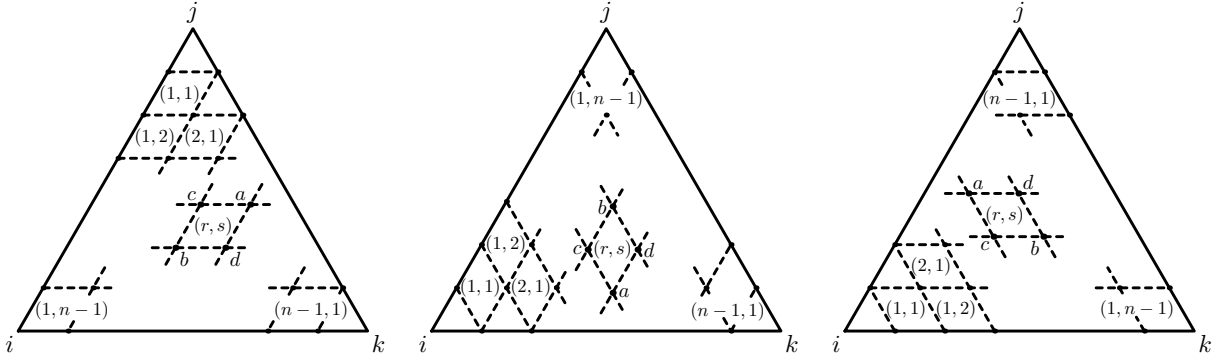

$$
\psdraw{SLrepFigures.9}{2in} \quad
\psdraw{SLrepFigures.10}{2in} \quad
\psdraw{SLrepFigures.11}{2in} 
$$
\caption{Diamonds of type $i$, $j$ and $k$. The diamond coordinates are $d^i_{r,s}=d^k_{r,s}=\frac{-ab}{cd}$, and $d^j_{r,s}=\frac{ab}{cd}$, where $a$, $b$, $c$, and $d$ are Ptolemy coordinates.}\lbl{diamondfigure}
\end{figure}

%Consider the tiling of $\Delta^2_n$ by \emph{diamonds} as shown in Figure~\ref{diamondfigure}. The \emph{diamond coordinate} of a diamond $(i,j)$ is 
%\begin{equation}\lbl{diamondcoordinate}
%d_{i,j}=\frac{c_{(j-1,n-i-j+1,i)}c_{(j+1,n-j-i,i-1)}}{c_{(j,n-j-i+1,i-1)}c_{(j,n-j-i,i}}.
%\end{equation}
%Note that the Ptolemy coordinates in the numerator and denominator of \eqref{diamondcoordinate} are the Ptolemy coordinates assigned to the vertices of the diamond $(i,j)$; see Figure~\ref{diamondfigure}.

%Figure captions: Tiling of $\Delta^2_n$ by diamonds,     A diamond $(i,j)$ with diamond coordinate $d_{i,j}=\frac{bc}{ad}$

%For $1\leq k\leq n-1$ and $a_1,\dots,a_k\in\C$, let 
%\begin{equation}
%\pi_k(a_1,\dots,a_k)=x_1(a_1)\cdots x_k(a_k).
%\end{equation}

\begin{proposition}
\lbl{shortedgeformula} 
The short edges $\alpha^i_{jk}$, $j<k$, of a generic $(G,N)$-cocycle are determined by the Ptolemy coordinates as follows:
\begin{equation}\lbl{alphaijk}
\alpha^i_{jk}=\pi_{n-1}(d^i_{1,1},\dots,d^i_{1,n-1})\pi_{n-2}(d^i_{2,1},\dots,d^i_{2,n-2})\cdots \pi_1(d^i_{n-1,1}),
\end{equation}
where the $d^i_{j,k}$'s are the type $i$ diamond coordinates on the face $ijk$.
\end{proposition}

\begin{proof} 
Let $(g_0N,\dots,g_lN)$ be a generic tuple, and let $A=g_i^{-1}g_j$ and $B=g_i^{-1}g_k$.
We assume that $i<j<k$, the other cases being similar. Note that the Ptolemy coordinates on the $ijk$ face are given by 
\begin{equation}
c_{t_ie_i+t_je_k+t_ke_k}=\det(e_1,\dots,e_{t_i},A_1,\dots,A_{t_j},B_1,\dots,B_{t_k}).
\end{equation}
Performing the sequence of moves in \eqref{movestring}, the result follows from Proposition~\ref{moveeffectonx}. 
\end{proof}

\begin{corollary}
A generic tuple is determined up to the diagonal $G$-action by its Ptolemy coordinates.\qed
\end{corollary}

This result together with Proposition~\ref{longedges} shows that there is at most one generic $(G,N)$-cocycle with a given collection of Ptolemy coordinates. We now prove that the Ptolemy relations are the only relations among the Ptolemy coordinates when $k\leq 3$.

\begin{example} 
Suppose Ptolemy assignments on $\Delta^2_n$, $n\in\{2,3\}$, are given as in Figure~\ref{Ptolemyandassignment}. Using~\eqref{qij} and \eqref{alphaijk}, we obtain that the corresponding $(G,N)$-cocycle is given by
% MOVE EXAMPLE TO END OF SECTION

\begin{equation}
\begin{gathered}
q_{01}=q(a,-1/a),\quad q_{12}=q(b,-1/b),\quad q_{02}=q(c,-1/c),\\
\alpha^0_{12}=x_1\Bigl(\frac{-b}{ac}\Bigr),\quad \alpha^1_{02}=x_1\Bigl(\frac{c}{ab}\Bigr),\quad \alpha^2_{01}=x_1\Bigl(\frac{-a}{cb}\Bigr)
\end{gathered}
\end{equation}
when $n=2$, and
\begin{equation}
\begin{gathered}
q_{01}=q(c,-a/c,1/a),\quad q_{12}-q(b,-e/b,1/e),\quad q_{02}=q(f,-g/f,1/g),\\
\alpha_{02}^1=x_1\Bigl(\frac{fa}{cd}\Bigr)x_2\Bigl(\frac{d}{ab}\Bigr)x_1\Bigl(\frac{gb}{de}\Bigr),\\
\alpha^0_{12}=x_1\Bigl(\frac{-bc}{ad}\Bigr)x_2\Bigl(\frac{-d}{cf}\Bigr)x_1\Bigl(\frac{-ef}{dg}\Bigr),\quad \alpha_{01}^2=x_1\Bigl(\frac{-cg}{fd}\Bigr)x_2\Bigl(\frac{-d}{ge}\Bigr)x_1\Bigl(\frac{-ae}{db}\Bigr)
\end{gathered}
\end{equation}
when $n=3$.
% DPT: The font is way too small in the next figure
% Also, I have no idea where SLrepFigures.23 comes from...
\begin{figure}[ht]
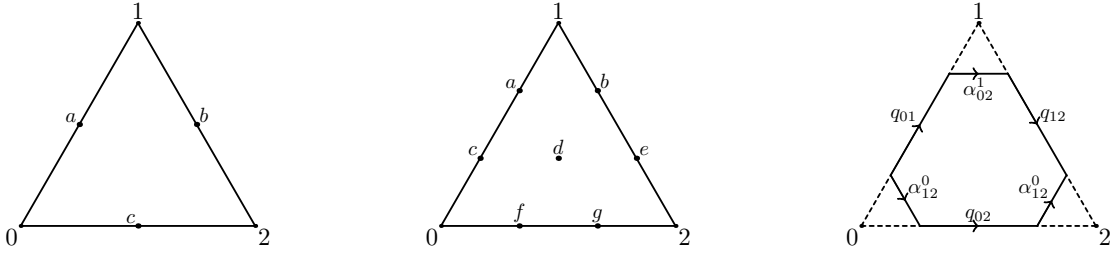

$$
\psdraw{SLrepFigures.23}{1.4in}\hspace{2cm}
\psdraw{SLrepFigures.19}{1.4in}\hspace{2cm}
\psdraw{SLrepFigures.20}{1.4in}
$$
\caption{Ptolemy assignments and the corresponding cocycle for $n=2$ and $n=3$.}\lbl{Ptolemyandassignment}
\end{figure}
\end{example}

\begin{lemma}
\lbl{uniquefactorization}
Let $a_{i,j}$ and $b_{i,j}$ be non-zero complex numbers. The equality 
\begin{equation} \pi_{n-1}(a_{1,1},\dots,a_{1,n-1})\cdots \pi_1(a_{n-1,1})=\pi_{n-1}(b_{1,1},\dots,b_{1,n-1})\cdots \pi_1(b_{n-1,1})
\end{equation}
holds if and only if $a_{i,j}=b_{i,j}$ for all $i,j$.
\end{lemma}

\begin{proof}
For any $c_{i,j}$, the $n$th column of $\pi_{n-1}(c_{1,1},\dots,c_{1,n-1})\cdots \pi_1(c_{n-1,1})$ is equal to the $n$th column of $\pi_{n-1}(c_{1,1},\dots,c_{1,n-1})$, which equals
\[(\prod_{i=1}^{n-1} c_{1,i},\prod_{i=2}^{n-1}c_{1,i},\dots,c_{1,n-1}).\]
This proves that $a_{1,j}=b_{1,j}$ for all $j$. The result now follows by induction.%By considering the $(n-1)$th column of $\pi_{n-2}(c_{2,1},\dots,c_{2,n-2})\cdots\penalty300 \pi_1(c_{n-1,1})$, we obtain that $a_{2,j}=b_{2,j}$ for all $j$. Continuing this process yields the result.
\end{proof}

\begin{proposition}
\lbl{reconstructtwotuple} 
For any assignment $c\colon\dot\Delta^2_n(\Z)\to\C^*$, there is a unique Ptolemy assignment $c\in Pt_2^n$ whose Ptolemy coordinates are $c_t$.
\end{proposition}

\begin{proof} 
We prove that the Ptolemy coordinates $c_t^\prime$ of $(N,q_{01}N,\alpha^0_{12}q_{02}N)$ equal $c_t$, where $q_{01}$, $q_{02}$ and $\alpha^0_{12}$ are given in terms of the $c_t$'s by \eqref{qij} and \eqref{alphaijk}. First note that $c_t=c^\prime_t$ if either $t_1$ or $t_2$ is $0$, i.e.~if $t$ is on one of the edges of $\Delta^2_n$ containing the $0$th vertex. 
Each of the other integral points $t$ is the upper right vertex of a unique diamond $(r,s)$ of type $0$.
Let $\tau_k$ be the upper right vertex of the $k$th diamond $D_k$ in the sequence
\begin{equation}
\lbl{diamondsequence}
(1,n-1),(1,n-2),\dots(1,1),(2,n-2),\dots,(2,1),\dots,(n-1,1).
\end{equation}
By Lemma~\ref{uniquefactorization}, $d^{0\prime}_{r,s}=d^0_{r,s}$ for all diamonds $(r,s)$ of type $0$. It thus follows that if $c_t=c^\prime_t$ for all but one of the vertices of a diamond $D$, then $c_t=c^\prime_t$ for all vertices of $D$. In particular $c^\prime_{\tau_1}=c_{\tau_1}$.
Suppose by induction that $c^\prime_{\tau_i}=c_{\tau_i}$ for all $i<k$. Then $c^\prime_t=c_t$, for all vertices of $D_k$ except $\tau_k$. Hence, we also have $c^\prime_{\tau_k}=c_{\tau_k}$, completing the induction.
\end{proof}

\begin{proposition}
\lbl{reconstructtriple} 
For any assignment $c\colon\\dot\Delta^3_n(\Z)\to\C^*$ satisfying the Ptolemy relations, there is a unique Ptolemy assignment $c\in Pt_3^n$ whose Ptolemy coordinates are $c_t$.
\end{proposition}

\begin{proof}
Let $c^\prime_t$ be the Ptolemy coordinates of the tuple $(N,q_{01}N,\alpha^0_{12}q_{02}N,\alpha^0_{13}q_{03}N)$ defined from the $c_t$'s by \eqref{qij} and \eqref{alphaijk}. We wish to prove that $c^\prime_t=c_t$ for all $t$. Note that if, for some subsimplex $\Delta^3(\alpha)$, $c^\prime_{\alpha_{ij}}=c_{\alpha_{ij}}$ for all but one of the $6$ $\alpha_{ij}$'s, then $c^\prime_{\alpha_{ij}}=c_{\alpha_{ij}}$ holds for all $\alpha_{ij}$. This is a direct consequence of the Ptolemy relations.
By Proposition~\ref{reconstructtwotuple}, $c^\prime_t=c_t$, when either $t_2$ or $t_3$ is zero.
Hence, for each $\alpha=(a_0,a_1,a_2,a_3)$ with $a_2=a_3=0$, $c^\prime_{\alpha_{ij}}=c_{\alpha_{ij}}$ except possibly when $(i,j)=(2,3)$. As explained above, $c^\prime_{\alpha_{23}}=c_{\alpha_{23}}$ as well. Now suppose by induction that $c^\prime_{\alpha_{ij}}=c_{\alpha_{ij}}$ for all $\alpha$ with $a_2+a_3<k$. Then $c^\prime_{\alpha_{ij}}=c_{\alpha_{ij}}$ holds except possibly when $(i,j)=(2,3)$. Again, $c^\prime_{\alpha_{23}}=c_{\alpha_{23}}$ must also hold, completing the induction. 
\end{proof}

A $(G,N)$-cocycle on $M$ obviously determines a decorated representation (up to conjugation). The main results of this section can thus be summarized by the diagram below.
\begin{equation}\label{summaryof1to1}
\xymatrix{\left\{\txt{Points in $P_n(K)$}\right\}\ar@{<->}[r]&\left\{\txt{Generic
$(G,N)$-cocycles\\on $M$}\right\}\ar@{<->}[r]&\left\{\txt{Generically decorated\\
$(G,N)$-representations}\right\}}
\end{equation}

\begin{remark}
We stress that the Ptolemy variety parametrizes decorated representations and \emph{not} decorated representations up to equivalence. In particular, the dimension of $P(K)$ depends on the triangulation, and may be very large if $K$ has many interior vertices. 
\end{remark}

%\begin{remark}
%\lbl{rem.FrenchPeopleandDgroupoid}
%For $n=3$, coordinates parametrizing $3$-cycles decorated by flags have also been considered (independently) by Bergeron, Falbel and Guilloux~\cite{FrenchPeople}.
%An alternative parametrization of representations is given by Kashaev's $\Delta$-groupoid~\cite{KashaevD}.
%\end{remark}

\subsection{Obstruction cocycles and the $p\SL(n,\C)$-Ptolemy varieties}
\lbl{sub.obstruction}
Suppose $n$ is even. The projection $G\to pG$ maps $N$ isomorphically onto its image (also denoted by $N$), and by elementary obstruction theory (see e.g.~Steenrod~\cite{Steenrod}), the obstruction to lifting a $(pG,N)$-representation $\rho$ to a $(G,N)$-representation is a class in 
\begin{equation}
H^2(M,\partial M;\Z/2\Z)=H^2(K;\Z/2\Z).
\end{equation} 
We can represent it by an explicit cocycle in $Z^2(K;\Z/2\Z)$ as follows: Pick any $(p\SL(n,\C),N)$-cocycle $\bar\tau$ on $M$ representing $\rho$ and a lift $\tau$ of $\bar\tau$ to a $(G,N)$-cochain. Each $2$-cell of $K$ corresponds to a hexagonal $2$-cell of $M$, and the $2$-cocycle $\sigma\in Z^2(K;\Z/2\Z)$ taking a $2$-cell to the product of the $\tau$-labelings along the corresponding hexagonal $2$-cell of $M$ represents the obstruction class. %One can easily check directly that the cohomology class of $\sigma$ is independent of the choice of $\bar\tau$ and $\tau$.

\begin{proposition} 
Suppose the interior of $M$ is a $1$-cusped hyperbolic $3$-manifold with finite volume. The obstruction class in $H^2(K;\Z/2\Z)$ to lifting the geometric representation is non-trivial.
\end{proposition}

\begin{proof} 
By a result of Calegari~\cite[Corollary~2.4]{Calegari}, any lift of the geometric representation takes a longitude to an element in $\SL(2,\C)$ with trace $-2$. This shows that no lift is boundary-unipotent, so the obstruction class must be non-trivial.
\end{proof}

Proposition~\ref{xyqExistence} also holds in $p\SL(n,\C)$, and we thus have a $1$-$1$ correspondence between generically decorated representations and $(pG,N)$-cocycles on $M$.

\begin{definition} 
Let $\sigma\in Z^2(K;\Z/2\Z)$. A lifted $(pG,N)$ cocycle on $M$ with obstruction cocycle $\sigma$ is a generic $(G,N)$-assignment on $M$ lifting a $(pG,N)$-cocycle on $M$ such that the 2-cocycle on $K$ obtained by taking products along hexagonal faces of $M$ equals $\sigma$.
\end{definition}

A $1$-cochain $\eta\in C^1(K;\Z/2\Z)$ acts on a lifted $(pG,N)$-cocycle $\tau$ by multiplying a long edge $e$ by $\eta(e)$. Note that if $\tau$ has obstruction cocycle $\sigma$, $\eta\tau$ has obstruction cocycle $\delta(\eta)\sigma$, where $\delta$ is the standard coboundary operator.
Recall that there is a $1$-$1$ correspondence between generic $(G,N)$-cocycles on $M$ and points in the Ptolemy-variety.
We shall prove a similar result for $pG$.

We wish to define a coboundary action on $pG$-Ptolemy assignments (see Definition~\ref{pSLPtolemydefn}). Let $c$ be a $pG$-Ptolemy assignment on $\Delta$, and let $\eta_{ij}\in C^1(\Delta;\Z/2\Z)$ be the cochain taking the edge $ij$ to $-1$ and all other edges to $1$. Define
\begin{equation}\lbl{actiononPtolemys}
\eta_{ij} c\colon\dot\Delta^3_n(\Z)\to \C^*,\qquad (\eta_{ij}c)_t=(-1)^{t_it_j}c_t
\end{equation} 
and extend in the natural way to define $\eta c$ for a $pG$-Ptolemy assignment $c$ on $K$ and $\eta\in C^1(K;\Z/2\Z)$. A priori $\eta c$ is only an assignment of complex numbers to the integral points of the simplices of $K$. However, we have:
 
\begin{lemma}
\lbl{boundaryactiononPtolemy} 
If $c$ is a $pG$-Ptolemy assignment on $K$ with obstruction cocycle $\sigma$, $\eta c$ is a $pG$-Ptolemy assignment on $K$ with obstruction cocycle $\delta(\eta)\sigma$.
\end{lemma}

\begin{proof}
It is enough to prove this for a simplex $\Delta$ and for $\eta=\eta_{ij}$. Let $c^\prime=\eta_{ij}c$. We assume for simplicity that $ij=01$; the other cases are proved similarly. For any $\alpha=(a_0,a_1,a_2,a_3)\in \Delta^k_{n-2}(\Z)$, we then have
\begin{equation}\lbl{actionpreservesPtolemy}
c^\prime_{\alpha_{03}}c^\prime_{\alpha_{12}}+c^\prime_{\alpha_{01}}c^\prime_{\alpha_{23}}-c^\prime_{\alpha_{02}}c^\prime_{\alpha_{13}}=
(-1)^{a_0+a_1}(c_{\alpha_{03}}c_{\alpha_{12}}-c_{\alpha_{01}}c_{\alpha_{23}}-c_{\alpha_{02}}c_{\alpha_{13}})
\end{equation}
Let $\tau=\delta(\eta_{01})$. Since $\delta(\eta_{01})_2=\delta(\eta_{01})_3=-1$ and $\delta(\eta_{01})_0=1$, \eqref{actionpreservesPtolemy} implies that
\begin{equation}
\tau_2\tau_3c^\prime_{\alpha_{03}}c^\prime_{\alpha_{12}}+\tau_0\tau_3c^\prime_{\alpha_{03}}c^\prime_{\alpha_{01}}c^\prime_{\alpha_{23}}=c^\prime_{\alpha_{02}}c^\prime_{\alpha_{13}},
\end{equation}
as desired.
\end{proof}

\begin{definition}
\lbl{pSLdiamondcoordinates}
The diamond coordinates of a $p\SL(n,\C)$-Ptolemy assignment with obstruction cocycle $\sigma$ are defined as in Definition~\ref{diamondcoordinates}, but multiplied by the sign (provided by $\sigma$) of the face.
\end{definition}

Note that for $\eta\in C^1(K;\Z/2/\Z)$, the diamond coordinates of $c$ and $\eta c$ are identical.

\begin{proposition}
\lbl{pGcocyclesandPtolemy} 
For any $\sigma\in Z^2(K;\Z/2\Z)$, there is a $1$-$1$ correspondence between $p\SL(n,\C)$-Ptolemy assignments on $K$ with obstruction cocycle $\sigma$, and lifted $(p\SL(n,\C),N)$-cocycles on $M$ with obstruction cocycle $\sigma$. The correspondence preserves the coboundary actions.
\end{proposition}

\begin{proof}
It is enough to prove this for a simplex $\Delta$. For a $pG$-Ptolemy assignment $c$ on $\Delta$ with obstruction cocycle $\sigma\in Z^2(\Delta;\Z/2\Z)$, define a cochain $\tau$ on $\overline\Delta$ by the formulas~\eqref{qij} and \eqref{alphaijk} using the $\sigma$-modified diamond coordinates (Definition~\ref{pSLdiamondcoordinates}). Let $\eta\in C^1(\Delta;\Z/2\Z)$ be such that $\delta\eta=\sigma$, where $\delta$ is the standard coboundary map. By Lemma~\ref{boundaryactiononPtolemy} $\eta c$ satisfies the $\SL(n,\C)$ Ptolemy relations~\eqref{Ptolemyrelation}, and hence corresponds to an $(\SL(n,\C),N)$-cocycle $\tau^{\prime}$. Since the diamond coordinates of $c$ and $\eta c$ are the same, the short edges of $\tau^\prime$ agree with those of $\tau$ and the long edges differ from those of $\tau$ by $\eta$. This proves that $\tau$ is a lifted $(pG,N)$-cocycle with obstruction cocycle $\sigma$. The inductive arguments of Propositions~\ref{reconstructtwotuple} and \ref{reconstructtriple} show that this is a $1$-$1$ correspondence. The fact that the actions by coboundaries correspond is immediate from the construction.
\end{proof}

\begin{corollary}
\lbl{pGvariety}
% DPT: Reworded here
Let $\sigma\in Z^2(K;\Z/2\Z)$. There is an algebraic variety
$P_n^\sigma(K)$ of generically decorated boundary-unipotent
representations $\rho\colon\pi_1(M)\to p\SL(n,\C)$ whose obstruction
class to lifting to $\SL(n,\C)$ is represented by~$\sigma$. Up to canonical
isomorphism, the variety
$P_n^\sigma(K)$ only depends on the cohomology class of~$\sigma$.
\end{corollary}

\begin{proof}
This follows immediately from Proposition~\ref{pGcocyclesandPtolemy}.
\end{proof}

Note that the canonical isomorphisms in Corollary~\ref{pGvariety} respect the extended Bloch group element. This follows from the $pG$ variant of Proposition~\ref{independenceoflog}. The analogue of~\eqref{summaryof1to1} is
\begin{equation}
\xymatrix{\left\{\txt{Points in $P^\sigma_n(K)$}\right\}\ar@{<->}[r]&\left\{\txt{Lifted
$(pG,N)$-cocycles on $M$\\with obstruction cocycle $\sigma$}\right\}\ar@{->>}[r]^{k:1}&\left\{\txt{Generically decorated\\
$(pG,N)$-representations\\with obstruction cocycle $\sigma$}\right\},}
\end{equation}
where $k$ is the number of lifts, i.e.~$k=\vert Z^1(K;\Z/2\Z)\vert$.

\subsection{Proof of Theorem~\ref{mainthmintro} and Theorem~\ref{mainthmpSLintro}}\label{Proofofmainthm}
Let $\rho\colon P_n(K)\to R_{G,N}(M)\big/\Conj$ be the composition of the map in~\eqref{summaryof1to1} with the forgetful map ignoring the decoration.  Let $c\in P_n(K)$. By Proposition~\ref{lambdaimageinBhat}, $\lambda(c)$ is in $\widehat\B(\C)$ and by Proposition~\ref{independenceofdec}, $\lambda(c)$ only depends on $\rho(c)$.  Commutativity of diagram~\eqref{mainthmdiag} follows from Remark~\ref{extensionremark}, and the fact that $\rho$ is surjective if $K$ is sufficiently fine follows from Proposition~\ref{barycentric}.
This concludes the proof of Theorem~\ref{mainthmintro}. The first part of Theorem~\ref{mainthmpSLintro} is proved similarly, and the last part follows from Theorem~\ref{Ptolemydetectsrhogeoproof} below.

\section{Examples}
\lbl{sec.examples}

In the examples below, all computations of Ptolemy varieties are exact, whereas the computations of complex volume are numerical with at least $50$ digits precision.

\begin{example}
[The $5_2$ knot complement]
\lbl{ex.52}
Consider the $3$-cycle $K$ obtained from the simplices in 
Figure~\ref{52triangulation} by identifying the faces via the unique 
simplicial attaching maps preserving the arrows. The space obtained 
from $K$ by removing the $0$-cell is homeomorphic to the complement of 
the $5_2$ knot, as can be verified by SnapPy~\cite{SnapPy}.

\begin{figure}[ht]
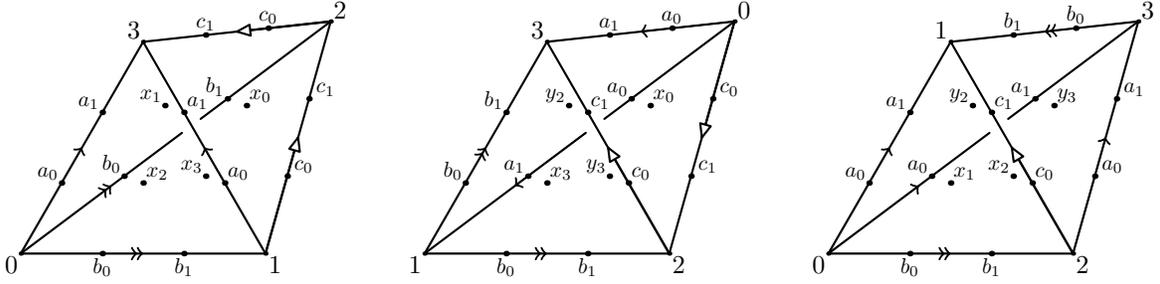

$$
\psdraw{SLrepFigures.12}{1.8in} \qquad
\psdraw{SLrepFigures.13}{1.8in} \qquad
\psdraw{SLrepFigures.14}{1.8in} 
$$
\caption{A $3$-cycle structure on the $5_2$ knot complement, and Ptolemy 
coordinates for $n=3$.}\lbl{52triangulation}
\end{figure}

Labeling the Ptolemy coordinates as in Figure~\ref{52triangulation}, the 
Ptolemy variety for $n=3$ is given by the equations
\begin{equation}
\begin{aligned}
&a_0x_3+b_0x_1=b_0x_2,&\quad &a_0y_3+a_0x_0=c_0y_2,&\quad &a_0x_2+b_0y_2=a_0x_1\\
&x_2c_0+b_1x_0=x_3a_0,&\quad &y_2b_0+a_1x_3=y_3b_0,&\quad &x_1a_0+b_1y_3=x_2c_0\\
&x_1c_1+x_3c_0=b_1x_0,&\quad &x_0b_1+y_3c_0=c_1x_3,&\quad &y_2a_1+x_2b_0=a_1y_3\\
&a_1x_0+x_2c_1=x_1a_1,&\quad &a_1x_3+y_2c_1=x_0b_1,&\quad &a_1y_3+x_1b_1=y_2c_1
\end{aligned} 
\end{equation}
together with an extra equation (involving an additional variable $t$)
\begin{equation}
a_0a_1b_0b_1c_0c_1x_0x_1x_2x_3y_2y_3t=1,
\end{equation}
making sure that all Ptolemy coordinates are non-zero. By 
Remark~\ref{normalizerremark}, the diagonal matrices act on the decorations, 
and one easily checks that the action by a matrix $\diag(x,y,z)$ 
with determinant $1$ multiplies a Ptolemy coordinate on an edge by $x^2y$ and a Ptolemy coordinate on a face by $x^3$.
Since we are not interested in the particular decoration, we may thus assume e.g.~that $a_0=y_3=1$.  
Using Magma~\cite{Magma}, one finds that 
the Ptolemy variety, after setting $a_0=y_3=1$, has three zero-dimensional
components with $3$, $4$ and $6$ points respectively. 
One of these is given by
\begin{equation}
\lbl{geometriccomp}
\begin{gathered}
a_0=a_1=y_3=1,\quad x_1=-1, \quad c_0=c_1=x_0^2+2x_0+1\\
y_2=x_0^2+2=-x_2,\quad x_3=-x_0^2-x_0-1\\
x_0^3+x_0^2+2x_0+1=0 
%  	 a0 - 1,
%        a1 - 1,
%        b0 + x0,
%        b1 + x0,
%        c0 - x0^2 - 2*x0 - 1,
%        c1 - x0^2 - 2*x0 - 1,
%        y3 - 1,
%        y2 - x0^2 - 2,
%	 x3 + x0^2 + x0 + 1,
%        x2 + x0^2 + 2,
%	 x1 + 1,
%        x0^3 + x0^2 + 2*x0 + 1
\end{gathered}
\end{equation}
Thus, this component gives rise to $3$ representations, one for each 
solution to $x_0^3+x_0^2+2x_0+1=0$. Using the fact that 
$R(\lambda(c))=i\Vol_\C(\rho)$, the complex volumes of these can be 
computed to be
\begin{equation}
\lbl{SL3geometric} 
0.0-4.453818209\dots i\in\C/4\pi^2i\Z,
\qquad \pm 11.31248835\ldots+12.09651350\dots i\in\C/4\pi^2i\Z
\end{equation}
corresponding to the values $x_0=-0.5698\dots$ and 
$x_0=-0.2150\mp1.3071\dots i$, respectively. 

In Zickert~\cite[Section~6]{ZickertDuke}, the complex volumes of the 
Galois conjugates of the geometric representation are computed to be
\begin{equation}
\lbl{PSL2geometric}
0.0-1.113454552\ldots i\in\C/\pi^2i\Z,
\qquad \pm 2.828122088\ldots + 3.024128376\ldots i\in\C/\pi^2i\Z.
\end{equation}
Notice that \eqref{SL3geometric} is (approximately) $4$ times 
\eqref{PSL2geometric}. It thus follows from Theorem~\ref{detectgeo} 
that the representations given by~\eqref{geometriccomp} are $\phi_3$ 
composed with the geometric component of $\PSL(2,\C)$-representations 
and that the factor of $4$ is exact.

Another component is given by
\begin{equation}
\begin{gathered} 
a_0=a_1=y_3=1,\quad x_1=-1,\quad b_1=-x_0\\
b_0=1/4x_0^3-1/4x_0^2+3/4x_0-1/2\\
c_0=c_1=1/4x_0^3-1/4x_0^2-1/4x_0+1/2\\
y_2=-x_2=1/4x_0^3+3/4x_0^2+7/4x_0+3/2\\
x_3=-x_0^2-x_0-1\\
x_0^4+x_0^3+x_0^2-4x_0-4=0.  
%        a0 - 1,
%        a1 - 1,
%        b0 - 1/4*x0^3 + 1/4*x0^2 - 3/4*x0 + 1/2,
%        b1 + x0,
%        c0 - 1/4*x0^3 + 1/4*x0^2 + 1/4*x0 - 1/2,
%        c1 - 1/4*x0^3 + 1/4*x0^2 + 1/4*x0 - 1/2,
%	 y3 - 1,
%        y2 - 1/4*x0^3 - 3/4*x0^2 - 7/4*x0 - 3/2,
%	 x3 + x0^2 + x0 + 1,
%        x2 + 1/4*x0^3 + 3/4*x0^2 + 7/4*x0 + 3/2,
%	 x1 + 1,
%        x0^4 + x0^3 + x0^2 - 4*x0 - 4
\end{gathered}
\end{equation}
In this case there are two distinct complex volumes given by:
\begin{equation}
0.0+2.631894506\ldots i=\frac{4}{15}\pi^2i\in\C/4\pi^2i\Z,
\quad 0.0+10.527578027\ldots i=\frac{16}{15}\pi^2i\in\C/4\pi^2i\Z.
\end{equation}

The third component has somewhat larger coefficients, but after introducing 
a variable $u$ with $u^6 + 5u^4 + 8u^2 - 2u + 1=0$, the defining equations 
simplify to
\begin{equation}
\begin{gathered}
a_0=y_3=1,\quad a_1=1/4u^5+1/4u^4+5/4u^3+1/2u^2+2u-3/4\\
b_0=b_1=-1/4u^4-3/4u^2-1/4u-3/4,\\
c_1=-1/4u^5-3/4u^3-1/4u^2-3/4u,\\
c_0=1/2u^5+9/4u^3+1/4u^2+7/2u-1/4,\\
y_2=-8/17u^5-1/34u^4-79/34u^3-3/17u^2-105/34u+26/17,\\
x_3=1/17u^5-1/17u^4+6/17u^3-6/17u^2+14/17u-16/17,\\
x_2=9/34u^5+4/17u^4+37/34u^3+31/34u^2+75/34u+13/17,\\
x_1=8/17u^5+1/34u^4+79/34u^3+3/17u^2+139/34u-9/17,\\
x_0=15/34u^5+1/17u^4+73/34u^3+29/34u^2+125/34u-1/17,\\
u^6 + 5u^4 + 8u^2 - 2u + 1=0.\\
%	a0 - 1,
%        a1 - 1/4*u^5 - 1/4*u^4 - 5/4*u^3 - 1/2*u^2 - 2*u + 3/4,
%        b0 + 1/4*u^4 + 3/4*u^2 + 1/4*u + 3/4,
%        b1 + 1/4*u^4 + 3/4*u^2 + 1/4*u + 3/4,
%        c0 - 1/2*u^5 - 9/4*u^3 - 1/4*u^2 - 7/2*u + 1/4,
%        c1 + 1/4*u^5 + 3/4*u^3 + 1/4*u^2 + 3/4*u,
%        y3 - 1,
%        y2 + 8/17*u^5 + 1/34*u^4 + 79/34*u^3 + 3/17*u^2 + 105/34*u - 26/17,
%        x3 - 1/17*u^5 + 1/17*u^4 - 6/17*u^3 + 6/17*u^2 - 14/17*u + 16/17,
%        x2 - 9/34*u^5 - 4/17*u^4 - 37/34*u^3 - 31/34*u^2 - 75/34*u - 13/17,
%        x1 - 8/17*u^5 - 1/34*u^4 - 79/34*u^3 - 3/17*u^2 - 139/34*u + 9/17,
%        x0 - 15/34*u^5 - 1/17*u^4 - 73/34*u^3 - 29/34*u^2 - 125/34*u + 1/17,
%        u^6 + 5*u^4 + 8*u^2 - 2*u + 1
\end{gathered}
\end{equation}
In this case, there are $3$ distinct complex volumes:
\begin{equation}
0.0+1.241598704\dots i,\quad \pm 6.332666642\ldots+1.024134714\dots i
\end{equation}
According to Conjecture~\ref{newconjecture}, $6.33\dots + 1.02\dots i$ 
should (up to rational multiples of $\pi^2i$) be an integral linear 
combination of complex volumes of hyperbolic manifolds. 
Using e.g.~Snap~\cite{snapprogram}, one checks that the complex volume of 
the manifold $m034$ is given by
\begin{equation}
3.166333321\ldots + 2.157001424\dots i,
\end{equation}
and we have
\begin{equation}
6.332666642\ldots+1.024134714\dots i=2\Vol_\C(m034)
-\frac{1}{3}\pi^2i\in\C/4\pi^2i\Z.
\end{equation}
%This was found using a script written by Matthias Goerner.
\end{example}

\begin{example}
[The figure $8$ knot complement]
\lbl{ex.41}
Let $K$ be the $3$-cycle in Figure~\ref{fig8triangulation}. Then $M=M(K)$ 
is the figure $8$ knot complement, and $H^2(K;\Z/2\Z)=H^2(M,\partial M;
\Z/2\Z)=\Z/2\Z$.

\begin{figure}[ht]
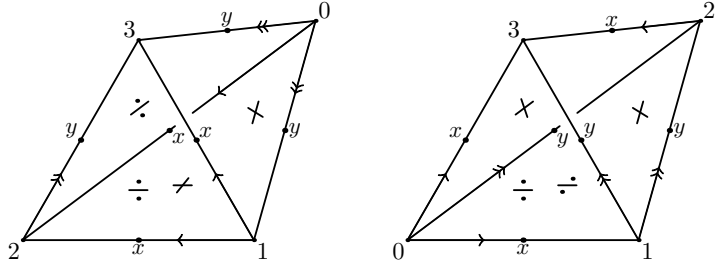

$$
\psdraw{SLrepFigures.15}{1.7in} \qquad
\psdraw{SLrepFigures.16}{1.7in} 
$$
\caption{A $3$-cycle structure on the figure $8$ knot complement and 
Ptolemy coordinates for $n=2$. The signs indicate the non-trivial second 
$\Z/2\Z$ cohomology class.} %We have $\epsilon_1=-1$, $\epsilon_2=1$.
\lbl{fig8triangulation}
\end{figure}

For the trivial obstruction class, the Ptolemy variety for $n=2$ is given by
\begin{equation}
yx+y^2=x^2,\qquad xy+x^2=y^2, 
\end{equation}
and is thus empty since $x$ and $y$ are non-zero. In fact, the only 
boundary-unipotent representations in $\SL(2,\C)$ are reducible, so this is not surprising. 
The non-trivial obstruction 
class can be represented by the cocycle indicated in 
Figure~\ref{fig8triangulation}, and the Ptolemy variety is given by
\begin{equation}
yx-y^2=x^2,\qquad xy-x^2=y^2.
\end{equation}
As in Example~\ref{ex.52}, we may assume $y=1$. Hence, the Ptolemy variety 
detects two (complex conjugate) representations corresponding to the 
solutions to $x^2-x+1=0$. The extended Bloch group elements are
\begin{equation}
-(-\widetilde x,-2\widetilde x)+(\widetilde x,2\widetilde x)
\in\widehat\B(\C)_{\PSL},
\end{equation}
with complex volume
\begin{equation}
\pm 2.029883212\ldots + 0.0i.
\end{equation}
We thus recover the well known complex volume of the figure $8$ knot 
complement.
\end{example}
For $n=3$, similar calculations as those in Example~\ref{ex.52} show that 
the Ptolemy variety detects $3$ zero-dimensional components, but the only one with non-zero volume 
is the one induced by the geometric representation. For $n=4$, lots of new 
complex volumes emerge. For the trivial obstruction class, the non-zero 
complex volumes are
\begin{equation}
\pm 7.327724753\ldots + 0.0i=2\Vol_\C(5^2_1)+\pi^2i/4,
\end{equation}
where the manifold $5^2_1$ is the whitehead link complement.
For the non-trivial obstruction class, the complex volumes are
\begin{equation}
\begin{gathered}
\pm20.29883212\ldots+0.0i=10\Vol_\C(4_1)\in\C/\pi^2i\Z\\
\pm4.260549384\ldots\pm 0.136128165\ldots i\\
\pm3.230859569\ldots+0.0i\\
\pm8.355502146\ldots + 2.428571615\ldots i=\Vol_\C(-9^3_{15})+2\pi^2i/3\\
\pm3.276320849\ldots + 9.908433886\ldots i.
\end{gathered}
\end{equation}

\begin{example}[$S^1\times S^2$]
\lbl{ex.S1S2}
Figure~\ref{S1S2triangulation} shows a triangulation of $M=S^1\times S^2$ 
taken from the Regina census~\cite{Regina}. Since $\pi_1(S^1\times S^2)=\Z$, 
all representations in $\PSL(2,\C)$ lift to $\SL(2,\C)$, so we expect the 
Ptolemy variety for the non-trivial class in $H^2(M;\Z/2\Z)$ to be zero. 
This class is represented by the cocycle shown in 
Figure~\ref{S1S2triangulation}, and the Ptolemy variety is given by
\begin{equation}
-zx+x^2=y^2,\qquad x^2+zx=y^2,
\end{equation}
which indeed has no solutions in $\C^*$. For the trivial cohomology class, 
all signs are positive, and the two equations are equivalent. The extended 
Bloch group element is 
\begin{equation}
(\widetilde z+\widetilde x-2\widetilde y,2\widetilde x-2\widetilde y)
-(\widetilde z+\widetilde x-2\widetilde y,2\widetilde x-2\widetilde y)
=0\in\widehat\B(\C).
\end{equation}
In fact, the extended Bloch group element of a Ptolemy assignment is trivial 
for all $n$, as one easily verifies (the subsimplices cancel out in pairs).

We wish to find out which representations are detected by $P_2(K)$.
A choice of fundamental domain $F$ for $K$ in $L$ determines a presentation 
of $\pi_1(M)$ with a generator for each face pairing of $F$ and a relation 
for each $1$-cell of $K$ (to see this consider the standard presentation 
for the dual triangulation of $K$). Letting $F$ be the fundamental domain 
of $S^1\times S^2$ given by gluing the bottom faces of the two simplices 
together, one easily checks that the generator of $\pi_1(M)=\Z$ is given 
by the self gluing of the first simplex taking the face opposite the third 
vertex to the face opposite the zeroth. For $\alpha\in\SL(2,\C)$, the 
representation given by taking the generator to $\alpha$ has a decoration 
as in Figure~\ref{S1S2triangulation}. For $A=\Mat{a}{b}{c}{d}$, let 
$c(A)=c$, and note that $\det(e_1,Ae_1)=c(A)$. Letting $x$, $y$ and $z$ 
denote the Ptolemy coordinates, we have
\begin{equation} 
x=c(\alpha),\quad y=c(\alpha^2)=x\Tr(\alpha),\quad 
z=c(\alpha^3)=x(\Tr(\alpha)^2-1),
\end{equation}
and it follows that the Ptolemy variety detects all representations except 
those where $\Tr(\alpha)=\pm 1$.

\begin{figure}[ht]
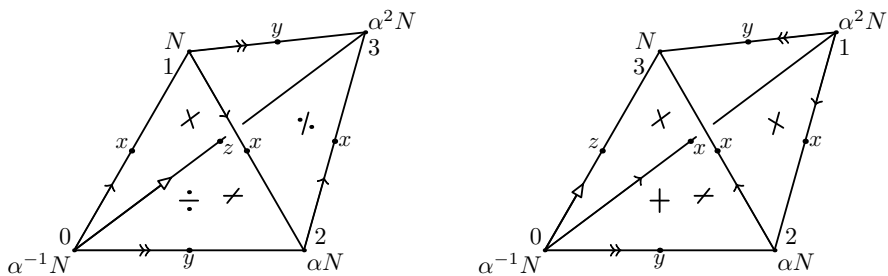

$$
\psdraw{SLrepFigures.17}{2.15in} \qquad
\psdraw{SLrepFigures.18}{2.15in} 
$$
\caption{A triangulation of $S^1\times S^2$. Both simplices have self 
gluings.}
\lbl{S1S2triangulation}
\end{figure}
\end{example}

\begin{remark}
When $n=2$, examples of Conjecture~\ref{newconjecture} are abundant. E.g.~for the $10_{155}$ knot complement ($10$ simplices), the volumes of the representations detected by the Ptolemy variety are (numerically)
\begin{equation}
\Vol(m032(6,1)),\quad 2\Vol(4_1),\quad 3\Vol(10_{155})-4\Vol(v3461),\quad \Vol(10_{155}).
\end{equation}
\end{remark}

\begin{remark} For the hyperbolic census manifolds, most of the components of the Ptolemy varieties (after fixing the action by diagonal matrices) tend to be zero-dimensional. By a result of Menal-Ferrer and Porti~\cite{FerrerPorti}, the composition of the geometric representation with $\phi_n$ is isolated among boundary-unipotent $p\SL(n,\C)$-representations. Higher dimensional components also occur (rarely for $n=2$, quite often for $n>2$), but as mentioned earlier, the complex volume is constant on components. 
\end{remark}

\begin{remark} 
If the face pairings do not respect the vertex orderings, one can still 
define a Ptolemy variety by introducing more signs. See~Garoufalidis--Goerner--Zickert~\cite{GaroufalidisGoernerZickert} for details.
\end{remark}

\section{The irreducible representations of $\SL(2,\C)$}
\lbl{sec.adjoint}

Let $\phi_n\colon\SL(2,\C)\to\SL(n,\C)$ denote the canonical irreducible 
representation. It is induced by the Lie algebra homomorphism 
$\mathfrak{sl}(2,\C)\to\mathfrak{sl}(n,\C)$ given by
\begin{equation}
\left[\begin{smallmatrix}0&1\\0&0\end{smallmatrix}\right]
\mapsto\diag^+(n-1,\dots,1),\quad 
\left[\begin{smallmatrix}0&0\\1&0\end{smallmatrix}\right]
\mapsto\diag^-(1,\dots,n-1),\quad \left[\begin{smallmatrix}1&0\\0&-1
\end{smallmatrix}\right]\mapsto\diag(n-1,n-3,\dots,-n+1),
\end{equation}
where $\diag^+(v)$ and $\diag^-(v)$ denote matrices whose first upper 
(resp.~lower) diagonal is $v$ and all other entries are zero. One has
\begin{align}
\phi_n\left(\left[\begin{smallmatrix}0&-a^{-1}\\a&0\end{smallmatrix}
\right]\right)&=q(a^{n-1},-a^{n-3},\dots,(-1)^{n-1}a^{-(n-1)})\lbl{phinlong}\\
\phi_n\left(\left[\begin{smallmatrix}1&x\\0&1\end{smallmatrix}\right]\right)&
=\pi_{n-1}(x,\dots ,x)\pi_{n-2}(x,\dots,x)\cdots\pi_1(x)\lbl{phinshort}.
\end{align}

\begin{proposition}
\lbl{adjointPtolemy} 
Let $c$ be a Ptolemy assignment on $\Delta_2^3$, and let $\tau$ denote the 
corresponding cocycle. The assignment 
\begin{equation}
\phi_n(c)\colon\dot\Delta^3_n(\Z)\to \C^*,\qquad t\mapsto\phi_n(c)_t
=\prod_{i<j}c_{ij}^{t_it_j}
\end{equation}
is a Ptolemy assignment on $\Delta_n^3$. If $c$ is a $\PSL(2,\C)$-Ptolemy 
assignment with obstruction cocycle $\sigma$, $\phi_n(c)$ is a 
$p\SL(n,\C)$-Ptolemy assignment with obstruction cocycle $\sigma$. Moreover, 
$\phi_n(c)$ is the Ptolemy assignment corresponding to $\phi_n(\tau)$.
\end{proposition}

\begin{proof}
Let $\alpha=(a_0,\dots,a_3)\in \Delta^3_{n-2}(\Z)$. Letting 
$k_\alpha=\prod_{i<j}c_{ij}^{a_ia_j}$, and $l_\alpha=\prod_{i<j}c_{ij}^{a_i+a_j}$, we have
\begin{equation}
\phi_n(c)_{\alpha_{03}}\phi_n(c)_{\alpha_{12}}=k_\alpha^2l_\alpha c_{03}c_{12},\quad 
\phi_n(c)_{\alpha_{01}}\phi_n(c)_{\alpha_{23}}=k_\alpha^2l_\alpha c_{01}c_{23},\quad 
\phi_n(c)_{\alpha_{02}}\phi_n(c)_{\alpha_{13}}=k_\alpha^2l_\alpha c_{02}c_{13}.
\end{equation}
Hence, the appropriate Ptolemy relations are satisfied, proving the first 
two statements. The long and short edges of the cocycle corresponding to 
$\phi_n(c)$ are given by \eqref{qij} and \eqref{alphaijk}, and we must 
prove that these agree with those of $\phi_n(\tau)$. For the long edges, 
this follows immediately from \eqref{phinlong}. For the short edges, an 
easy computation shows that all the diamond coordinates of a face are equal, 
and equal to the corresponding diamond coordinate of $c$. For example, the 
type $1$ diamond coordinate on face $3$ whose left vertex is 
$t=(t_0,t_1,t_2,0)$ is given by
\begin{equation}
\begin{aligned}
\frac{\phi_n(c)_{t+(0,-1,1,0)}\phi_n(c)_{t+(-1,1,0,0)}}{\phi_n(c)_{t}\phi_n(c)_{t+(-1,0,1,0)}}&=\frac{c_{01}^{t_0(t_1-1)}c_{02}^{t_0(t_2+1)}c_{12}^{(t_1-1)(t_2+1)}c_{01}^{(t_0-1)(t_1+1)}c_{02}^{(t_0-1)t_2}c_{12}^{(t_1+1)t_2}}{c_{01}^{t_0t_1}c_{02}^{t_0t_2}c_{12}^{t_1t_2}c_{01}^{(t_0-1)t_1}c_{02}^{(t_0-1)(t_2+1)}c_{12}^{t_1(t_2+1)}}\\&=\frac{c_{02}}{c_{01}c_{12}},
\end{aligned}
\end{equation}
which is a diamond coordinate for $c$. By~\eqref{phinshort} the short edges 
thus agree with those of $\phi_n(\tau)$, proving the result.  
\end{proof}

\begin{corollary}\label{adjointPtolemycor} If a representation $\rho\colon\pi_1(M)\to\PSL(2,\C)$ is detected by $P_2^\sigma(K)$ then $\phi_{2k+1}\circ\rho$ is detected by $P_{2k+1}(K)$ and $\phi_{2k}\circ\rho$ is detected by $P_{2k}^\sigma(K)$.\qed
\end{corollary}

\begin{theorem}
Let $\rho$ be a boundary-unipotent representation in $\SL(2,\C)$ or 
$\PSL(2,\C)$. The extended Bloch group element of $\phi_n\circ \rho$ 
is $\binom{n+1}{3}$ times that of $\rho$. In fact, the shapes of all subsimplices are equal.
\end{theorem}

\begin{proof}
By refining the triangulation if necessary, we may represent $\rho$ by a 
Ptolemy assignment $c$ on $K$. Then $\phi=\phi_n(c)$ is a Ptolemy assignment 
representing $\phi_n\circ \rho$, and the extended Bloch group element of 
$\phi_n\circ\rho$ is given by
\begin{equation}
[\phi_n(\rho)]=\sum_i\epsilon_i\sum_{\alpha\in \Delta^3_{n-2}(\Z)}
(\widetilde{\phi}^i_{\alpha_{03}}+\widetilde{\phi}^i_{\alpha_{12}}
-\widetilde{\phi}^i_{\alpha_{02}}-\widetilde{\phi}^i_{\alpha_{13}},
\widetilde{\phi}^i_{\alpha_{01}}+\widetilde{\phi}^i_{\alpha_{23}}
-\widetilde{\phi}^i_{\alpha_{02}}-\widetilde{\phi}^i_{\alpha_{13}}).
\end{equation} 
By Proposition~\ref{independenceoflog}, we may choose the logarithms 
independently as long as we use the same logarithm for identified points. 
Defining $\widetilde{\phi}^i_t=\sum_{j<k}t_jt_k\widetilde c^i_{jk}$,
we see that
\begin{equation}
(\widetilde{\phi}^i_{\alpha_{03}}+\widetilde{\phi}^i_{\alpha_{12}}
-\widetilde{\phi}^i_{\alpha_{02}}-\widetilde{\phi}^i_{\alpha_{13}},
\widetilde{\phi}^i_{\alpha_{01}}+\widetilde{\phi}^i_{\alpha_{23}}
-\widetilde{\phi}^i_{\alpha_{02}}-\widetilde{\phi}^i_{\alpha_{13}})
=(\widetilde c_{03}+\widetilde c_{12}-\widetilde c_{02}-\widetilde c_{13},
\widetilde c_{01}+\widetilde c_{23}-\widetilde c_{02}-\widetilde c_{13}),
\end{equation}
which means that the flattenings assigned to each subsimplex of $\Delta^i_n$ 
are equal. By Lemma~\ref{numberofsubsimplices}, 
$\Norm{\Delta^3_{n-2}(\Z)}=\binom{n+1}{3}$, and the result follows.
\end{proof}

\subsection{Essential edges}
\lbl{sub.essential}

\begin{definition}
\lbl{def.essential}
An edge of $K$ is \emph{essential} if the lifts to $L$ have distinct end 
points.
\end{definition}

Note that an edge may be essential even though it is homotopically trivial 
in $K$. Let $L^{(0)}$ denote the zero skeleton of $L$.

\begin{lemma}
\lbl{developingdecoration}
Let $\rho$ be a representation in $\SL(2,\C)$ or $\PSL(2,\C)$. A decoration 
of $\rho$ determines a $\rho$-equivariant map
\begin{equation}
\lbl{developingCinfty}
D\colon L^{(0)}\to \partial\overline\H^3=\C\cup\{\infty\},
\quad e_i\mapsto g_i\infty.
\end{equation} 
Every such map comes from a decoration, and the decoration is generic 
if and only if the vertices of each simplex of $L$ map to distinct points 
in $\C\cup\{\infty\}$.
\end{lemma}

\begin{proof}
Equivariance of \eqref{developingCinfty} follows from the definition of a 
decoration. A $\rho$-equivariant map $D\colon L^{(0)}\to\Cinfty$ is uniquely 
determined by its image of lifts $\widetilde e_i\in L$ of the zero cells 
$e_i$ of $K$. Picking $g_i$ such that $g_i\infty=D(\widetilde e_i)$, we 
define a decoration by assigning the coset $g_iN$ to $\widetilde e_i$. 
The last statement follows from the fact that $\det(g_1e_1,g_2e_1)=0$ if 
and only if $g_1\infty=g_2\infty$.
\end{proof}

In the following we assume that the interior of $M$ is a cusped hyperbolic $3$-manifold $\H^3/\Gamma$ with finite volume.
\begin{proposition}\label{detectsgeometric}
If all edges of $K$ are essential, 
the geometric representation has a generic decoration.
\end{proposition}

\begin{proof}
We identify $\pi_1(M)$ with $\Gamma\subset\PSL(2,\C)$. Each vertex of $L$ 
corresponds to either a cusp of $M$ or an interior point of $M$. Accordingly, 
we have $L^{(0)}=L^{(0)}_{\cusp}\cup L^{(0)}_{\textnormal{int}}$. 
Each point in $L^{(0)}_{\cusp}$ determines a parabolic subgroup of $\PSL(2,\C)$ 
stabilizing a unique point in $\C\cup\{\infty\}$. We thus have an equivariant 
map $D\colon L^{(0)}_{\cusp}\to \C\cup\{\infty\}$ taking a point to its 
stabilizer. Let $e_1$ and $e_2$ be points in $L^{(0)}_{\cusp}$ connected by an 
edge. Since all edges of $K$ are essential, $e_1\neq e_2$. It is well known 
that the point stabilizers of different cusps are distinct. Hence, 
$D(e_1)\neq D(e_2)$ if $e_1$ and $e_2$ correspond to different cusps. If 
$e_1$ and $e_2$ correspond to the same cusp, there exists an element in 
$\Gamma$ taking $e_1$ to $e_2$. Since only peripheral elements (i.e.~cusp 
stabilizers) have fixed points in $\Cinfty$, it follows that 
$D(e_1)\neq D(e_2)$. We extend $D$ to $L^{(0)}$ by choosing any equivariant 
map $L^{(0)}_{\textnormal{int}}\to\C\cup\{\infty\}$. Since such map is uniquely 
determined by finitely many values (which may be chosen freely), we can 
pick the extension so that the vertices of each simplex map to distinct 
points. This proves the result. 
\end{proof}

\begin{theorem}\label{Ptolemydetectsrhogeoproof} Suppose all edges of $K$ are essential. The representation $\phi_n\circ\rho_{\geo}$ is detected by $P_n(K)$ if $n$ is odd, and by $P_n^{\sigma_{\geo}}(K)$ if $n$ is even.
\end{theorem} 
\begin{proof}
By Proposition~\ref{detectsgeometric}, $P_2^{\sigma_{\geo}}(K)$ detects $\rho_{\geo}$. The result now follows from Corollary~\ref{adjointPtolemycor}.
\end{proof}

\begin{remark}
The census triangulations all have essential edges.
\end{remark}

%\subsection{Reconstructing a representation from the Ptolemy coordinates}
%Two methods that are in some sense dual. One via long and short edges.

%Maximal tree.
%Let $P$ be a fundamental polyhedron of $K$ in $L$ and let $\Delta_1,\dots,\Delta_k$ be the simplices of $K$. Let $c$ be a Ptolemy assignment on $K$. By Proposition~\ref{reconstructtriple}, the restriction $c^i$ of $c$ to $\Delta_i$ is the Ptolemy assignment of a tuple $\alpha_i=(N,q_1^iN,\alpha_2^iq_2^iN,\alpha_3^iq_3^iN)$. We wish to define a decoration of $M$ whith Ptolemy assignment $c$. 
%Decorate $\Delta_1$ by $\alpha_1$, and suppose by induction that we have a decoration of a triangulated subcomplex $Q$ of $P$ whose Ptolemy assignment is $c$.
%Let $\Delta_i$ be a simplex in $P$, not contained in $Q$, sharing a face $F$ with $Q$. By Proposition~\ref{reconstructtriple} there is a unique $g\in \SL(n,\C)$ such that the restriction of the decoration $g\alpha_i$ of $\Delta_i$ to $F$ agrees with the decoration of $Q$. We have thus obtained a decoration of $Q\cup\Delta_i$, whose Ptolemy assignment agrees with $c$, completing the induction.  

\section{Gluing equations and Ptolemy assignments}
\lbl{sec.gluingequations}

In this section we discuss the relation between Ptolemy assignments and 
solutions to the gluing equations. The latter were invented by 
Thurston~\cite{ThurstonNotes} to explicitly compute the hyperbolic 
structure (and its deformations) of a triangulated hyperbolic manifold,
and used effectively in \cite{NeumannZagier, snapprogram, SnapPy}.  
The gluing equations make sense for any $3$-cycle. 
They are defined by assigning a \emph{cross-ratio} $z_i\in\Cremove$ to each simplex 
$\Delta_i$ of $K$. Given these, we assign cross-ratio parameters to 
the edges of $\Delta_i$ as in Figure~\ref{Crossratioparameters}.

\begin{figure}[ht]
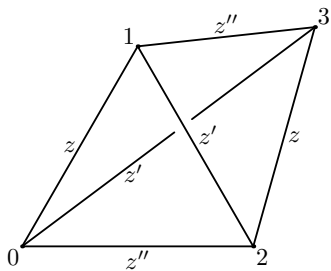

$$
\psdraw{SLrepFigures.22}{1.7in}
$$
\caption{Assigning cross-ratio parameters to the edges of $\Delta_i$. By definition, $z^\prime=\frac{1}{1-z}$ and $z^{\prime\prime}=1-\frac{1}{z}$.}
\lbl{Crossratioparameters}
\end{figure}
There is a gluing equation for each edge $E$ in $K$ and each generator 
$\gamma$ of the fundamental group of each boundary component of $M$. These 
are given by
\begin{equation}\lbl{gluingequations}
\prod_{e\mapsto E}z(e)^{\epsilon_i(e)}=1,\qquad \prod_{\gamma\textnormal{ passes }e}
z(e)^{\epsilon_i(e)}=1.
\end{equation}
Here $z(e)$ denotes the cross-ratio parameter assigned to $e$, and 
$\epsilon_i(e)=\epsilon_i$ if $e$ is an edge of $\Delta_i$. 
%We refer to Neumann-Zagier~\cite{NeumannZagier} for more details.
It follows that the set of assignments $\Delta_i\mapsto z_i\in\Cremove$ 
satisfying the gluing equations~\eqref{gluingequations} is an algebraic 
set $V(K)$.
 
\begin{lemma} 
For every point $\{z_i\}\in V(K)$ there is a map $D\colon L^{(0)}\to\Cinfty$ 
such that if $\widetilde\Delta_i$ is a lift of $\Delta_i$ with vertices 
$e_1,\dots,e_3$ in $L$, the cross-ratio of the ideal simplex with vertices 
$D(e_1),\dots,D(e_3)$ is $z_i$. It is unique up to multiplication by an 
element in $\PSL(2,\C)$. Moreover, there is a unique (up to conjugation) 
boundary-unipotent representation $\pi_1(M)\to\PSL(2,\C)$ such that $D$ is 
$\rho$-equivariant.
\end{lemma}

\begin{proof}
Pick a fundamental domain $F$ for $K$ in $L$. Pick a simplex $\Delta$ in 
$F$ and define $D$ by mapping the first $3$ vertices of $\Delta$ to $0$, 
$\infty$ and $1$. The map $D$ is now uniquely determined by the cross-ratios. 
The fundamental group of $M$ has a presentation with a generator for each 
face pairing of $F$. The second statement thus follows from the fact that 
$\PSL(2,\C)$ is $3$-transitive. We leave the details to the reader.
\end{proof}

Given a Ptolemy assignment on $K$, we assign the cross-ratio $z_i=\frac{c^i_{03}c^i_{12}}{c^i_{02}c^i_{13}}$ to $\Delta_i$.
Note that the Ptolemy relations imply that the cross-ratio parameters are given by
\begin{equation}\label{eq:shapeparameterformula}
z_i=\frac{c^i_{03}c^i_{12}}{c^i_{02}c^i_{13}},\qquad z_i^\prime=\frac{c^i_{02}c^i_{13}}{c^i_{01}c^i_{23}}, \qquad z_i^{\prime\prime}=-\frac{c^i_{01}c^i_{23}}{c^i_{03}c^i_{12}}. 
\end{equation}

\begin{theorem} 
There is a surjective regular map
\begin{equation}
\coprod_{\sigma\in H^2(K;\Z/2\Z)} P_2^\sigma(K)\to V(K),
\quad c\mapsto \{z_i=\frac{c^i_{03}c^i_{12}}{c^i_{02}c^i_{13}}\}.
\end{equation}
The fibers are disjoint copies of $(\C^*)^h$, where $h$ is the number of zero-cells of $K$.
\end{theorem}

\begin{proof}
By a simple cancellation argument (as in the proof of Zickert~\cite[Theorem~6.5]{ZickertDuke}), the gluing equations would be satisfied if the formula \eqref{eq:shapeparameterformula} for $z_i^{\prime\prime}$ did not have the minus sign. The minus sign appears whenever the edge is $02$ or $13$. As explained in the proof of Proposition~\ref{independenceoflog}, any curve passes these an even number of times. It thus follows that the cross-ratios satisfy the gluing equations. Surjectivity follows from Lemma~\ref{developingdecoration}, and the fact that fibers are $(\C^*)^h$ follows from the fact that $g_1\infty=g_2\infty$ if and only if $g_1N=g_2dN$ for a unique diagonal matrix $d$.
\end{proof}

\begin{remark}
Gluing equation varieties for $n>2$ are studied in Garoufalidis-Goerner-Zickert~\cite{GaroufalidisGoernerZickert}.
\end{remark}
\section{Other fields}
\lbl{otherfields}
The Ptolemy varieties $P_n(K)$ and $P_n^\sigma(K)$ may be defined over an arbitrary field $F$, and as in Section~\ref{proofsection}, a Ptolemy assignment determines a boundary-unipotent representation in $\SL(n,F)$, respectively, $p\SL(n,F)$. If $E$ is a primitive extension of $F^*$ by $\Z$, there are maps
\begin{equation}
V_n(K)_F\to\widehat\B_E(F),\qquad V_n^\sigma(K)_F\to\widehat\B_E(F)_{\PSL}
\end{equation}
defined as in~\eqref{defoflambda} using a set theoretic section of $E\to F^*$ instead of a logarithm. If $F$ is infinite, the chain complex of Ptolemy assignments computes relative homology (see Proposition~\ref{freeaction}) and we have maps
\begin{equation}
H_3(\SL(n,F))\to\widehat\B_E(F),\qquad H_3(p\SL(n,F))\to\widehat\B_E(F)_{\PSL}.
\end{equation}
It thus follows that every boundary-unipotent representation has an extended Bloch group element $[\rho]$. If $F$ is a number field, the extended Bloch groups are independent of the extension $E$.  

\begin{theorem}
Let $F$ be a number field, and let $\rho\colon\pi_1(M)\to\SL(n,F)$ be a boundary-unipotent representation. If $\rho$ is irreducible, $[\rho]$ lies in $\widehat\B(\Tr(\rho))$. 
\end{theorem}
\begin{proof}
Let $\sigma$ be an automorphism of $F$ over $\Tr(\rho)$ and let $\tau\colon F\to\C$ be an embedding. Then $\rho$ and $\sigma\circ\rho$ have the same traces, so $\tau\!\circ\!\rho$ and $\tau\!\circ\!\sigma\!\circ\!\rho$ are conjugate in $\SL(n,\C)$, and thus have the same extended Bloch group element in $\widehat\B(\C)$. By Corollary~\ref{injectivity}, it follows that $[\rho]=[\sigma\!\circ\!\rho]\in\widehat\B(F)$. Hence, $[\rho]$ is invariant under all automorphisms of $F$ over $\Tr(\rho)$, so $[\rho]\in\widehat\B(\Tr(\rho))$ by Galois descent.
\end{proof}

%\section{Appendix}

%Data and implementation notes. To be written by Matthias Goerner.

\bibliographystyle{plain}
\bibliography{BibFile}

\end{document}